\documentclass[reqno,11pt]{amsart}
\usepackage{amssymb}
\usepackage{amsmath}
\usepackage[usenames]{color}
\usepackage{mathrsfs}

\numberwithin{equation}{section}

\allowdisplaybreaks

\begin{document}

\renewcommand{\AA}{\mathcal{A}}
\newcommand{\BB}{\mathcal{B}}
\newcommand{\CC}{\mathcal{C}}
\newcommand{\DD}{\mathcal{D}}
\newcommand{\EE}{\mathcal{E}}
\newcommand{\FF}{\mathcal{F}}
\newcommand{\GG}{\mathcal{G}}
\newcommand{\HH}{\mathcal{H}}
\newcommand{\II}{\mathcal{I}}
\newcommand{\JJ}{\mathcal{J}}
\newcommand{\KK}{\mathcal{K}}
\newcommand{\LL}{\mathcal{L}}
\newcommand{\MM}{\mathcal{M}}
\newcommand{\NN}{\mathcal{N}}
\newcommand{\OO}{\mathcal{O}}
\newcommand{\PP}{\mathcal{P}}
\newcommand{\QQ}{\mathcal{Q}}
\newcommand{\RR}{\mathcal{R}}
\renewcommand{\SS}{\mathcal{S}}
\newcommand{\TT}{\mathcal{T}}
\newcommand{\UU}{\mathcal{U}}
\newcommand{\VV}{\mathcal{V}}
\newcommand{\WW}{\mathcal{W}}
\newcommand{\XX}{\mathcal{X}}
\newcommand{\YY}{\mathcal{Y}}
\newcommand{\ZZ}{\mathcal{Z}}


\newcommand{\A}{\mathbb{A}}
\newcommand{\B}{\mathbb{B}}
\newcommand{\C}{\mathbb{C}}
\newcommand{\D}{\mathbb{D}}
\newcommand{\E}{\mathbb{E}}
\newcommand{\F}{\mathbb{F}}
\newcommand{\G}{\mathbb{G}}
\renewcommand{\H}{\mathbb{H}}
\newcommand{\I}{\mathbb{I}}
\newcommand{\J}{\mathbb{J}}
\newcommand{\K}{\mathbb{K}}
\renewcommand{\L}{\mathbb{L}}
\newcommand{\M}{\mathbb{M}}
\newcommand{\N}{\mathbb{N}}
\renewcommand{\O}{\mathbb{O}}
\renewcommand{\P}{\mathbb{P}}
\newcommand{\Q}{\mathbb{Q}}
\newcommand{\R}{\mathbb{R}}
\renewcommand{\S}{\mathbb{S}}
\newcommand{\T}{\mathbb{T}}
\newcommand{\U}{\mathbb{U}}
\newcommand{\V}{\mathbb{V}}
\newcommand{\W}{\mathbb{W}}
\newcommand{\X}{\mathbb{X}}
\newcommand{\Y}{\mathbb{Y}}
\newcommand{\Z}{\mathbb{Z}}


\newcommand{\al}{\alpha}
\newcommand{\be}{\beta}
\newcommand{\ga}{\gamma}
\newcommand{\de}{\delta}
\newcommand{\ep}{\varepsilon}
\newcommand{\ze}{\zeta}
\newcommand{\et}{\eta}
\newcommand{\vth}{\vartheta}
\renewcommand{\th}{\theta}
\newcommand{\io}{\iota}
\newcommand{\ka}{\kappa}
\newcommand{\la}{\lambda}
\newcommand{\rh}{\rho}
\newcommand{\si}{\sigma}
\newcommand{\ta}{\tau}
\newcommand{\up}{\upsilon}
\newcommand{\ph}{\varphi}
\newcommand{\ch}{\chi}
\newcommand{\ps}{\psi}
\newcommand{\om}{\omega}

\newcommand{\Ga}{\Gamma}
\newcommand{\De}{\Delta}
\newcommand{\Th}{\Theta}
\newcommand{\La}{\Lambda}
\newcommand{\Si}{\Sigma}
\newcommand{\Up}{\Upsilon}
\newcommand{\Ph}{\xi}
\newcommand{\Om}{\Omega}

\newcommand{\inj}{\hookrightarrow}
\newcommand{\stetein}{\overset{s}{\hookrightarrow}}
\newcommand{\dichein}{\overset{d}{\hookrightarrow}}
\newcommand{\pa}{\partial}
\newcommand{\re}{\restriction}
\newcommand{\tief}{\downharpoonright}

\newcommand{\bra}{\langle}
\newcommand{\ket}{\rangle}
\newcommand{\bs}{\backslash}
\newcommand{\divv}{\operatorname{div}}
\newcommand{\Dt}{\frac{\mathrm d}{\mathrm dt}}

\newcommand{\sm}{\setminus}
\newcommand{\es}{\emptyset}

\newtheorem{theorem}{Theorem}[section]
\newtheorem{corollary}{Corollary}
\newtheorem*{main}{Main Theorem}
\newtheorem{lemma}[theorem]{Lemma}
\newtheorem{proposition}{Proposition}
\newtheorem{conjecture}{Conjecture}
\newtheorem*{problem}{Problem}
\theoremstyle{definition}
\newtheorem{definition}[theorem]{Definition}
\newtheorem{remark}{Remark}
\newtheorem*{notation}{Notation}

\numberwithin{equation}{section}

\sloppy
\title[Stability of electro-kinetic flows]
      {Global well-posedness and stability of electro-kinetic flows}
\author[D. Bothe, A. Fischer and J. Saal]{Dieter Bothe, Andr\'e Fischer and J{\tiny\"u}rgen Saal}
\address{Center of Smart Interfaces, Technische Universit\"at Darmstadt, Petersenstra\ss e 32,
64287 Darmstadt, Germany}
\email{bothe@csi.tu-darmstadt.de}
\address{Center of Smart Interfaces, Technische Universit\"at Darmstadt, Petersenstra\ss e 32,
64287 Darmstadt, Germany}
\email{fischer@csi.tu-darmstadt.de}
\address{Center of Smart Interfaces, Technische Universit\"at Darmstadt, Petersenstra\ss e 32,
64287 Darmstadt, Germany}
\email{saal@csi.tu-darmstadt.de}

\date{\today}
\thispagestyle{empty}
\parskip0.5ex plus 0.5ex minus 0.5ex
\bibliographystyle{plain}

\begin{abstract}
We consider a coupled system of Navier-Stokes and Nernst-Planck
equations,
describing the evolution of the velocity and the concentration fields of dissolved constituents in an electrolyte solution.
Motivated by recent applications in the field of micro- and nanofluidics, we consider the model in such generality that electrokinetic flows
are included. This prohibits employing the assumption of electroneutrality of the total solution,
which is a common  approach in the mathematical literature in order to determine the electrical potential.
Therefore we complement the system of mass and momentum balances with a Poisson equation for the
electrostatic potential, with the charge density stemming from the concentrations of the ionic species.
For the resulting Navier-Stokes-Nernst-Planck-Poisson system we prove the existence of unique local strong solutions
in bounded domains in $\R^n$ for any $n\geq2$ as well as the existence of unique global strong solutions and
exponential convergence to uniquely determined steady states in two dimensions.
\end{abstract}
\maketitle

\tableofcontents

\section{Introduction}

Electrokinetic effects have received a lot of attention from various branches of research and applications,
recently in particular in the field of micro- and nanofluidics. Two well-known phenomena are \emph{electrophoresis}
and \emph{electro-osmotic flow (EOF)}.
While in electrophoresis, a dispersed (colloidal) particle is dragged through a fluid under an electrical field,
EOF describes the motion of an aqueous solution past a solid wall as a response to an external electrical field.
The latter offers a powerful method to manipulate liquids in mini- and micro-devices, since it does not depend on mechanical tools.
Indeed, employing electrical fields to pump a liquid through a (sub-)micro-channel can economize this task considerably,
if there are charged solutes in dispersion; see, e.g., \cite{MasBha,Probstein}.
EOF can be applied, for instance, for cooling systems in micro-electronics.
In analytical chemistry, electrical fields serve for the separation of chemicals
according to their electrophoretic characteristics (see \cite{kirby}), 
but it also has a great potential for
promoting mixing on micro-scales \cite{chang}. For further applications we refer to \cite{troj} and the references therein.

The present paper studies the mathematical model governing electro-osmotic flow.
From the significance of the EOF, a proper mathematical understanding is indispensable for applications.
We consider a rather general situation in which a dilute viscous solution with dissolved charged species
is placed in a container with solid walls, the container being situated in an electrical field.
A typical concrete case is an aqueous electrolyte solution, where of course uncharged species may also be dissolved in the solute.
The externally applied electrical field induces the EOF due to the presence of charged constituents in solution.
On the other hand, the ions themselves generate an intrinsic electrical field as well, which in turn also affects their fluxes.
Thus, even without external electrical fields, electrical effects have
to be accounted for whenever charged species are involved.
Concerning the solid-liquid interface, it is well-known from
electro-chemistry, see e.g.\ \cite{Newman}, that electro-chemical double
layers appear on the surface, thus the surface is typically charged.
In a simplified view the surface can be interpreted locally as a plate capacitor.
In contrast, far away from walls or other interfaces, the net charge density is essentially zero,
i.e.\ the solution is \emph{electro-neutral} inside the bulk; see, e.g., \cite{HHV, Newman}.
But note that in nanofluidic applications as well as in transport processes in nanopores within porous media,
there is no macroscopic bulk phase and, hence, electroneutrality does not even hold in a subdomain.
Even the more general approach to describe the electric potential via the Poisson-Boltzmann equation may not be accurate enough,
since the Boltzmann statistics does not apply in overlapping double layers; cf.\ \cite{Li2011}.

In the mathematical literature, the electroneutrality condition has been frequently used as a simplifying condition, 
see e.g.\ \cite{amann_renardy,BothePruss, wiedmann}.
As a consequence of electroneutrality,
there also is no electrical body-force acting on the mixture, such that the equations for the fluid motion decouple.
There still is a non-trivial electrical field, implicitly determined by
the algebraic condition of electro-neutrality which leads to a strong
coupling of the mass fluxes,
but still this simplification is too restrictive here, 
since it rules out the electrokinetic effects.

In papers of Y.S.~Choi and R.~Lui, see \cite{CL92,CL_multidim}, as well as in papers of
H.~Gajewski, A.~Glitzky, K.~Gr\"oger and R.~H\"unlich, e.g.~\cite{gaj85,GGH95,Glitzky,Glitzky_2},
the electro-neutrality condition is not employed. Instead, in the context of reaction-diffusion systems,
a more general mathematical model comprising species mass balances with fluxes according to
the Nernst-Planck equation combined with a Poisson equation for the electrical potential
is investigated. However, no fluid motion is taken into account and, hence, no momentum balance is accounted for.
Employing suitable Lyapunov functionals, the existence of global
solutions in two dimensions is proven and a careful analysis of the
long-time behavior is provided.
In this approach the mass flux $j_i$ of constituent $i$ is composed of a Fickian diffusion term and
an electro-migration term due to the electrical field.

In \cite{Schmuck}, the afore mentioned system of Nernst-Planck equation and the Poisson equation is
complemented by the Navier-Stokes system. Well-posedness issues are accounted for under the
assumption that all diffusivities of the individual species are constant and equal.
From the physical point of view, these assumptions are rather strong.
In fact, the diffusivity of one species will in general depend on the full composition of the system; cf.\ \cite{cussler}.
While this dependence is negligible in the dilute case considered here, the values of the diffusivities of 
individual constituents will still differ significantly, up to orders of magnitude.
Moreover, in \cite{Schmuck} the electrostatic potential is assumed to satisfy the homogeneous Neumann condition,
which rules out boundary charges in contrast to the possible occurrence of double layers at the boundary.

Further in some way related mathematical literature is given by
\cite{bedthom2011,jerome2002}. In \cite{bedthom2011}, for instance, motion of a
rigid charged macro-molecule immersed in an incompressible ionized fluid
is considered and local existence of a weak solution is proved.
The work \cite{jerome2002} deals with a two carriers (anions and cations)
electrophoretic model in the whole space $\R^n$. There local-in-time
existence of a unique strong solution in a Hilbert space setting is proved.
A major point in \cite{jerome2002} is uniformity of the existence
interval with respect to the viscosity, allowing for the inviscid limit to the 
Euler-Nernst-Planck-Poisson system.

Apparently, no further analytical works are available concerning the system
consisting of the Navier-Stokes equations for the fluid motion, the 
Nernst-Planck equation for the species concentrations
and the Poisson equation for the electrostatic potential. 
Therefore, our approach generalizes the ones given, for instance, in
\cite{CL92,CL_multidim,gaj85,GGH95,Glitzky,Glitzky_2}, 
where electro(-reaction)-diffusion systems are considered,
not taking into account convection induced by fluid motion or a 
coupling to the Navier-Stokes equations. 
In this context, we also want to stress the fact that our approach does not need
the electro-neutrality assumption, as is assumed e.g.\ in 
\cite{amann_renardy,BothePruss, wiedmann}, 
and also allows for different diffusivities  
for the individual species (in contrast to \cite{Schmuck})
as well as for non-trivial charge density on the interface.
Our result on local-in-time well-posedness for the full
Navier-Stokes-Nernst-Planck-Poisson system in a strong $L^p$-setting
hence is new.
The methods used for the proof of global strong solutions in
two space dimensions are close to the ones utilized in
\cite{CL_multidim}. The results obtained in the underlying note, however, 
differ from the ones derived in
\cite{CL_multidim} in several respects. Apart from the fact that
in \cite{CL_multidim} there is no convection, 
an approach in H\"older spaces is presented whereas we prefer to work
in an $L^p$-setting.
Our proof of existence and uniqueness of steady states is much simpler
and we are
able to remove a restrictive global electro-neutrality condition imposed
in \cite{CL_multidim} for uniqueness of steady states. 
Moreover, we prove exponential stability of steady states
which is not given in \cite{CL_multidim}. 
The methods applied here are related to
\cite{Glitzky}, where
exponential stability of steady states for concentrations is derived 
in the absence of fluid motion. In this regard, we note that the proof of
exponential
convergence to a steady state simultaneously for fluid velocity
and concentrations of species is new.

This article is organized as follows. First, we give a justification
of the mathematical model considered in this work, 
the Navier-Stokes-Nernst-Planck-Poisson
system, by illustrating its derivation from basic 
principles. We then
state our main results on local well-posedness in arbitrary dimension
and on global well-posedness and exponential stability in
two dimensions. 
In Section~\ref{local} we give a proof
of the local-in-time well-posedness result based on maximal regularity.
Utilizing an appropriate Lyapunov functional and conservation of mass for the
concentrations, which gives the starting point  for
a bootstrap argument, 
we prove in Section~\ref{global} that in two
dimensions the local solution extends to a global one.
The content of Section~\ref{stat} is the proof of existence of 
a unique steady state solution to the
Navier-Stokes-Nernst-Planck-Poisson system 
for each given distribution of initial masses. 
Finally, in Section~\ref{asymp} we prove exponential
stability of this stationary solution.


\subsection{The mathematical model}
\label{subsec_model}

The basic continuum mechanical model for transport processes in (non-reactive) electrolyte solutions
consists of the Navier-Stokes equations for the mixture as a whole, coupled to a set of
species equations for which the mass fluxes are modeled via the Nernst-Planck equation with the electrical
potential given by a Poisson equation; see \cite{Newman}.
Then, depending on the concrete situation, the Poisson equations is often
replaced by either the above mentioned condition of electro-neutrality or by the Poisson-Boltzmann equation in which
the concentrations are related to the electrical potential by a Boltzmann statistics.
Here, in order to maintain the applicability to EOF, we keep the more general Poisson equation.

But there is one hidden difficulty which forces us to briefly reconsider the derivation of the mathematical model. 
Indeed, the Nernst-Planck equation in the usual form, stating that the diffusive (molar) mass fluxes $j_i$ are
given as
\begin{equation}
\label{eq_Nernst-Planck}
j_i=-D_i \nabla c_i - m_i \frac{F}{RT}z_i c_i\nabla\phi,
\end{equation}
already relies on additional assumptions. In fact, the standard derivation from the more fundamental generalized Maxwell-Stefan equations
(\cite{TaylorKrishna}) already considers the mixture to be electrically neutral. But the
resulting form is also applied to the general case which obviously 
requires clarification.
Fortunately, it turns out that the constitutive equation \eqref{eq_Nernst-Planck} is also valid without
electro-neutrality if the electrolyte solution is dilute - the case we are considering here.

We consider a fluid composed of $N+1$ constituents with molar concentrations $c_i$, i.e.\
$c_i=\rho_i/M_i$ with $\rho_i$ the mass density and $M_i$ the molar mass, where $i=0$ refers to the solvent
in which the other species ($i=1,\ldots ,N$) are dissolved.
A sufficiently fundamental starting point is the system of partial mass balances together with a common mixture momentum balance,
where the diffusive fluxes of species mass are determined from the system of Maxwell-Stefan equations.
The latter follow as a reduced set of balances of partial momenta, simplified by neglecting diffusive waves due to 
differential acceleration of the constituents.
The considered balances read
\begin{align}
\label{eq_part_bal}
\pa_tc_i+\divv(c_iu+j_i)&=0,\qquad i=0,\ldots,N,\\
\label{eq_mom_bal}
\pa_t(\rho u)+\divv(\rho u\otimes u-S)&=-\nabla \pi+\rho b,
\end{align}
where
\[
\rho:=\sum_{i=0}^N\rho_i,\qquad\rho u:=\sum_{i=0}^N\rho_iu_i,\qquad\rho
b:=\sum_{i=0}^N\rho_ib_i
\]
define the total density $\rho$, the barycentric velocity $u$ and the
total body force $b$, given
the individual velocities $u_i$ and forces $b_i$. The molar mass flux $j_i$ is defined as
$j_i:=c_i(u_i-u)$, but also has to be modeled, since individual velocities are not accounted for. In
(\ref{eq_mom_bal}), $S$ denotes the (viscous) extra stress and $\pi$ the pressure. Note that equations
(\ref{eq_part_bal}) are not independent, since total mass is conserved. Indeed,
\begin{equation*}
\sum_{i=0}^NM_ij_i=0,
\end{equation*}
and one of the species equations in (\ref{eq_part_bal}) can be replaced by the mixture continuity equation
\begin{equation}
\label{eq_cont_eq}
\pa_t\rho+\divv(\rho u)=0.
\end{equation}
The closure of (\ref{eq_part_bal}) and (\ref{eq_mom_bal}), i.e.\ the modeling of the $j_i$ is
achieved by the Maxwell-Stefan equations. Formulated for molar mass concentrations and under
isothermal conditions they read (cf.\ \cite{TaylorKrishna})
\begin{equation}
\label{eq_MS}
-\sum_{k=0}^N\frac{x_kj_i-x_ij_k}{\DJ_{ik}}=\frac{c_i}{RT}\nabla\mu_i-\frac{y_i}{RT}\nabla\pi-\frac{\rho_i}{RT}(b_i-b),\quad i=0,\ldots,N.
\end{equation}
Here $x_i=c_i/c$ with $c:=\Sigma_{k=0}^N c_k$ are the molar fractions, $\DJ_{ij}$ the Maxwell-Stefan
diffusivities, $R$ is the universal gas constant, $T$ the absolute temperature and
$y_i:=\rho_i/\rho$ the mass fractions. For extension to the non-isothermal case, see e.g.\ Chapter 24 in
\cite{BSL}. The molar-based chemical potentials $\mu_i$ are defined as
\begin{equation*}
\mu_i=M_i\frac{\pa(\rho\psi)}{\pa\rho_i}=\frac{\pa(\rho\psi)}{\pa c_i},
\end{equation*}
where $\rho\psi=\rho\psi(T,\rho_0,\ldots,\rho_N)$ denotes the free energy density of the mixture.
Hence $\mu_i$ are functions of $T$ and all $\rho_i$ as well, but in the literature on chemical
thermodynamics and multicomponent transport, it is common to use $(T,\pi,x_0,\ldots,x_N)$ as
independent variables instead. Of course not all $x_i$ are independent, hence
$(T,\pi,x_1,\ldots,x_{N})$, say, is more appropriate. If $\mu_i=\tilde\mu_i(T,\pi,x_j)$ denotes this
function, then (\ref{eq_MS}) can be written as
\begin{equation}
\label{eq_MS_1}
-\sum_{k=0}^N\frac{x_kj_i-x_ij_k}{\DJ_{ik}}=\frac{c_i}{RT}\nabla_{\pi,T}\tilde\mu_i+\frac{\phi_i-y_i}{RT}\nabla\pi-\frac{\rho_i}{RT}(b_i-b),
\end{equation}
where $\phi_i$ is the volume fraction of species $i$ and
\begin{equation*}
\nabla_{\pi,T}\tilde\mu_i:=\sum_j\frac{\pa\tilde\mu_i}{\pa x_j}\nabla x_j
\end{equation*}
is the ''gradient taken at constant $\pi$ and $T$''. The three terms on the right-hand side of
(\ref{eq_MS_1}) correspond to diffusional transport due to composition gradients, pressure gradients
and external forces, respectively.

In case of an electrolyte solution, the individual body force densities are
\begin{equation}
\label{eq_body_force}
b_i=-\frac{F}{M_i}z_i\nabla\phi,
\end{equation}
where $F$ is the Faraday constant (charge of one mole of electrons), $z_i$ are the charge numbers
and $\phi$ is the electrical potential. The latter is assumed to be determined by the electrostatic
relation
\begin{equation}
\label{eq_poisson_eq}
-\divv(\ep\nabla\phi)=F\sum_{k=0}^Nz_kc_k,
\end{equation}
i.e.\ dynamic effects and the influence of magnetic fields are ignored which is usually sound
for the applications which have been mentioned above. In (\ref{eq_poisson_eq}), the material parameter
$\ep$ denotes the electrical permittivity of the medium.

The system of equations (\ref{eq_part_bal}), (\ref{eq_mom_bal}), (\ref{eq_cont_eq}),
(\ref{eq_MS_1}), (\ref{eq_body_force}) and (\ref{eq_poisson_eq}) has to be complemented by an
equation of state, relating $\pi$ with $(T,\rho_0,\ldots,\rho_N)$, and by a material function modeling
the free energy of the mixture; alternatively, and more common, the chemical potentials need to be
modeled. These are highly complicated tasks, in general.

From here on, we only consider the case of a dilute solution, for which species $0$, say, is the
solvent with $x_0\approx1$ and all other species are present in trace quantities, i.e.\ $x_i\ll1$ for
$i=1,\ldots,N$. This leads to several simplifications. First, for $i=1,\ldots,N$, the Maxwell-Stefan equations reduce
to
\begin{equation}
\label{eq_MS_2}
-\frac1{\DJ_{i0}}j_i=\frac{c_i}{RT}\nabla_{\pi,T}\tilde\mu_i+\frac{\phi_i-y_i}{RT}\nabla
\pi-\frac{\rho_i}{RT}(b_i-b).
\end{equation}
Next, the difficult request for an equation of state is resolved by assuming the mixture to maintain
constant density, which is reasonable since $\rho\approx \rho_0$ and the latter will be almost
constant at isothermal conditions. Therefore the species equation for $i=0$ is omitted, but
\begin{equation}
\label{eq_rho_div}
\rho\equiv const, \qquad\divv u=0
\end{equation}
is used instead. The effect of pressure diffusion is usually negligible, since $\phi_i/y_i$ is not
far from $1$ and here both $\phi_i$ and $y_i$ are even close to zero. It remains to model the
chemical potentials. For a dilute solution one has, \cite{atkins},
\begin{equation}
\label{eq_mod_chem_pot}
\tilde\mu_i=\tilde\mu_i^0(T,\pi)+RT\log x_i,\qquad i=1,\ldots N.
\end{equation}
With these simplifications and with $D_i:=\DJ_{i0}$, the diffusive fluxes $j_i$ from (\ref{eq_MS_2}) become
\begin{equation}
\label{eq_flux}
j_i=-D_i\left(c\nabla x_i+\frac{F}{RT}(z_ic_i-\sum_{k=0}^Nz_kc_k)\nabla\phi\right),\qquad
i=1,\ldots,N.
\end{equation}
This is close to the Nernst-Planck form, up to the term $\sum_{k=0}^Nz_kc_k$. The latter vanishes in
the electro-neutral case, but this would restrict the applicability of the model too far. Instead, we
note that in the dilute case necessarily $z_0=0$, hence $b_0=0$, and then
\begin{equation}
b_i-b=b_i-\sum_{k=0}^Ny_kb_k=(1-y_i)b_i+\!\!\!
\sum_{i\neq k=1}^N y_k b_k\approx  b_i,\quad i=1,\ldots,N.
\end{equation}
Consequently, $b$ is small as compared to $b_i$ (for $i$ such that $z_i\ne0$), but this
does not imply $b$ to be small as compared to the other quantities in the momentum balance. In fact
the electrical forces acting on dilute components can drag the whole mixture as exploited in EOF.
To sum up, using the fact that $c\approx c_0$ is nearly constant, we see that
in a dilute electrolyte solution the diffusive molar mass fluxes of the solutes can be modeled via
the Nernst-Planck constitutive equation
\begin{equation}
\label{eq_flux_fin}
j_i=-D_i\left(\nabla c_i+\frac{F}{RT}z_ic_i\nabla\phi\right),\qquad i=1,\ldots,N.
\end{equation}
Note that this derivation from the Maxwell-Stefan equations
automatically also yields the Nernst-Einstein relation,
saying that the mobility coefficients in the electromigration term, i.e.\ the $m_i$ in \eqref{eq_Nernst-Planck}
equal $D_i /RT$.

Concerning (\ref{eq_mom_bal}), the viscous stress tensor $S$ is modeled by
\[S=\mu(\nabla u+(\nabla u)^{\sf T})\]
with constant viscosity $\mu>0$, hence $\divv S=\mu\Delta u$ since $u$ is divergence free by (\ref{eq_rho_div}).
The permittivity $\ep$ in equation (\ref{eq_poisson_eq})
for the electrostatical potential is assumed to be positive and constant. Thus the full model reads
\begin{align}
\label{eq_NS_}
\rho(\pa_tu+(u\cdot\nabla)u)-\mu\Delta u+\nabla
\pi&=-F\sum_{j=1}^{N}z_jc_j\nabla\phi,\\
\label{eq_NS_div_}
\divv u&=0,\\
\label{eq_NP_}
\partial_tc_i+\divv \left(c_i u -D_i\left(\nabla
c_i+\frac{F}{RT}z_ic_i\nabla\phi\right)\right)&=0,\quad i=1,\ldots ,N,\\
\label{eq_P_}
-\ep\Delta\phi&=F\sum_{j=1}^{N}z_jc_j,
\end{align}
for $(t,x)\in(0,T)\times\Omega$, where $\Omega\subset\R^n$, $n\geq 2$, is the domain under
consideration. We note that the force term on the right hand side 
of (\ref{eq_NS_}) is usually called Coulomb force. 

At $\partial \Omega$, we impose the boundary conditions
\begin{align}
\label{eq_NS_bdy_}
u&=0,\\
\label{eq_NP_bdy_}
-D_i(\pa_\nu c_i-\frac{F}{R T} z_ic_i\pa_\nu\phi)&=0,\qquad i=1,\ldots ,N,\\
\label{eq_P_bdy_}
\pa_\nu\phi+\tau\phi&=\xi,
\end{align}
for $(t,x)\in(0,T)\times\pa\Omega$, where $\nu$ denotes the outer normal to $\Omega$ and $\pa_\nu$ denotes the
derivative in direction $\nu$.

Condition (\ref{eq_NS_bdy_}) is the commonly assumed no-slip condition for the velocity field
$u$ and the no-flux condition (\ref{eq_NP_bdy_}) models impermeable walls and assures the conservation of mass in $\Omega$.
Relation (\ref{eq_P_bdy_}) requires some physical explanation. By the Maxwell equations we have
\begin{equation}
\label{eq_surface_charge}
\ep\pa_\nu\phi=\sigma_s \quad \mbox{ on } \partial \Omega,
\end{equation}
where $\sigma_s$ denotes the surface charge density. A typical choice for boundary
conditions for the electrical potential (see e.g. \cite{CL_multidim,Schmuck}) are given by
homogeneous Neumann conditions, implicitly claiming that the surface is charge free. As already
mentioned, this condition is not sufficient for our situation. However, finding appropriate models for
the charge density on the surface is a delicate matter. 
The boundary condition
(\ref{eq_P_bdy_})
is the same as employed in \cite{Glitzky}. For a physical motivation we 
imagine to append a layer of (dimensionless) thickness $0<\de\ll1$ 
around the volume $\Omega$ which gives us locally a plate capacitor 
for which the surface charge density is known to be
\[\sigma_s=\frac{\ep}\de(\phi_\de-\phi),\]
where $\phi_\de$ is the electrical potential at the outer boundary of the layer. Plugging this
equality into (\ref{eq_surface_charge}) we obtain
\[\pa_\nu\phi+\frac1\de\phi=\frac1\de\phi_\de.\]
The parameter $\tau:=1/\de>0$ in (\ref{eq_P_bdy_}) hence refers to the local capacity of the boundary.
The right-hand side $\xi$ in (\ref{eq_P_bdy_}) depends also on an external electrical potential.

The system of equations (\ref{eq_NS_})-(\ref{eq_P_}),
(\ref{eq_NS_bdy_})-(\ref{eq_P_bdy_}) complemented by the initial
conditions
\begin{align}
u(0,x)&=u^0(x),\quad x\in\Omega,
\label{eq_NS_ic_}\\
c(0,x)&=c^0(x),\quad x\in\Omega,
\label{eq_NP_ic_}
\end{align}
represents the full model, where here and in the sequel 
$c=(c_1,\ldots,c_N)$ always denotes the full vector of
concentrations. 

\subsection{Main results}
\label{subsec_mainresult}

Before introducing our main results let us fix some notation which will be frequently used below.

Let $n \in\N$ and $\Omega\subset\R^n$ be an open domain with boundary $\pa\Omega$ and $1\leq
p\leq\infty$. By $L^p(\Omega),L^p(\pa\Omega),W^{m,p}(\Omega)$ we denote the usual 
Lebesgue and Sobolev spaces. We do not distinguish between spaces of functions
and spaces of vector fields, i.e.\ we write also $L^p(\Omega)$ for $L^p(\Omega)^n$ for example. The notion $W^{m,p}_0(\Omega)$ describes the 
closure of smooth and compactly supported functions $C^\infty_0(\Omega)$
in $W^{m,p}(\Omega)$. For
$s>0$, $p\in(1,\infty)$ we write $W^s_p(\Omega)=(L_p(\Omega),W^{m,p}(\Omega))_{\theta,p}$ for the
Sobolev-Slobodeckii space, where $(\cdot,\cdot)_{\theta,p}$ is the real interpolation
functor with exponents $\theta,p$ and $s=\theta m$. It is worth mentioning that $W^m_2(\Omega)=W^{m,2}(\Omega)$ for $m\in\N$, see \cite{triebel}. Let
$C_{0,\sigma}^\infty(\Omega)$ denote the solenoidal functions with compact support in $\Omega$ and
\[L^p_\sigma(\Omega)=\overline{C_{0,\sigma}^\infty(\Omega)}^{\|\cdot\|_p}\]
denotes the space of solenoidal $L^p$-functions. Occasionally we write $W^{m,p}_{0,\sigma}(\Omega)$
instead of $W^{m,p}_0(\Omega)\cap L^p_\sigma(\Omega)$. The corresponding norms are denoted by
$\|\cdot\|_p,\|\cdot\|_{p,\pa\Omega},\|\cdot\|_{W^{m,p}}$, etc. For $J=(0,T)$, $T\in(0,\infty]$
and $X$ a Banach space we write $L^p(J;X), W^{m,p}(J;X),W^s_p(J;X)$ for the corresponding spaces of $X$-valued
functions. For $s\ne1/p$ let
\[\hspace{.01ex}_0W^s_p(J,X)=\overline{C^\infty_0((0,T],X)}^{\|\cdot\|_{W^s_p}}.\]
Note that $\hspace{.01ex}_0W^s_p(J;X)=W^s_p(J;X)$ for $s<1/p$ and that 
for $1/p<s<1+1/p$ this space coincides with all functions in $W^s_p(J,X)$
having vanishing time trace at zero.

Now let $\Omega\subset\R^n$, $n\geq2$ be a bounded smooth domain. Concerning the model derived in
Subsection~\ref{subsec_model} we set the parameters $\rho,\mu,F,RT$ all to 1 for technical simplicity;
all the results obtained in this work remain valid in the general case where $\rho,\mu,F,RT>0$, but constant.

\begin{remark}
\label{rem_sohr}
From the functional analytic approach to the Navier-Stokes equations it is sufficient to solve
for the velocity field rather than both for the velocity and the pressure field, because the pressure can be
recovered once the velocity is known, see e.g. \cite{Sohr}. In this spirit, by
formal application of
the Helmholtz projection $P$ (see \cite{Sohr}) to (\ref{eq_NS_}), solving system
(\ref{eq_NS_}), (\ref{eq_NS_div_}), (\ref{eq_NS_bdy_}), (\ref{eq_NS_ic_}) is equivalent to solving the problem
\begin{align}
\label{eq_NS_evo_1}
\pa_tu+A_Su&=-P(u\cdot\nabla)u-P\left(\Sigma_jz_jc_j\nabla\phi\right),\quad t>0,\\
\label{eq_NS_evo_2}
u|_{t=0}&=u^0,
\end{align}
where $A_S=-P\Delta$ is the Stokes operator subject to homogeneous Dirichlet boundary conditions
with domain
\[\DD(A_S)=W^{2,p}(\Omega)\cap W^{1,p}_{0}(\Omega)\cap L^p_\sigma(\Omega),\quad p\in(1,\infty).\]
Once (\ref{eq_NS_evo_1}), (\ref{eq_NS_evo_2}) is solved, the pressure can be
recovered by
\begin{equation}\label{pressrec}
	\nabla \pi=(I-P)\left(\Delta u-(u\cdot\nabla)u
	         -\left(\Sigma_jz_jc_j\nabla\phi\right)\right).
\end{equation}
\end{remark}

\begin{remark}
\label{rem_triebel}
Consider problem (\ref{eq_P_}), (\ref{eq_P_bdy_}).

(a) For the Robin-Laplacian $-\Delta_R$ in $L^p(\Omega)$ with $\DD(-\Delta_R)=\{v\in
W^{2,p}(\Omega):\pa_\nu v+\tau v=0\}$, $\tau>0$ a constant, it is well-known that
$0\in\rho(-\Delta_R)$, therefore the operator $(-\Delta_R,\BB)v:=(-\Delta v,\pa_\nu v+\tau v)$, is a
bijection from $W^{m+2,p}(\Omega)$ to
$W^{m,p}(\Omega)\times W^{m+1-1/p}_p(\pa\Omega)$, $m\in\N$ and $1<p<\infty$, see e.g.
\cite[Chapter 5]{triebel}.

(b) We only consider time-independent functions
$\xi=\xi(x)$. This will prove important in various places throughout this work.

(c) Note that once $c$ is known, the potential $\phi$ is known as well.
\end{remark}

From Remarks~\ref{rem_sohr} and \ref{rem_triebel} it suffices
to consider solutions $(u,c)$ rather than solutions $(u,c,\pi,\phi)$ 
of the following problem:
\begin{align}
\label{eq_NS}
\pa_tu+A_Su+P(u\cdot\nabla)u+P\left(\Sigma_jz_jc_j\nabla\phi\right)&=0,\quad t>0,\\
\label{eq_NP}
\partial_tc_i+\divv \left(c_iu-D_i\nabla c_i-D_iz_ic_i\nabla\phi\right)&=0,\quad t>0, \
x\in\Omega\\
\label{eq_P}
-\ep\Delta\phi&=\Sigma_jz_jc_j,\quad t>0, \ x\in\Omega\\
\label{eq_NP_bdy}
\pa_\nu c_i+z_ic_i\pa_\nu\phi&=0,\quad t>0, \ x\in\pa\Omega,\\
\label{eq_P_bdy}
\pa_\nu\phi+\tau\phi&=\xi,\quad t>0, \ x\in\pa\Omega\\
\label{eq_NS_ic}
u(0)&=u^0,\quad x\in\Omega\\
\label{eq_NP_ic}
c_i(0)&=c_i^0,\quad x\in\Omega,
\end{align}
for $i=1,\ldots ,N$. System (\ref{eq_NS})-(\ref{eq_NP_ic}) will be called problem ($P$).

We will frequently use the abbreviation
\begin{equation}
\label{eq_def_J}
J_i:=j_i+c_iu=-D_i\nabla c_i-D_iz_ic_i\nabla\phi+c_iu.
\end{equation}
With this notation and the no-slip condition for $u$, equations (\ref{eq_NP}) and (\ref{eq_NP_bdy}) can be written as
\begin{alignat*}{2}
\pa_tc_i+\divv J_i&=0,\quad\text{ in }(0,T)\times\Omega,&i=1,\ldots,N,\\
J_i\cdot\nu&=0,\quad\text{ on }(0,T)\times\pa\Omega,\quad&i=1,\ldots,N.
\end{alignat*}

For the local well-posedness in any dimension $n\geq2$ we apply maximal regularity. Let us define
the usual spaces of maximal $L^p$-regularity for $1<p<\infty$ by
\begin{align*}
\E_{T,p}^u&:=W^{1,p}(0,T;L^p_\sigma(\Omega))\cap L^p(0,T;\DD(A_S)),\\
\E_{T,p}^c&:=W^{1,p}(0,T;L^p(\Omega))\cap L^p(0,T;W^{2,p}(\Omega)),\\
\E_{T,p}&:=\E_{T,p}^u\times\E_{T,p}^c,\\
\X_p&:=\left(W^{2-2/p}_p(\Omega)\cap W^{1,p}_{0,\sigma}(\Omega)\right)\times W^{2-2/p}_p(\Omega).
\end{align*}
We also set
\begin{align*}
	\E_{T,p}^\pi&:= L^p(0,T;\widehat W^{1,p}(\Omega))\\
	\E_{T,p}^\phi&:=W^{1,p}(0,T;W^{2,p}(\Omega))\cap L^p(0,T;W^{4,p}(\Omega))
\end{align*}
for the spaces of the corresponding pressure and potential. Our main result
on local-in-time strong solvability now reads as
\begin{theorem}
\label{theorem_strong_loc}
Let $\Omega\subset\R^n$, $n\geq2$, be a bounded and smooth 
domain and let $(n+2)/3<p<\infty$, $p\ne3$. Let
$\xi\in W^{3-1/p}_p(\pa\Omega)$ and $(u^0,c^0)\in \X_p$. In the case $p>3$ let $c^0$ and $\xi$
satisfy the compatibility condition
\begin{equation}
\label{eq_comp}
\pa_\nu c_i^0+z_ic_i^0\pa_\nu\phi^0|_{\pa\Omega}=0,
\end{equation}
where
\[-\ep\Delta\phi^0=\Sigma_jz_jc_j^0,\quad\pa_\nu\phi^0+\tau\phi^0=\xi.\]
Then there is $T>0$ such that there exists a unique solution
$(u,c)\in\E_{T,p}$ to problem ($P$). The latter is equivalent
to saying that
\[
	(u,c,\pi,\phi)\in \E_{T,p}\times\E_{T,p}^\pi\times\E_{T,p}^\phi 
\]
with $\pi$ given by (\ref{pressrec}) 
is the unique solution of system (\ref{eq_NS_})-(\ref{eq_P_bdy_})
and (\ref{eq_NS_ic_}), (\ref{eq_NP_ic_}).
If $c^0\geq0$ (componentwise) we have $c\geq0$ as long as the solution exists.
\end{theorem}
\begin{remark}
(a) The local-in-time solution $(u,c)$ remains unique, as long as the solution exists, see Remark~\ref{rem_unique}.

(b) Remark~\ref{rem_triebel} and the fact that $\pa_t$ and $\Delta$
commute cause the potential $\phi$
to be contained in the following regularity class:
\[\phi\in W^{1,p}(0,T;W^{2,p}(\Omega))\cap L^p(0,T;W^{4,p}(\Omega)),\]
if $\xi$ is smooth enough. This is accomplished by $\xi\in W^{3-1/p}_p(\pa\Omega)$ which motivates the 
assumption on $\xi$ in Theorem~\ref{theorem_strong_loc}. Note, however,
that this is done for simplicity;
this constraint can be weakened, while still obtaining the same results presented in this paper, see Remark~\ref{rem_phi_1}.

(c) We also remark that at this stage minimal boundary regularity
is none of our major objectives. By the higher regularity for $\phi$ we
work with, the approach presented in this note should at least 
work for domains $\Omega$ of class $C^4$. This is also more than enough
to guarantee maximal regularity for the Stokes equations, see
e.g.\ \cite{noll2003}. A closer inspection of what is really needed
for the potential $\phi$ in the proofs of our main Theorems 
(which is not done here) and again \cite{noll2003} will yield
that even $C^3$ shall be enough.

(d) The compatibility condition (\ref{eq_comp}) is required for $p>3$ only, because the trace in time
for $t=0$ of the boundary spaces is only well-defined in this case.

(e) The main difficulty in proving Theorem~\ref{theorem_strong_loc} lies in the fact that (\ref{eq_NP_bdy}) represents a nonlinear boundary condition which has to be treated in the according trace space
\[\F_{T,p}^{c,\pa\Omega}=W^{1/2-1/2p}_p(0,T;L^p(\pa\Omega))\cap L^p(0,T;W^{1-1/p}_p(\pa\Omega)).\]

(f) From the physical point of view it is important to note that the total mass of each species is
conserved due to the no-flux condition at the boundary; by nonnegativity this implies
$\|c_i(t)\|_1\equiv\|c_i^0\|_1$ as long as the solution exists.
\end{remark}
Now, let $(u,c)$ denote the local strong solution for the case 
$p=2$ and let the total free energy function $E$ be defined by
\[E(u,c):=\frac12\|u\|_2^2+\sum_{i=1}^N\int_\Omega c_i\log
c_idx+\frac\ep2\|\nabla\phi\|_2^2+\frac{\ep\tau}2\|\phi\|_{2,\pa\Omega}^2,\]
for $(u,c)\in L^2_\sigma(\Omega)\times L^2(\Omega)$ with $c_i\geq0$. Here $\frac12\|u\|_2^2$ represents the kinetic
energy, $\int c_i\log
c_i$ refers to the chemical potential for species $i$, see e.g. \cite{Newman}, and the last two
terms correspond to electrical energy. It turns out that the function $V$, given by
\[V(t):=E(u(t),c(t)),\quad t\geq0,\]
is a Lyapunov functional for system ($P$) and its derivative $\Dt V$, the dissipation rate, is
given explicitly by
\begin{equation}\label{dissiprate}
\Dt V(t)=-\|\nabla u(t)\|_2^2-\sum_{i=1}^N\int_\Omega\frac{1}{D_ic_i(t)}|j_i(t)|^2dx\leq0,
\end{equation}
where $j_i=-D_i\nabla c_i-D_iz_ic_i\nabla\phi$ is the flux of species
$i$. For a suitable definition of the possibly degenerate integral 
we refer to Lemma~\ref{lem_lya_der}.

In the two-dimensional case this is the cornerstone for 
the following result on global well-posedness.

\begin{theorem}
\label{theorem_glob_ex}
Let $\Omega\subset\R^2$ be bounded and smooth and let $(u^0,c^0)\in W^{1,2}_{0,\sigma}(\Omega)\times(W^{1,2}(\Omega)\cap
L^\infty(\Omega))$ with $c_i^0\geq0$, $i=1,\ldots,N$. Then the local strong solution
$(u,c)\in\E_{T,2}$ to ($P$) (equivalently, 
$(u,c,\pi,\phi)\in \E_{T,2}\times\E_{T,2}^\pi\times\E_{T,2}^\phi$
to (\ref{eq_NS_})-(\ref{eq_P_bdy_}), (\ref{eq_NS_ic_}),
(\ref{eq_NP_ic_})) from Theorem~\ref{theorem_strong_loc}
extends to a unique global strong solution.
\end{theorem}
\begin{remark}
(a) Comparing the requirements on the initial data of Theorem~\ref{theorem_strong_loc} for $p=2$ and of
Theorem~\ref{theorem_glob_ex}, taking into account $W^1_2(\Omega)=W^{1,2}(\Omega)$ we note that there is an additional $L^\infty$-condition on $c^0$.
This is inserted for technical reasons. However, it is possible to ease
this constraint, as explained in Remark~\ref{rem_weak_sol_2}.

(b) We emphasize that for the global existence of solutions due to
Theorem~\ref{theorem_glob_ex} no smallness or compatibility conditions
on the data are required. 
We also would like to stress at this stage that concerning 3D not only
the Navier-Stokes but also the Nernst-Planck equation seems to be 
difficult to handle. So, even under the assumption of small data, 
to the authors at the present stage it is unclear how to extend
Theorem~\ref{theorem_glob_ex} to three space dimensions, 
see also Remark~\ref{rem_3d}.

(c) From the fact that $V$ is a Lyapunov functional we readily confirm that the quantities $\|u\|_2$ and
$\|c_i\log c_i\|_1$ are uniformly
bounded in time. Having this
information at hand it is possible to prove that the $L^2(\Omega)$-norm of the concentrations $c_i$
is
uniformly bounded in time, as it is done in \cite{CL_multidim}. Applying further energy estimates,
the norm of the strong solutions $(u,c)$ is shown to remain
finite for finite time, whence the solution can be extended up to any time $T$.

\end{remark}
In view of the long-time behavior of system ($P$) we are interested in 
steady state solutions in two dimensions, i.e., in existence and 
uniqueness of solutions
$(u^\infty,c^\infty)$ to the stationary system
\begin{equation}
\label{eq_stat}
\tag{$P_s$}
\left.
\begin{array}{rlll}
A_Su+P((u\cdot\nabla) u)+P(\Sigma_jz_jc_j\nabla\phi)&=0,&\\
\divv (c_iu-D_i\nabla c_i-D_iz_ic_i\nabla\phi)&=0,& x\in\Omega,&i=1,\ldots ,N,\\
-\ep\Delta\phi&=\Sigma_jz_jc_j,&x\in\Omega,\\
\pa_\nu c_i+z_ic_i\pa_\nu\phi&=0,&x\in\pa\Omega,& i=1,\ldots ,N,\\
\pa_\nu\phi+\tau\phi&=\xi,&x\in\pa\Omega,
\end{array}\right\}
\end{equation}
subject to $\int_\Omega c_i^\infty(x)dx=m_i$ for given total masses $m_i>0$.

\begin{theorem}
\label{theorem_stat}
Let $\Omega\subset\R^2$ be bounded and smooth and let numbers $m_i>0$, $i=1,\ldots ,N,$ be given.

(a) Then there is a unique solution
\[(u^\infty,c^\infty)\in (W^{1,2}_{0,\sigma}(\Omega)\cap W^{2,2}(\Omega))\times W^{2,2}(\Omega)\]
to problem ($P_s$) subject to $c_i^\infty\geq0$, $\int_\Omega c_i^\infty dx=m_i$, $i=1,\ldots ,N$.
More precisely, we have $u^\infty=0$ and 
$c^\infty\in W^{2,2}(\Omega)$ is the unique nonnegative solution to
\begin{equation}
\label{eq_stat_u0}
\tag{$P'_s$}
\left.
\begin{array}{rll}
\divv(-D_i\nabla c_i-D_iz_ic_i\nabla\phi)&=0,& i=1,\ldots ,N,\\
(\pa_\nu c_i+z_ic_i\pa_\nu\phi)|_{\pa\Omega}&=0,& i=1,\ldots ,N,\\
-\ep\Delta\phi&=\Sigma_jz_jc_j,&\\
(\pa_\nu\phi+\tau\phi)|_{\pa\Omega}&=\xi,&
\end{array}\right\}
\end{equation}
subject to $\int_\Omega c_i^\infty dx=m_i$, $i=1,\ldots,N$.

The pressure $\pi^\infty$ in the equilibrium state is given by $\pi^\infty=\Sigma_jc_j^\infty$ up to a constant.

(b) Let $(u^0,c^0)\in W^{1,2}_{0,\sigma}(\Omega)\times (W^{1,2}(\Omega)\cap L^\infty(\Omega))$ with
$c_i^0\geq0$ and
$\int_\Omega c_i^0(x)dx=m_i$, $i=1,\ldots,N$. Let $c^\infty\in
W^{2,2}(\Omega)$ be the unique
nonnegative solution
to ($P'_s$) subject to $\int_\Omega c_i^\infty dx=m_i$, $i=1,\ldots,N$, and let $(u,c)$ be the unique global strong solution to ($P$) with initial data
$(u^0,c^0)$ from Theorem~\ref{theorem_glob_ex}, then
\begin{equation}
\label{eq_conv_steady}
\lim_{t\to\infty}(u(t),c(t))=(0,c^\infty)
\end{equation}
in $W^{1,2}_{0,\sigma}(\Omega)\times W^{1,2}(\Omega)$.
\end{theorem}

It is worth to mention that there are no equilibrium 
states with non-trivial velocity, i.e.\ with Theorem
\ref{theorem_stat} the long-time dynamics of the mixture is determined by the
Nernst-Planck equation and the Poisson equation, problem (\ref{eq_stat_u0}), and the initial masses
$m_i=\int c_i^0(x)dx>0$.

The proof of Theorem~\ref{theorem_stat} can be outlined as follows.
For existence we employ the global strong solution from Theorem \ref{theorem_glob_ex}. Indeed,
let $0\leq c_i^0\in W^{1,2}(\Omega)\cap L^\infty(\Omega)$ with $\int_\Omega c_i^0(x)dx=m_i$,
$i=1,\ldots,N$, and $u^0\in W^{1,2}_{0,\sigma}(\Omega)$ be arbitrary and let $(u,c)$ be the
corresponding global strong solution from Theorem~\ref{theorem_glob_ex}.
The estimates derived in the proof of Theorem \ref{theorem_glob_ex} below imply
compactness of the semi-orbits $(u,c)$ in $L^2_\sigma(\Omega)\times L^2(\Omega)$, where any
accumulation point turns out to be a steady
state. From mass conservation we know that each constituent's mass is conserved, hence a
steady state for masses $m_i$ is found.

With the help of an alternative representation of the dissipation rate
(\ref{dissiprate}) it is possible to
obtain a characterization of steady states, which allows us to conclude the uniqueness of steady
states subject to prescribed masses $m_i>0$.

With this reasoning, assertion $(b)$ of Theorem \ref{theorem_stat}
essentially follows as a by-product of $(a)$: by compactness
of semi-orbits and the uniqueness of the accumulation point, the steady state, the system converges
to an equilibrium as time tends to infinity. This is first obtained 
with respect to the $L^2$-topology, but with
additional effort convergence with respect to $W^{1,2}$ can be proven.

Concerning the rate of convergence, we have the following result.

\begin{theorem}
\label{theorem_conv}
Let $\Omega\subset\R^2$ be bounded and smooth, let $(u^0,c^0)\in W^{1,2}_{0,\sigma}(\Omega)\times
(W^{1,2}(\Omega)\cap L^\infty(\Omega))$ with $c_i^0\geq0$ and
$\int_\Omega c_i^0dx=m_i$, $i=1,\ldots,N$, let $c^\infty\in
W^{2,2}(\Omega)$ be the unique nonnegative solution to (\ref{eq_stat_u0}) with $\int_\Omega c_i^\infty=m_i$,
$i=1,\ldots,N$, and let $(u,c)$ be the global strong solution to ($P$) from
Theorem~\ref{theorem_glob_ex} with initial data $(u^0,c^0)$.
Then there are constants $C,\omega>0$ such that
\[\|u(t)\|_2+\|c(t)-c^\infty\|_1\leq Ce^{-\omega t}.\]
\end{theorem}

In order to derive exponential convergence results, 
the so-called entropy method has been
widely used, c.f.\ \cite{DF06}, \cite{Glitzky} and references therein. The idea is the following:
Consider a Lyapunov functional $\Psi(\psi(t))$ with
\[\Dt \Psi(\psi(t))=-D(\psi(t)),\]
where $\psi$ represents the state of the (physical) system, the dissipation rate $D$ is nonnegative and
$\Psi(\psi)=0$ if and only if $\psi$ is an equilibrium. Suppose that $\Psi(\psi(t))\leq CD(\psi(t))$
holds for some constant $C>0$. Then Gronwall's inequality implies
$\Psi(\psi(t))\leq \Psi(0)\exp(-t/C)$, i.e.\ the value of the Lyapunov functional decreases with
exponential speed. In mathematical works it seems quite common to refer
to $\Psi$ as the entropy. Note,
however, that this notion can be misleading, since the entropy is a physical
quantity in thermodynamics which is a priori not connected with a
generic Lyapunov functional.

In (\ref{eq_def_PSI}) we define
\[\Psi(t):=E(u(t),c(t))-E(0,c^\infty).\]
From the physical point of view $\Psi(t)$ can be interpreted as the difference of the total free energy at
time $t$ and the total free energy in the equilibrium state. The main task for 
proving Theorem~\ref{theorem_conv} lies in the proof of the exponential
decay of $\Psi$. To this end, note
that $\Dt\Psi=\Dt V$, so there is an explicit representation for the derivative. It is therefore
sufficient to show $\Psi(t)\leq C\Dt\Psi(t)$ for some constant $C>0$. Performing a similar
approach as in \cite{Glitzky} this can be shown by an indirect (somewhat
technical) reasoning, see Theorem~\ref{thm_asymp}. 
The assertion of Theorem~\ref{theorem_conv} then follows by
a straight forward calculation.

\section{Local in time well-posedness}
\label{local}
Let $\Omega\subset\R^n$ be bounded with smooth boundary, $n\geq2$. By means of maximal
$L^p$-regularity we will show local well-posedness simultaneously for system
(\ref{eq_NS})-(\ref{eq_NP_ic}) and
for system (\ref{eq_NS})-(\ref{eq_NP_ic}) with (\ref{eq_NP}) and (\ref{eq_NP_bdy}) replaced by
\begin{align}
\label{eq_NP+}
\tag{\ref{eq_NP}'}
\partial_tc_i+\divv \left(-D_i\nabla c_i-D_iz_ic_i^+\nabla\phi+c_iu\right)&=0,\quad t>0, \
x\in\Omega\\
\label{eq_NP_bdy+}
\tag{\ref{eq_NP_bdy}'}
\pa_\nu c_i+z_ic_i^+\pa_\nu\phi&=0,\quad t>0, \ x\in\pa\Omega,
\end{align}
where $f^+(x)=\max\{0,f(x)\}$ for $f:\R^n\to\R$.

This auxiliary problem will be denoted by ($P'$). Its well-posedness will give us the possibility to
derive nonnegativity of concentrations $c_i$ in an easy way without applying the strong maximum
principle. In the case of classical solutions $(u,c)$ the maximum principle implies that
concentrations $c_i$ are nonnegative, see e.g. \cite{CL_multidim}. The reason why we hesitate to
follow this line lies in the fact that, in general, strong solutions $(u,c)$, see
Definiton~\ref{def_strong_sol}, are not classical.
Nevertheless it is possible to show that $(u,c)$ is classical via maximal regularity and
bootstrapping. Since due to the strong coupling of the differential equations of ($P$) this is rather technical, we
believe that our approach to nonnegativity, inspired by \cite{gaj85}, is appropriate.

\begin{definition}[Strong solution]
\label{def_strong_sol}
Let $\Omega\subset\R^n$, $n\geq2$, be bounded and smooth, $T>0$ and $1<p<\infty$. The function
\[(u,c)\in W^{1,p}(0,T;L^p_\sigma(\Omega)\times L^p(\Omega))\cap L^p(0,T;\DD(A_S)\times
W^{2,p}(\Omega))\]
is called \textit{strong solution} to ($P$) resp.\ ($P'$) if $(u,c)$ satisfies ($P$) resp.\ ($P'$) almost
everywhere.
\end{definition}

\begin{remark}
\label{rem_loc_nonneg}
In order to show local well-posedness for ($P$) it is sufficient to show that ($P'$) is well-posed.
The reason lies in the fact that the (nonlinear) terms in ($P'$) have less regularity as compared to
those in ($P$). To be precise, for $c\in W^{1,p}(0,T; L^p(\Omega))\cap L^p(0,T;W^{2,p}(\Omega))$
we only know that $c^+\in W^{1,p}(0,T; L^p(\Omega))\cap
L^p(0,T;W^{1,p}(\Omega))$, cf.\ \cite{GilTrud}. Hence, once we have proved local well-posedness for ($P'$), the same line of arguments yields local well-posedness for the original problem ($P$).
\end{remark}

Following this approach we first have to examine the associated linear problem and prove maximal regularity for
the corresponding solution operator. In the next step we write the full problem as a fixed-point
equation and apply the contraction mapping theorem to deduce a unique fixed point, the
local-in-time strong solution.

In order to sort the terms in ($P'$) into linear and nonlinear terms, we split the potential $\phi$
into $\phi_1$ and $\phi_2$ satisfying
\begin{alignat}{4}
\label{eq_phi_1}
-\ep\Delta\phi_1=&0&&\quad\text{ in }\Omega,&\qquad\pa_\nu\phi_1+\tau\phi_1=&\xi&&\quad\text{on
}\pa\Omega,\\
\label{eq_phi_2}
-\ep\Delta\phi_2=&\Sigma_jz_jc_j&&\quad\text{ in
}\Omega,&\qquad\pa_\nu\phi_2+\tau\phi_2=&0&&\quad\text{on }\pa\Omega.
\end{alignat}

In view of Remark~\ref{rem_triebel} the functions $\phi_1$ and $\phi_2$
are well-defined, if the corresponding right-hand sides 
in (\ref{eq_phi_1}) and (\ref{eq_phi_2}) belong to suitable regularity
classes. This yields
the following nonlinear problem,
\begin{align}
\label{eq_nlin_1}
\partial_tu+A_Su+P\left[ \Sigma_jz_jc_j\nabla\phi_1\right]&=F(u,c),&\\
\label{eq_nlin_2}
\partial_tc_i-D_i\Delta c_i&=G_i(u,c),&\quad i=1,\ldots ,N,\\
\label{eq_nlin_3}
\pa_\nu c_i&=H_i(c),&\quad i=1,\ldots ,N.\\
\label{eq_nlin_4}
u|_{t=0}&=u^0,&\\
\label{eq_nlin_5}
c_i|_{t=0}&=c_i^0,&\quad i=1,\ldots ,N,
\end{align}
where the nonlinear terms are defined by
\begin{align}
\label{eq_F}
F(u,c)&=-P\left((u\cdot\nabla)u\right)-P\left( \Sigma_jz_jc_j\nabla\phi_2(c)\right),\\
\label{eq_G}
G_i(u,c)&=D_iz_i\left[\nabla c_i^+\cdot(\nabla\phi_1+\nabla\phi_2(c))-c_i^+\Delta\phi_2(c)\right]-\nabla
c_i\cdot u,\\
	&\hspace{5.35cm} i=1,\ldots ,N,\\
\label{eq_H}
H_i(c)&=-z_i\left[c_i^+\pa_\nu(\phi_1+\phi_2(c))\right],\qquad i=1,\ldots ,N.
\end{align}

The linear problem associated to (\ref{eq_nlin_1})-(\ref{eq_nlin_5}) reads as 
\begin{alignat}{2}
\label{eq_lin_1}
\partial_tu+A_Su+ P\left[ \Sigma_jz_jc_j\nabla\phi_1\right]&=f,&\\
\label{eq_lin_2}
\partial_tc_i-D_i\Delta c_i&=g_i,&\quad i=1,\ldots ,N,\\
\label{eq_lin_3}
\pa_\nu c_i&=h_i,&\quad i=1,\ldots ,N,\\
\label{eq_lin_4}
u|_{t=0}&=u^0,&\\
\label{eq_lin_5}
c_i|_{t=0}&=c_i^0,&\quad i=1,\ldots ,N,
\end{alignat}
with right hand sides in suitable function spaces.

Strong $L^p$-solvability of this linear problem is characterized 
in the following lemma.

\begin{lemma}
\label{lemma_linear}
Let $\Omega\subset\R^n$, $n\geq2$, be bounded and smooth, $(n+2)/3<p<\infty$, $0<T<\infty$ and
$\xi\in W^{3-1/p}_p(\pa\Omega)$. Then there is a unique solution
\begin{align*}
u&\in W^{1,p}(0,T;L^p_\sigma(\Omega))\cap L^p(0,T;\DD(A_S))=:\E_{T,p}^u,\\
c&\in W^{1,p}(0,T;L^p(\Omega))\cap L^p(0,T;W^{2,p}(\Omega))=:\E_{T,p}^c,
\end{align*}
to problem (\ref{eq_lin_1})-(\ref{eq_lin_5}) iff the data is contained in the following regularity
classes
\begin{align}
\label{eq_data_1}
f&\in L^p(0,T;L^p_\sigma(\Omega))=:\F_{T,p}^u,\\
\label{eq_data_2}
g&\in L^p(0,T;L^p(\Omega))=:\F_{T,p}^c,\\
\label{eq_data_3}
h&\in W^{1/2-1/2p}_p(0,T;L^p(\partial\Omega))\cap
L^p(0,T;W^{1-1/p}_p(\pa\Omega))=:\F_{T,p}^{c,\pa\Omega},\\
\label{eq_data_4}
u^0&\in W^{2-2/p}_p(\Omega)\cap W^{1,p}_{0,\sigma}(\Omega)=:\X_p^u,\\
\label{eq_data_5}
c^0&\in W^{2-2/p}_p(\Omega)=:\X_p^c
\end{align}
and the initial data satisfies the following compatibility condition
\begin{equation}
\label{eq_data_6}
\pa_\nu c_i^0=\ga_0h_i,\quad i=1,\ldots ,N,\ \text{in the case }p>3,
\end{equation}
 where $\ga_0$ is the trace operator in time for $t=0$.
\end{lemma}

\begin{proof}
The existence of a unique solution $c\in\E_{T,p}^c$ to the parabolic 
problem (\ref{eq_lin_2}), (\ref{eq_lin_3})
and (\ref{eq_lin_5}) iff the data fulfills (\ref{eq_data_2}), (\ref{eq_data_3}), (\ref{eq_data_5})
and (\ref{eq_data_6}) is well-known and contained e.g.\ in \cite{DHP}. From $\xi\in
W^{3-1/p}_p(\pa\Omega)$ we obtain, cf.\ Remark~\ref{rem_triebel},
$\phi_1\in W^{4,p}(\Omega)$, hence $\nabla\phi_1\in W^{3,p}(\Omega)$. 
Note that $\phi_1$ is time-independent. For equation
(\ref{eq_lin_1}) we compute
\begin{equation}
\label{eq_lin_strong}
\|c_i\nabla\phi_1\|_{L^p(L^p)}\leq\|c_i\|_{L^p(L^p)}\|\nabla\phi_1\|_{BUC(\overline\Omega)},
\end{equation}
due to $W^{3,p}(\Omega)\inj BUC(\overline\Omega)$ from Sobolev's embedding theorem, \cite{Adams}. So $\tilde
f:=f-P( \Sigma_iz_ic_i\nabla\phi_1)\in L^p(0,T;L^p_\sigma(\Omega))$. By maximal regularity for the
Stokes operator, see e.g. \cite{sol77, giga85, noll2003} 
the remaining claims follow.
\end{proof}

Let $\E_{T,p}:=\E_{T,p}^u\times\E_{T,p}^c$,
$\F_{T,p}:=\F_{T,p}^u\times\F_{T,p}^c\times\F_{T,p}^{c,\pa\Omega}$ and $\X_p:=\X_p^u\times \X_p^c$.
From Lemma~\ref{lemma_linear}
for every given $T>0$ there is a continuous solution operator
\begin{equation}
\label{eq_sol_op}
L\colon\{(f,g,h,u^0,c^0)\in\F_{T,p}\times \X_p:\text{(\ref{eq_data_6}) is valid}\}\to\E_{T,p},
\end{equation}
for problem (\ref{eq_lin_1})-(\ref{eq_lin_5}). Observe that for $h\in\F_{T,p}^{c,\pa\Omega}$ the
trace at $t=0$ is well-defined if $p>3$. For $p\ne3$ we set
\[
	\hspace{.1ex}_0\F_{T,p}^{c,\pa\Omega}
	:=\left({\hspace{.1ex}_0W}^{1/2-1/2p}_p(0,T;L^p(\pa\Omega))
	\cap L^p(0,T;W^{1-1/p}_p(\pa\Omega))\right)
\]
and denote by $\hspace{.1ex}_0\F_{T,p}$ the space
\[
	\hspace{.1ex}_0\F_{T,p}
	=\F_{T,p}^u\times\F_{T,p}^c
        \times \hspace{.1ex}_0\F_{T,p}^{c,\pa\Omega}.
\]

\begin{lemma}
Let $n\geq2$, $(n+2)/3<p<\infty$, $p\ne3$ and $0< T'\leq T<\infty$. Then there is $C>0$ independent of $T'\leq
T$such that
\[\|L_{|\hspace{.1ex}_0\F_{T',p}\times\{0\}^2}\|_{\LL(\F_{T',p}\times\{0\}^2,\E_{T',p})}\leq C,\]
where $L$ is the solution operator from (\ref{eq_sol_op}).
\end{lemma}

\begin{proof}
We abbreviate $L_T:=L_{|\hspace{.1ex}_0\F_{T,p}}$. From \cite[Lemma 2.5]{Meyries} for any Banach
space $X$ there is a bounded extension operator $\mathcal
E_T^0\colon\hspace{.01ex}_0W^s_p(0,T';X)\to\hspace{.01ex}_0W^s_p(0,T;X)$, whose norm is
independent of $T'\leq T$ for all $s\in[0,2]$. So we may extend $(f,g,h)\in\hspace{.1ex}_0\F_{T',p}$ to some $(\tilde
f,\tilde g,\tilde h)\in\hspace{.1ex}_0\F_{T,p}$ with $(\tilde f,\tilde g,\tilde
h)_{|(0,T')}=(f,g,h)$. Thanks to the fact that
$L_T(\tilde f,\tilde g,\tilde h,0,0)_{|(0,T')}=L(f,g,h,0,0)$ we deduce
\begin{align*}
\|L(f,g,h,0,0)\|_{\E_{T',p}}&\leq\|L_T(\tilde f,\tilde g,\tilde
h,0,0)\|_{\E_{T,p}}\leq\|L_T\|\|(\tilde f,\tilde g,\tilde h)\|_{\hspace{.1ex}_0\F_{T,p}}\\
	&\leq C\|L_T\|\|(f,g,h)\|_{\hspace{.1ex}_0\F_{T',p}}
\end{align*}
for $(f,g,h)\in\hspace{.1ex}_0\F_{T',p}$, where $C>0$ does not depend on $T'\leq T$ by virtue of
\cite[Lemma 2.5]{Meyries}.
\end{proof}

Knowing the mapping properties of the linear solution operator we can
now treat the nonlinear problem (\ref{eq_nlin_1})-(\ref{eq_nlin_5}) 
involving the nonlinear terms $F,G,H$,
given in (\ref{eq_F})-(\ref{eq_H}). We let $N(u,c)=(F(u,c),G(u,c),H(c))$
in the following and write $\phi_2=\phi_2(c)$ for the solution to
(\ref{eq_phi_2}) to indicate the (linear)
$c$-dependence of $\phi_2$. The fact that $F,G,H$ map into the ''right''
spaces when $p$ is chosen
appropriately is the content of

\begin{lemma}
\label{lem_NL}
Let $n\geq2$, $(n+2)/3<p<\infty$ and $T\in(0,\infty)$. Then $N(u,c)\in\F_{T,p}$ for any $(u,c)\in\E_{T,p}^u\times\E_{T,p}^c$.
\end{lemma}

For the proof we have to estimate the individual terms of $N(u,c)$. However, these
computations are contained in the proof of Lemma~\ref{lem_DN} and are therefore omitted here.

Let $(u^0,c^0)\in \X_p$ be fixed from now on. We will prove the existence of a unique local-in-time
strong solution to (\ref{eq_nlin_1})-(\ref{eq_nlin_5}) with $(u(0),c(0))=(u^0,c^0)$. This is
equivalent to the existence of a unique fixed point of the equation
\begin{equation}
\label{eq_fix_a}
(u,c)=L(N(u,c),u^0,c^0).
\end{equation}
It will be convenient to split the solution in a 
part with zero time trace at $t=0$ plus
a remaining part taking care of the non-zero traces. 
If $c\in\E_{T,p}^c$ is a solution of the Nernst-Planck system, 
in the case $p>3$ due to Lemma~\ref{lem_NL} by taking trace of $H(c)$ we obtain
\[
	H(c){|_{t=0}}=H(c^0)\in W^{1-3/p}(\pa\Omega).
\]
We set $h^*:=e^{t\Delta_{\pa\Omega}}H(c^0)$, where $\Delta_{\pa\Omega}$
denotes the Laplace-Beltrami operator on the smooth manifold
$\pa\Omega$. By maximal regularity for $\Delta_{\pa\Omega}$, see e.g.\ 
\cite{PruessSaalSimonett, koehnepruesswilke}, then we have
$
	h^*\in \F_{T,p}^{c,\pa\Omega}
$
and $h^*(0)=H(c^0)$.

If $p<3$ we just set $h^*=0$. According to Lemma~\ref{lemma_linear} we may define a function
$(u^*,c^*)$ by
\begin{equation}
\label{eq_ref_sol}
(u^{*},c^{*}):=L(0,0,h^*,u^0,c^0).
\end{equation}
Subtracting (\ref{eq_ref_sol}) from (\ref{eq_fix_a}) yields the new fixed point problem
\begin{equation}
\label{eq_fix}
(\bar u,\bar c)=L(\bar N(\bar u,\bar c),0,0),
\end{equation}
where
\begin{equation}
\label{eq_def_fix}
\begin{split}
\bar N(\bar u,\bar c)&:=(\bar F(\bar u,\bar c),\bar G(\bar u,\bar c),\bar H(\bar c))\\
	&:=(F(\bar u+u^{*},\bar c+c^{*}),G(\bar u+u^{*},\bar c+c^{*}),H(\bar c+c^{*})-h^*),
\end{split}
\end{equation}
for $(\bar u,\bar c)\in \hspace{1pt}_0\E_{T,p}:=\{(u,c)\in\E_{T,p}:(u(0),c(0))=0\}$.

The advantage of applying the fixed point argument to the zero trace space $\hspace*{0.05em}_0\E_T$
lies in the fact that the embedding constants for embeddings in time do not depend on $T$. This
enables us to choose $T$ as small as we please without having 
the embedding constants blow up.

\begin{lemma}
\label{lem_const_indep}
Let $p\in(1,\infty)$, $I\in\N$, $s_i\in[0,2]$, and $X_i$ be Banach spaces for $i=1,\ldots,I$. Suppose there is $q\in(1,\infty)$ and a Banach space $Y$ such that there is a continuous embedding 
\begin{equation}
\label{eq_emb_mdr_1}
\bigcap_{i=1}^IW^{s_i}_p(0,T_0;X_i)\inj L^q(0,T_0;Y).
\end{equation}
Then for each
$T'\leq T_0$ there is a continuous embedding 
\begin{equation}
\label{eq_emb_mdr_2}
\bigcap_{i=1}^I\hspace{.01ex}_0W^{s_i}_p(0,T';X_i)\inj L^q(0,T';Y),
\end{equation}
whose embedding constant does not depend on $T'\leq T_0$.
\end{lemma}

\begin{proof}
For $T'\in(0,T_0]$ let $\mathcal E_{T'}^0$ be the extension operator 
\[\mathcal E_{T'}^0\colon\hspace{.01ex}_0W^s_p(0,T';X)\to\hspace{.01ex}_0W^s_p(0,T_0;X),\]
$s\in[0,2]$, from \cite[Lemma 2.5]{Meyries}, whose norm does not depend
on $T'$; note that $\mathcal E_T^0$ in \cite[Lemma 2.5]{Meyries} does
not depend on $s\in[0,2]$, that is, for each $s\in[0,2]$ we have the
same extension operator. Let $\mathcal
R_{T'}$ be the corresponding restriction operator
\[\mathcal R_{T'}\colon W^s_p(0,T_0;Y)\to W^s_p(0,T';Y),\]
Then the claim follows, since the following diagram commutes:
\begin{equation*}
\begin{array}{ccc}
\bigcap_{i=1}^I\hspace{.01ex}_0W^{s_i}_p(0,T';X_i)&\xrightarrow{\mathcal E_{T'}^0}&\bigcap_{i=1}^I\hspace{.01ex}_0W^{s_i}_p(0,T_0;X)\\
\downarrow&&\downarrow\\
L^q(0,T';Y)&\xleftarrow{\mathcal R_{T'}}&L^q(0,T_0;Y),
\end{array}
\end{equation*}
$T'\leq T_0$.
\end{proof}
\begin{remark}
Note that, due to the embedding (\ref{eq_emb_mdr_1}), the embedding
constant for (\ref{eq_emb_mdr_2}) can depend on $T_0$.
\end{remark}
For the proof of Theorem~\ref{theorem_strong_loc} it will be necessary to show a contraction estimate for the map $L\circ
\bar N$, cf.\ (\ref{eq_fix}). This is somewhat technical due to the
intricate structure of the trace space
$\F_{T,p}^{c,\pa\Omega}$. Therefore we prove it as a separate lemma.

\begin{lemma}
\label{lem_DN}
Let $n\geq2$, $(n+2)/3<p<\infty$, $0<T'\leq T<\infty$ and $R>0$. Then there are $\al,C>0$, independent of $T'$
and $R$ such that
\begin{equation}
\label{eq_Fre_est}
\begin{split}
\|&\bar N(\bar u^1,\bar c^1)-\bar N(\bar u^2,\bar c^2)\|_{\F_{T',p}}\leq C\|(\bar u^1,\bar
c^1)-(\bar u^2,\bar c^2)\|_{\E_{T',p}}\times\\
&\bigg[\bigg\{R+\|u^*\|_{L^{3p}(0,T';L^{3p})}+\|\nabla
u^*\|_{L^{3p/2}(0,T';L^{3p/2})}\\
	&\quad\quad+\|c^*\|_{L^{3p}(0,T';L^{3p})}+\|c^*\|_{L^{3p/2}(0,T';W^{1,3p/2})}\\
&\quad\quad+\|\nabla\phi_2(c^*)\|_{L^{3p}(0,T';L^{3p})}+\|\nabla\phi_2(c^*)\|_{L^{3p/2}(0,T';W^{1,3p/2})}\\
	&\quad\quad+\|\pa_tc^*\|_{L^p(0,T';L^p)}
	+\|\pa_t(\nabla\phi_2(c^*))\|_{L^p(0,T';W^{1,p})}\bigg\}\\
&\quad\quad+T'^\al\bigg\{\|c^*\|_{BUC(0,T';W^{2-2/p}_p)}
+\|\nabla\phi_1\|_{W^{1,3p}}\\
&\quad\quad+\|\nabla\phi_2(c^*)\|_{BUC(0,T';BUC)}
\bigg\}\bigg]
\end{split}
\end{equation}
for $(\bar u^1,\bar c^1),(\bar u^2,\bar c^2)\in B_{\hspace*{0.05em}_0\E_{T',p}}(0,R)$, the closed ball centered at $0$
in $\hspace*{0.05em}_0\E_{T',p}$ with radius $R$.
\end{lemma}

\begin{proof}
Before estimating the individual terms in $\bar N(\bar u^1,\bar c^1)-\bar N(\bar u^2,\bar c^2)$
recall the following estimate which is obtained as a consequence of the 
mixed derivative theorem; 
for $\al,\be\in\N$, $p\in(1,\infty)$ it holds for $s\in(0,1)$ that
\begin{equation}
\label{eq_MDT}
W^{\al,p}(0,T;L^p(\Omega))\cap L^p(0,T;W^{\be,p}(\Omega))\inj W^{\al
s}_p(0,T;W^{\be(1-s)}_p(\Omega)).
\end{equation}
A proof for the case $T=\infty$ and the whole space can be found in \cite{Sob75}. Applying
suitable extensions relation (\ref{eq_MDT}) follows for our situation, see also
\cite{PruessSaalSimonett}.
With (\ref{eq_MDT}) and Sobolev's embedding theorem it is straight forward to see
that
\begin{equation}
\label{eq_DN_emb_1}
\begin{split}
W^{1,p}(0,T;L^p)\cap L^p(0,T;W^{2,p})&\inj W^{s_1}_p(0,T;W^{2-2s_1}_p)\inj L^{3p}(0,T;L^{3p}),
\end{split}
\end{equation}
\begin{equation}
\label{eq_DN_emb_2}
\begin{split}
\nabla(W^{1,p}(0,T;L^p)\cap L^p(0,T;W^{2,p}))&\inj W^{s_2}_p(0,T;W^{1-2s_2}_p)\\
	&\inj L^{3p/2}(0,T;L^{3p/2})
\end{split}
\end{equation}
for $s_1=2/(3p)$ and $s_2=1/(3p)$. Of course (\ref{eq_MDT})-(\ref{eq_DN_emb_2}) are also valid for $T$ substituted by $T'$.

These facts represent crucial ingredients for the proof. Let us, for the sake of readability,
introduce the following abbreviations:
\[\hat u:=\bar u^1-\bar u^2,\quad \hat c:=\bar c^1-\bar c^2,\quad \hat c_i:=\bar c_i^1-\bar
c_i^2\]
and accordingly $\hat c^+$ and $\hat c_i^+$.
The individual terms in $\bar N(\bar u^1,\bar c^1)- \bar N(\bar u^2,\bar c^2)$ can therefore be
written in the following form:
\begin{align}
\label{eq_repr_F}
\begin{split}
&\bar F(\bar u^1,\bar c^1)-\bar F(\bar u^2,\bar c^2)=P\big((u^*+\bar u^2)\cdot\nabla\hat u+\hat
u\cdot\nabla(u^*+\bar u^1)\\
&\quad+\sum_jz_j((c_j^*-\bar c_j^2)\nabla\phi_2(\hat c)+\hat c_j\nabla\phi_2(c^*+\hat c^1))\big),
\end{split}\\
\label{eq_repr_G}
\begin{split}
&\bar G_i(\bar u^1,\bar c^1)-\bar G_i(\bar u^2,\bar c^2)=D_iz_i\big(\nabla\hat
c_i^+(\nabla\phi_1+\nabla\phi_2(c^*+\bar c^1))\\
	&\quad+\nabla(c_i^*+\bar c_i^{2,+})\nabla\phi_2(\hat c)+(c_i^*+\bar c_i^{2,+})\Delta\phi_2(\hat c)\\
&\quad+\hat c_i^+\Delta\phi_2(c^*+\bar c^1)\big)+(u^*+\bar u^2)\nabla\hat c_i+\hat
u\cdot\nabla(c_i^*+\bar c_i^1)
\end{split}\\
\label{eq_repr_H}
\begin{split}
&\bar H_i(\bar c^1)-\bar H_i(\bar c^2)=-z_i\big((c_i^*+\bar c_i^{2,+})\pa_\nu\phi_2(\hat c)+\hat
c_i^+\pa_\nu(\phi_1+\phi_2(c^*+\bar c^1))\big).
\end{split}
\end{align}

Concerning the individual summands in $F$ and $G$ we observe that the term resulting from the
convection term $(u\cdot\nabla)u$ has the weakest regularity properties as compared to the remaining
terms in $F$ and $G$. Since $F,G$ have both to be estimated in $L^p(0,T';L^p)$ it is sufficient to
consider only the first two summands in (\ref{eq_repr_F}); the other terms left can be treated in
exactly the same way. By the continuity of $P$ and H\"older's inequality, there is a constant $C>0$ not depending
on $T'$ such that
\begin{equation}
\label{eq_est_barF1}
\begin{split}
\|&P\big((u^*+\bar u^2)\nabla\hat u+\hat u\nabla(u^*+\bar u^1)\big)\|_{L^p(0,T';L^p)}\\
&\leq C((\|u^*\|_{L^{3p}(0,T';L^{3p})}+\|\bar u^2\|_{L^{3p}(0,T';L^{3p})})\|\nabla\hat
u\|_{L^{3p/2}(0,T';L^{3p/2})}\\
&\quad+(\|\nabla u^*\|_{L^{3p/2}(0,T';L^{3p/2})}+\|\nabla\bar u^1\|_{L^{3p/2}(0,T';L^{3p/2})})\|\hat
u\|_{L^{3p}(0,T';L^{3p})}).\\
\end{split}
\end{equation}
From the fact that $\hat u\in\hspace{.01ex}_0\E_{T',p}^u$ and relations (\ref{eq_DN_emb_1}),
(\ref{eq_DN_emb_2}) Lemma~\ref{lem_const_indep} assures the existence of a constant $C'>0$, not
depending on $T'$, such that
\begin{equation}
\label{eq_est_barF2}
\|\hat u\|_{L^{3p}(0,T';L^{3p})}+\|\nabla\hat u\|_{L^{3p/2}(0,T';L^{3p/2})}\leq C'\|\hat
u\|_{\E_{T',p}^u}.
\end{equation}
So combining (\ref{eq_est_barF1}) and (\ref{eq_est_barF2}) we deduce
\begin{equation*}
\begin{split}
\|P\big(&(u^*+\bar u^2)\nabla\hat u+\hat u\nabla(u^*+\bar u^1)\big)\|_{L^p(0,T';L^p)}\\
&\leq C\|\hat u\|_{\E_{T',p}}(R+\|u^*\|_{L^{3p}(0,T';L^{3p})}+\|\nabla
u^*\|_{L^{3p/2}(0,T';L^{3p/2})})
\end{split}
\end{equation*}
For further estimation note that, Remark~\ref{rem_triebel},
\begin{equation}
\label{eq_rem_triebel}
\|\phi_2(\bar c)\|_{W^{1,p}(0,T';W^{2,p})\cap L^p(0,T';W^{4,p})}\leq C\|\bar c\|_{\E_{T',p}^c},
\end{equation}
for $\bar c\in\hspace{.01ex}_0\E_{T',p}^c$, with some constant $C>0$ independent of $T'$. So in
performing analogous computations for the other summands of $\bar F$ and $\bar G$, taking into
account $\hat c_i^+\leq|\hat c_i|$, we establish
\begin{equation}
\label{eq_est_F_G}
\begin{split}
\|(\bar F,\bar G)&(\bar u^1,\bar c^1)-(\bar F,\bar G)(\bar u^2,\bar c^2)\|_{L^p(0,T';L^p)}\leq
C\|(\bar u^1,\bar c^1)-(\bar u^2,\bar c^2)\|_{\E_{T',p}}\times\\
&\quad(R+\|u^*\|_{L^{3p}(0,T';L^{3p})}+\|
u^*\|_{L^{3p/2}(0,T';W^{1,3p/2})}+\|c^*\|_{L^{3p}(0,T';L^{3p})}\\
	&\qquad+\|c^*\|_{L^{3p/2}(0,T';W^{1,3p/2})}
	+\|\nabla\phi_1\|_{L^{3p/2}(0,T';L^{3p/2})}\\
&\qquad+\|\nabla\phi_2(c^*)\|_{L^{3p}(0,T';L^{3p})}
+\|\nabla\phi_2(c^*)\|_{L^{3p/2}(0,T';W^{1,3p/2})},
\end{split}
\end{equation}
for $(\bar u^1,\bar c^1),(\bar u^2,\bar c^2)\in B_{\hspace*{0.05em}_0\E_{T',p}}(0,R)$.

We are left with the estimation of (\ref{eq_repr_H}) in $\F_{T',p}^{c,\pa\Omega}$. For the treatment
of this expression let us note that since $p>(n+2)/3$ from 
(\ref{eq_MDT}), Sobolev's embedding theorem, and by 
choosing $\de<\frac1{2p}(3p-n-2)$ such that
$s=1/p+\de\in(0,1)$ we have
\begin{equation}
\label{eq_BUC_phi}
\nabla\phi\in W^s_p(0,T';W^{3-2s}_p(\Omega))\inj BUC([0,T'];BUC(\overline\Omega)).
\end{equation}

Because of the smoothness of $\pa\Omega$ the normal vector field $\nu$ is $BUC^1(\pa\Omega)$. From
$BUC^1(\pa\Omega)\cdot L^p(\pa\Omega)\inj L^p(\pa\Omega)$ and $BUC^1(\pa\Omega)\cdot
W^{1,p}(\pa\Omega)\inj W^{1,p}(\pa\Omega)$ we see, by interpolation, \cite{triebel},
$BUC^1(\pa\Omega)\cdot W^{1-1/p}_p(\pa\Omega)\inj W^{1-1/p}_p(\pa\Omega)$; thus we may estimate
\begin{equation}
\label{eq_est_H_first}
\begin{split}
\|\bar H_i&(\bar c^1)-\bar H_i(\bar c^2)\|_{\F_{T',p}^{c,\pa\Omega}}\leq C\|(c_i^*+\bar c_i^{2,+})\nabla\phi_2(\hat c)\|_{ L^p(0,T';W^{1,p})\cap
W^{1/2}_p(0,T';L^p)}\\
&\qquad+C\|\hat c_i^+(\nabla\phi_1+\nabla\phi_2(c^*+\bar c^1))\|_{ L^p(0,T';W^{1,p})\cap
W^{1/2}_p(0,T';L^p)},
\end{split}
\end{equation}
using the fact that $\F_{T',p}^{c,\pa\Omega}$ is the spacial 
trace space of $ L^p(0,T';W^{1,p})\cap
W^{1/2}_p(0,T';L^p)$, see e.g. \cite{DHP}. Note that, due to $\bar
H_i(\bar c^j)\in\hspace{.01ex}_0\F_{T',p}^{c,\pa\Omega}$, the constant $C$ in (\ref{eq_est_H_first}) does not depend on $T'$. This can be seen by a similar argument as in the proof of Lemma~\ref{lem_const_indep}.

Let us consider the first term on the right hand side of (\ref{eq_est_H_first}). 
First we estimate the norm in $L^p(0,T';W^{1,p})$. 
Applying the H\"older inequality yields
\begin{equation}
\label{eq_est_H_1}
\begin{split}
& \|(c_i^*+\bar c_i^{2,+})\nabla\phi_2(\hat c)\|_{L^p(0,T';W^{1,p})}\\
&\leq(\|c_i^{*}\|_{L^{3p/2}(0,T';W^{1,3p/2})}
+\|\bar c_i^2\|_{L^{3p/2}(0,T';W^{1,3p/2})})\|\nabla\phi_2(\hat
c)\|_{L^{3p}(0,T';W^{1,3p})}.
\end{split}
\end{equation}
Next we use the fact that $\bar c_i^1(0)=\bar c_i^1(0)=0$ and hence also  
$\phi_2(\hat c)(0)=0$, which gives us by applying 
(\ref{eq_DN_emb_1}), Lemma~\ref{lem_const_indep}, and afterwards
relation (\ref{eq_rem_triebel}) that
\begin{equation}
\label{eq_est_H_12}
\begin{split}
& \|(c_i^*+\bar c_i^{2,+})\nabla\phi_2(\hat c)\|_{L^p(0,T';W^{1,p})}\\
	&\leq C (\|
	c_i^*\|_{L^{3p/2}(0,T';W^{1,3p/2})}+R)\|\phi_2(\hat
	c)\|_{\E_{T',p}^c}\\
	&\leq C (\|
	c_i^*\|_{L^{3p/2}(0,T';W^{1,3p/2})}+R)\|\hat c\|_{\E_{T',p}^c}.
\end{split}
\end{equation}

Let us now consider the norm in $W^{1/2}_p(0,T';L^p(\Omega))$.
We will not calculate directly with Sobolev-Slobodeckij-norms. Instead,
since $\phi_2$ enjoys
rather convenient regularity properties, we can afford to use the interpolation inequality,
\cite{triebel},
\begin{equation}
\label{eq_est_H_int}
\begin{split}
\|(c_i^*+&\bar c_i^{2,+})\nabla\phi_2(\hat c)\|_{W^{1/2}_p(0,T';L^p)}\\
&\leq C(T)\|(c_i^*+\bar c_i^{2,+})\nabla\phi_2(\hat c)\|_{ L^p(0,T';L^p)}^{1/2}\|(c_i^*+\bar
c_i^{2,+})\nabla\phi_2(\hat c)\|_{W^{1,p}(0,T';L^p)}^{1/2}.
\end{split}
\end{equation}
The fact that the constant $C(T)$ in (\ref{eq_est_H_int}) does not depend on $T'$ 
again can be seen by applying the extension operator in Lemma~\ref{lem_const_indep} 
to $(0,T)$ and the fact that $(c_i^*+\bar c_i^{2,+})\nabla\phi_2(\hat c)\in\hspace{.01ex}_0W^{1/2}_p(0,T';L^p)$.  Expanding the norm in $W^{1,p}(0,T';L^p)$ gives
\begin{equation}
\label{eq_est_H_6}
\begin{split}
\|(&c_i^*+\bar c_i^{2,+})\nabla\phi_2(\hat c)\|_{W^{1,p}(0,T';L^p)}=\|(c_i^*+\bar
c_i^{2,+})\nabla\phi_2(\hat c)\|_{L^p(0,T';L^p)}\\
&+\|\pa_t(c_i^*+\bar c_i^{2,+})\nabla\phi_2(\hat c)\|_{L^p(0,T';L^p)}+\|(c_i^*+\bar
c_i^{2,+})\pa_t(\nabla\phi_2(\hat c))\|_{L^p(0,T';L^p)}.
\end{split}
\end{equation}
Note that, since $\phi_2(\hat c)(0)=0$, we have by relations
(\ref{eq_rem_triebel}), (\ref{eq_BUC_phi}), and
Lemma~\ref{lem_const_indep} that
\begin{equation}
\label{eq_est_H_2}
\begin{split}
\|\nabla&\phi_2(\hat c)\|_{BUC(0,T';BUC)}\leq C(T)\|\nabla\phi_2(\hat
c)\|_{W^s_p(0,T';W^{3-2s}_p(\Omega))}\\
&\leq C'(T) \|\phi_2(\hat c)\|_{W^{1,p}(0,T';W^{2,p})\cap L^p(0,T';W^{4,p})}
\leq C''(T)\|\hat
c\|_{\E_{T',p}^c}.
\end{split}
\end{equation}
This implies for the second term in (\ref{eq_est_H_6}) that
\begin{equation}
\label{eq_est_H_7}
\begin{split}
\|\pa_t&(c_i^*+\bar c_i^{2,+})\nabla\phi_2(\hat c)\|_{L^p(0,T';L^p)}\\
&\leq (\|\pa_tc_i^*\|_{L^p(0,T';L^p)}+\|\pa_t\bar c_i^2\|_{L^p(0,T';L^p)})\|\nabla\phi_2(\hat
c)\|_{BUC(0,T';BUC)}\\
	&\leq C(T)(\|\pa_tc_i^*\|_{L^p(0,T';L^p)}+R)\|\hat
	c\|_{\E_{T',p}^c}.
\end{split}
\end{equation}
For the third term in (\ref{eq_est_H_6}) we compute
\begin{equation}
\label{eq_est_H_8}
\begin{split}
\|(c_i^*&+\bar c_i^{2,+})\pa_t(\nabla\phi_2(\hat c))\|_{L^p(0,T';L^p)}\\
&\leq(\|c_i^*\|_{BUC(0,T';L^q)}+\|\bar c_i^2\|_{BUC(0,T';L^q)})\|\pa_t(\nabla\phi_2(\hat
c))\|_{L^p(0,T';L^{q'})}
\end{split}
\end{equation}
with $1/q+1/q'=1/p$. From \cite{amann} it holds
\begin{equation}
\label{eq_amann}
W^{1,p}(0,T';L^p(\Omega))\cap L^p(0,T';W^{2,p}(\Omega))\inj BUC(0,T';W^{2-2/p}_p(\Omega)).
\end{equation}
So the term $\|c_i^*\|_{BUC(0,T';L^q)}$ in (\ref{eq_est_H_8}) is well-defined if it holds true that
$W^{2-2/p}_p(\Omega)\inj L^q$. At the same time we have to assure that $W^{1,p}(\Omega)\inj
L^{q'}(\Omega)$, since $\pa_t\nabla\phi_2\in L^p(W^{1,p})$. The critical cases occur when
$p<(n+2)/2$. In this situation choose $q=np/(n-2p+2)$ and $q'=np/(2p-2)$. So, recalling $p>(n+2)/3$,
it holds that $W^{2-2/p}_p(\Omega)\inj L^q$ and $W^{1,p}(\Omega)\inj L^{q'}$ thanks to Sobolev's
embedding theorem, hence
\begin{equation}
\label{eq_est_H_9}
\begin{split}
&\|(c_i^*+\bar c_i^{2,+})\pa_t(\nabla\phi_2(\hat c))\|_{L^p(0,T';L^p)}\\
&\leq C(\|c_i^*\|_{BUC(0,T';W^{2-2/p}_p)}+\|\bar
c_i^2\|_{BUC(0,T';W^{2-2/p}_p)})\|\pa_t(\nabla\phi_2(\hat c))\|_{L^p(0,T';W^{1,p})}\\
	&\leq C(T)(\|c_i^*\|_{BUC(0,T';W^{2-2/p}_p)}+R)\|\hat c\|_{\E_{T',p}^c}.
\end{split}
\end{equation}
In contrast to the $L^r$- norms in time the appearing $BUC$-norms of 
$c_i^*$ will in general not become small as $T'\to 0$. Here we need an 
additional factor $T'^\alpha$ which will be provided by estimating
the first factor in (\ref{eq_est_H_int}) a little more carefully as in 
(\ref{eq_est_H_1}). Indeed, we can estimate
\begin{equation}
\begin{split}
\label{eq_est_H_3}
&\|(c_i^*+\bar c_i^{2,+})\nabla\phi_2(\hat c)\|_{L^p(0,T';L^p)}\\
&\leq (\|
c_i^*\|_{L^{3p/2}(0,T';L^{3p/2})}+\|
\bar c_i^{2,+}\|_{L^{3p/2}(0,T';L^{3p/2})})\|\nabla\phi_2
(\hat c)\|_{L^{3p}(0,T';L^{3p})}\\
&\leq C(T)T'^\al (\|
c_i^*\|_{L^{3p}(0,T';L^{3p})}+R)\|\hat c\|_{\E_{T',p}^c}
\end{split}
\end{equation}
for a certain $\alpha>0$. Note that here and in the sequel the appearing
$\alpha$'s at first might be different. At the end, however, we
just need one $\alpha>0$ and we can always choose the smallest of
finitely many ones. Thus, as for the generic constant $C$, in the sequel
$\alpha$ can change from line to line.

From relations (\ref{eq_est_H_int}), (\ref{eq_est_H_6}), (\ref{eq_est_H_7}),
(\ref{eq_est_H_8}), (\ref{eq_est_H_9}), (\ref{eq_est_H_3}), 
and Young's inequality we now obtain
\begin{equation}
\label{eq_est_H_tot_2}
\begin{split}
&\|(c_i^*+\bar c_i^{2,+})\nabla\phi_2(\hat c)\|_{W^{1/2}_p(0,T';L^p)}\\
	&\leq
	C(T)T'^\al\big(\|c_i^*\|_{L^{3p}(0,T';L^{3p})}+R\big)^{1/2}\\
&\qquad\cdot\big(\|\pa_tc_i^*\|_{L^p(0,T';L^p)}
+\|c_i^*\|_{BUC(0,T';W^{2-2/p}_p)}+R\big)^{1/2}\|
\hat c\|_{\E_{T',p}^c}\\
&\leq
	C(T)T'^\al\big(\|c_i^*\|_{L^{3p}(0,T';L^{3p})}\\
&\qquad+\|\pa_tc_i^*\|_{L^p(0,T';L^p)}
+\|c_i^*\|_{BUC(0,T';W^{2-2/p}_p)}+R\big)\|\hat c\|_{\E_{T',p}^c}.
\end{split}
\end{equation}
Collecting (\ref{eq_est_H_12}) and (\ref{eq_est_H_tot_2}) 
then gives
\begin{equation}
\label{eq_est_H_first_}
\begin{split}
&\|(c_i^*+\bar c_i^{2,+})\nabla\phi_2(\hat
c)\|_{\F_{T',p}^{c,\pa\Omega}}\\
&\leq
C(T)\bigg(\|c^*\|_{L^{3p/2}(0,T';W^{1,3p/2})}
+\|c^*\|_{L^{3p}(0,T';L^{3p})}\\
&\qquad+\|\pa_tc^*\|_{L^p(0,T';L^p)}
+T'^\al\|c^*\|_{BUC(0,T';W^{2-2/p}_p)}+R
\bigg)\|\hat c\|_{\E_{T',p}^c}.
\end{split}
\end{equation}

Now we turn to the second term in (\ref{eq_est_H_first}). 
This term is basically of the same structure as the first term.
For its estimation essentially analogous
calculations and arguments can be carried out. Thus
we will be brief in detail here.
Also note that $\nabla\phi_1$ has the same
regularity in space as $\nabla\phi_2(c^*)$ has, but due to its 
time-independence it can even be treated 
in an easier way. This is why we leave this
term out of the following computation.
In fact, by merely interchanging the roles of the factors
we obtain similarly to (\ref{eq_est_H_1}) and (\ref{eq_est_H_12})
that
\begin{equation}
\label{eq_est_H_1_4}
\begin{split}
&\|\hat c_i^+(\nabla\phi_2(c^*+\bar c^1))\|_{L^p(0,T';W^{1,p})}\\
&\leq C(T)
\big(\|\nabla\phi_2(c^*)\|_{L^{3p}(0,T';W^{1,3p})}+R\big)
\|\hat
c_i\|_{\E_{T',p}^c}.
\end{split}
\end{equation}
Accordingly for the norm in $W^{1/2}_p(0,T';L^p)$,
by performing similar arguments as in 
(\ref{eq_est_H_int})-(\ref{eq_est_H_tot_2}) we deduce
\begin{equation}
\label{eq_est_H_tot_3}
\begin{split}
&\|\hat c_i^+(\nabla\phi_2(c^*+\bar c^1))\|_{W^{1/2}_p(0,T';L^p)}\\
&\leq
C(T)T'^\al\bigg(\|\nabla\phi_2(c^*)\|_{BUC(0,T';BUC)}
+\|\nabla\phi_2(c^*)\|_{L^{3p}(0,T';L^{3p})}\\
&\qquad+\|\pa_t(\nabla\phi_2(c^*))\|_{L^p(0,T';W^{1,p})}+R\bigg)\|\hat
c\|_{\E_{T',p}^c}.
\end{split}
\end{equation}
By the fact that 
$\nabla\phi_1$ and $\nabla\phi_2(c^*)$ have the same regularity in
space, the same terms will appear for $\nabla\phi_1$. Since 
$\nabla\phi_1$ is independent of time then by the Sobolev embedding
merely a term as $T'^\alpha\|\nabla \phi_1\|_{W^{1,3p}}$ remains.
Hence, collecting (\ref{eq_est_H_1_4}) and (\ref{eq_est_H_tot_3}) gives us
\begin{equation}
\label{eq_est_H_second_}
\begin{split}
&\|\hat c_i(\nabla\phi_1+\nabla\phi_2(c^*+\bar
c^1))\|_{\F_{T',p}^{c,\pa\Omega}}\\
&\leq
C(T)\bigg(R+T'^\al\big[\|\nabla\phi_2(c^*)\|_{BUC(0,T';BUC)}
+\|\nabla\phi_1\|_{W^{1,3p}}\big]\\
&\qquad+\|\nabla\phi_2(c^*)\|_{L^{3p}(0,T';W^{1,3p})}
+\|\pa_t(\nabla\phi_2(c^*))\|_{L^p(0,T';W^{1,p})}\bigg)\|\hat
c\|_{\E_{T',p}}.
\end{split}
\end{equation}
Finally, we conclude (\ref{eq_Fre_est}) from relations (\ref{eq_est_F_G}), (\ref{eq_est_H_first}),
(\ref{eq_est_H_first_}) and (\ref{eq_est_H_second_}).
\end{proof}


Having Lemma~\ref{lem_DN} at hand, we are in position to prove Theorem~\ref{theorem_strong_loc}.

\begin{proof}[Proof. (of Theorem~\ref{theorem_strong_loc})]
Recall the solution operator $L$ from (\ref{eq_sol_op}) and its restriction to
$\hspace{.01ex}_0\F_{T',p}$. Its operator norm is denoted by $\|L\|$. 
Thanks to Lemma~\ref{lem_DN} we
can choose first $R>0$ and then $T'>0$ small enough such that
\begin{align*}
\|N(\bar u^1,\bar c^1)-N(\bar u^2,\bar c^2)\|_{\F_{T',p}}\leq\frac 1{2\|L\|}\|(\bar u^1,\bar
c^1)-(\bar u^2,\bar c^2)\|_{\E_{T',p}},
\end{align*}
for $(\bar u^1,\bar c^1),(\bar u^2,\bar c^2)\in B_{\hspace*{0.05em}_0\E_{T',p}}(0,R)$.
We may assume that $T'>0$ was also chosen such that
\begin{align*}
\|\bar N(0,0)\|_{\F_{T',p}}=\|N(u^*,c^*)\|_{\F_{T',p}}\leq\frac{R}{2\|L\|},
\end{align*}
since all terms in $N(u^*,c^*)$ consist of time integrals over given functions. Now we can apply the
contraction mapping theorem. The estimates
\begin{align*}
\|L\bar N(\bar u^1,\bar c^1)-L\bar N(\bar u^2,\bar c^2)\|_{\E_{T',p}}&\leq\|L\|\|N(\bar u^1,\bar
c^1)-N(\bar u^2,\bar c^2)\|_{\F_{T',p}}\\
	&\leq\frac 12\|(\bar u^1,\bar c^1)-(\bar u^2,\bar c^2)\|_{\E_{T',p}}
\end{align*}
for $(\bar u^1,\bar c^1),(\bar u^2,\bar c^2)\in B_{\hspace*{0.05em}_0\E_{T',p}}(0,R)$ and
\begin{align*}
\|L\bar N(\bar u,\bar c)\|&\leq\|L\|\big(\|\bar N(\bar u,\bar c)-\bar N(0,0)\|+\|\bar
N(0,0)\|\big)\\
	&\leq\|L\|\left(\frac{1}{2\|L\|}\|(\bar u,\bar c)\|+\frac{R}{2\|L\|}\right)\\
	&\leq R
\end{align*}
for $(\bar u,\bar c)\in B_{\hspace*{0.05em}_0\E_{T',p}}(0,R)$ assure the existence of a fixed point and hence the existence and uniqueness of a local-in-time
solution to ($P'$). Note that our computation implies also the unique local-in-time existence of
strong solutions to the original problem ($P$), cf.\ Remark~\ref{rem_loc_nonneg}.
The claimed regularity for the pressure then is obtained as an
easy consequence of formula (\ref{pressrec}).

Let us denote the solutions to ($P'$) and ($P$) by $(u^a,c^a)$ and $(u,c)$ respectively and let
$c_i^0\geq0$, for $i=1,\ldots,N$. If we can show that for $(u^a,c^a)$ we have $c^a\geq0$ almost
everywhere, then it is also a solution to ($P$) and by uniqueness we see that $(u,c)=(u^a,c^a)$. Let $c_i^-=\min\{c_i,0\}$.
Note that
\[\frac12\Dt\|c_i^{a,-}\|_2^2=\int_\Omega\pa_tc_i^{a,-}c_i^{a,-}dx=\int_\Omega\pa_tc_i^{a}c_i^{a,-}dx=\int_\Omega\pa_tc_i^{a,-}c_i^adx.\]
From the fact that $\pa_tc_i^{a,-}\in L^p(0,T';L^p)$ and $c_i^a\in\E_{T',p}^c$ it is easy to see
using H\"older's inequality and Sobolev's embedding theorem, that the last integral is well-defined.
Hence formal multiplication of (\ref{eq_NP+}) by $c_i^{a,-}$ and integration over $\Omega$ gives
\[\frac12\Dt\|c_i^{a,-}\|_2^2+\int_\Omega(-D_i\nabla c_i^a-D_iz_ic_i^{a,+}\nabla\phi+c_i^au^a)\nabla
c_i^{a,-}dx=0.\]
Integration by parts, taking into account (\ref{eq_NP_bdy+}), yields
\begin{equation}
\label{eq_nonneg_1}
\frac12\Dt\|c_i^{a,-}\|_2^2+D_i\|\nabla c_i^{a,-}\|_2^2=D_iz_i\int_\Omega c_i^{a,+}\nabla\phi\cdot\nabla
c_i^{a,-}dx-\int_\Omega c_i^au\cdot\nabla c_i^{a,-}dx.
\end{equation}
The first integral on the right-hand side is zero since $c_i^{a,+}\cdot\nabla c_i^{a,-}=0$. The
second one also vanishes due to solenoidality of $u$, using
\[\int_\Omega c_i^au\cdot\nabla c_i^{a,-}dx=\frac12\int_\Omega u\cdot\nabla( c_i^{a,-})^2=0.\]
Here we used the fact that $c_i^a\nabla c_i^a\in L^1(0,T';L^2(\Omega))$, to make sure that the
integral exists. Finally, integrating (\ref{eq_nonneg_1}) over $(0,t)$ gives
\begin{align*}
\|c_i^{a,-}(t)\|_2^2+2D_i\int_0^t\|\nabla c_i^{a,-}\|_2^2=\|c_i^{0,-}\|_2^2\leq0,
\end{align*}
for almost all $t\in(0,T)$. So $c_i^a\geq0$ almost everywhere, whence $(u^a,c^a)$ and $(u,c)$
coincide. This proves Theorem~\ref{theorem_strong_loc}.
\end{proof}

\begin{remark}
\label{rem_unique}
Note that the solution $(u,c)$ is unique up to the maximal time of existence $T_{max}$. This can be seen as follows: Let $(u,c),(\hat u,\hat c)\in\E_{T,p}$ be two solutions with initial value $(u^0,c^0)\in\X_p$, and let $(0,T')$ be the maximal time interval such that $(u,c),(\hat u,\hat c)$ coincide, $T'\leq T<T_{max}$. Suppose $T$ could be chosen in such a way that $T'<T$. From \cite{amann} we have $(u,c),(\hat u,\hat c)\in BUC(0,T';\X_p)$, so therefore $(u(T'),c(T'))=(\hat u(T'),\hat c(T'))$ and applying Theorem~\ref{theorem_strong_loc} with initial value $(u(T'),c(T'))\in\X_p$ implies that $(u,c)$ and $(\hat u,\hat c)$ necessarily coincide on $(T',T'+\ep)$, say; 
this contradicts the maximality of $T'$.
\end{remark}

\begin{remark}
\label{rem_phi_1}
The condition $\xi\in W^{3-1/p}_p(\pa\Omega)$ in
Theorem~\ref{theorem_strong_loc} can be relaxed to, e.g.,
\begin{equation}
\label{eq_xi}
\xi\in W^{1+2n/(3p)-1/p}_p(\pa\Omega).
\end{equation}
It is a straight forward calculation that with this condition the facts
that $\nabla\phi_1\in BUC(\overline\Omega)$ and $\nabla^2\phi_1\in
L^{3p}(\Omega)$ remain valid. Having this information at hand, we see
that estimate (\ref{eq_lin_strong}) and the reasoning to obtain 
(\ref{eq_est_H_second_}) is still true. There are no further points where the regularity of $\phi_1$ is used. However, because of the unphysical dependence on the space dimension, we chose to state Theorem~\ref{theorem_strong_loc} in the given form.
\end{remark}

\begin{corollary}
\label{cor_mass_cons}
Let $(u,c)$ be the local strong solution to ($P$) from Theorem \ref{theorem_strong_loc} with initial
data $(u^0,c^0)$ where $c_i^0\geq0$, $i=1,\ldots,N$. Then the initial masses $m_i:=\int_\Omega
c_i^0(x)dx$ are conserved, i.e.
\[\|c_i(t)\|_1=\|c_i^0\|_1,\quad t\geq0,\quad i=1,\ldots N.\]
\end{corollary}

\begin{proof}
The claim follows from nonnegativity of $c_i$ and
\begin{equation}
\label{eq_mass_cons}
\begin{split}
\Dt\int_\Omega c_idx&=\int_\Omega \pa_tc_idx=-\int_\Omega\divv J_idx=-\int_{\pa\Omega}J_i\cdot\nu dx=0,\quad i=1,\ldots N,
\end{split}
\end{equation}
due to (\ref{eq_NP}) and (\ref{eq_NP_bdy}).
\end{proof}

\section{Global well-posedness in two dimensions}
\label{global}

Before we turn our attention to the estimation of the local solution $(u,c)$ from Theorem~\ref{theorem_strong_loc} we collect some known technical results
which will be employed in the proofs of global well-posedness and
stability. Below we write
\[\log^+x=\max\{0,\log x\},\text{ for }x>0.\]

\begin{lemma}
\label{lemma32}\cite[Lemma 3.2]{CL_multidim}
Let $v\geq0$ be defined on $\Omega$. Then for any $p\geq1$, $\ep>0$ and $0\leq\al<1$, there is a
constant $C=C(p,\al,\ep,\Omega)>0$ such that
\[\|v\|_p\leq\ep\|v(\log^+v)^{1-\al}\|_p+C.\]
\end{lemma}
\begin{lemma}
\label{lemma33}\cite[Lemma 3.3]{CL_multidim}
Let $m\geq1$, $q>r\geq1$ and define
\[\gamma=\frac{\frac1r-\frac1q}{\frac1n-\frac1m+\frac1r}.\]
Suppose that $0<\ga<1$. Then there exists $\al$ with $\ga<\al<1$ and a constant
$C=C(m,n,q,r,\Omega)$ such that for all $v\in W^{1,m}(\Omega)$ with $v\geq0$
\[\|v(\log^+v)^{1-\al}\|_q\leq C\|v\log^+v\|_r^{1-\al}(\|\nabla
v\|_m^\ga\|v\|_r^{\al-\ga}+\|v\|_1^\al).\]
\end{lemma}
This result is contained in Lemma~3.3 of \cite{CL_multidim}. Note that 
the restriction $r>1$ assumed there is not needed in the proof.
\begin{lemma}
\label{lemma34}
Let $\Omega\subset\R^n$ be bounded with smooth boundary. Let $1\leq q<\infty$
if $n=2$ and $1\leq q<2n/(n-2)$ if $n\ge 3$. Then there is a
constant $C=C(n,q,\Omega)$ such that
\[\|v\|_q\leq C(\|\nabla v\|_2^\al\|v\|_1^{1-\al}+\|v\|_1),\]
for any $v\in W^{1,2}(\Omega)$, where
\[\al=\frac{1-\frac1q}{\frac1n+\frac12}\]
and $0\leq\al<1$.
\end{lemma}
This is a direct consequence of the Gagliardo-Nirenberg inequality without boundary conditions, for a proof see e.g.\ \cite{LSU}.


For convenience we also state the Poincar\'e inequality in a form which will be used frequently.
\begin{lemma}
\label{lem_poinc}\cite[Chapter 5.8, Theorem 1]{Evans} Let $1\leq p\leq\infty$ and $\Omega\subset\R^n$ be a
bounded Lipschitz domain. Then there is a constant $C=C(p,n,\Omega)>0$ such that for any $v\in
W^{1,p}(\Omega)$
\[\|v\|_p\leq C\|\nabla v\|_p+\|v\|_1.\]
\end{lemma}

We finally recall a ''uniform version'' of Gronwall's inequality.
\begin{lemma}
\label{lem_gron}\cite[Chapter III, Lemma 1.1]{TemamDyn} Let $f,g,h$ be three nonnegative locally integrable
functions on $(t_0,\infty)$ such that $f'(t)$ is locally integrable on $(t_0,\infty)$,
and which satisfy
\[f'(t)\leq g(t)f(t)+h(t),\quad t\geq t_0,\]
\[\int_t^{t+r}f(s)ds\leq a_1,\quad\int_t^{t+r}g(s)ds\leq a_2,\quad \int_t^{t+r}h(s)ds\leq a_3,\quad
t\geq t_0,\]
where $r,a_i$ are positive constants. Then
\[f(t+r)\leq\left(\frac{a_1}r+a_3\right)\exp(a_2),\quad t\geq t_0.\]
\end{lemma}

For the case $p=2$ the initial data $(u^0,c^0)$ has to satisfy $(u^0,c^0)\in
W^{1,2}_{0,\sigma}(\Omega)\times W^{1,2}(\Omega)$ for the local strong solution $(u,c)$ to be well-defined by virtue of Theorem~\ref{theorem_strong_loc}.
For simplicity we impose the following
additional initial regularity condition on the concentrations for the rest of this article:
\[c_i^0\in W^{1,2}(\Omega)\cap L^\infty(\Omega).\]
Note that the $L^\infty$-constraint is not a restriction from the physical point of view, since concentrations of dissolved species are naturally finite. From the mathematical point of view, it is still possible to show that all results concerning global well-posedness and asymptotics in two dimensions presented in this work essentially remain valid, if we allow for initial data $(u^0,c^0)\in L^2_\sigma(\Omega)\times L^2(\Omega)$ with $c_i^0\geq0$. This will be explained in Remark~\ref{rem_weak_sol_2}.

\begin{definition}[Global solution]
\label{def_glob_well}
Let $\Omega\subset\R^n$, $n\geq2$. The function $(u,c)\colon[0,\infty)\to L^2_\sigma(\Omega)\times
L^2(\Omega)$ with $c_i\geq0$ is called \textit{global solution} to ($P$), if it satisfies ($P$) almost everywhere and if
for all $T<\infty$ it holds true that
\[(u,c)\in W^{1,2}(0,T;L^2_\sigma(\Omega)\times L^2(\Omega))\cap L^2(0,T;\DD(A_S)\times
W^{2,2}(\Omega)).\]
\end{definition}

Let from now on $\Omega\subset\R^2$ be bounded and smooth. For the proof of Theorem~\ref{theorem_glob_ex} we aim for an estimate of type
\[\|(u,c)\|_{\E_{T,2}}\leq C(1+T)\]
with $\E_{T,2}$ the space of solutions defined right before (\ref{eq_sol_op}).
This requires appropriate estimates of the local 
solution $(u,c)$. Applying energy methods
we will derive several uniform-in-time bounds for quantities like $\|u\|_2$,
$\|c\|_2$, etc., e.g. $\|u(t)\|_2\leq C$ for $t\geq0$ and $C>0$ not depending on $t$. A
justification for the use of energy methods, i.e.\ the formal multiplication of certain equations of
($P$) with quantities like $u,c$ and integration over $\Omega$, can be given by the fact that
$(u,c)$ satisfies ($P$) almost everywhere and the integrals in the calculations being finite. Note
at this stage that since we deal with strong solutions these boundedness relations are merely valid
in an almost everywhere sense. However, as we aim at controlling the norm in $\E_{T,2}$, this is
sufficient for our purpose.

Let from now on $(u,c)\in\E_{T,2}$ denote the unique local strong solution from Theorem
\ref{theorem_strong_loc} for initial data
\[(u^0,c^0)\in W^{1,2}_{0,\sigma}(\Omega)\times (W^{1,2}(\Omega)\cap L^\infty(\Omega)).\]

For $(u,c)\in L^2_\sigma(\Omega)\times L^2(\Omega)$ with $c_i\geq0$ we define the energy functional
\begin{equation}
\label{eq_energy_fct}
E(u,c):=\frac12\|u\|_2^2+\sum_{i=1}^N\int_\Omega c_i\log
c_idx+\frac\ep2\|\nabla\phi\|_2^2+\frac{\ep\tau}2\|\phi\|_{2,\pa\Omega}^2,
\end{equation}
where $\phi$ is the solution to
\begin{equation}
\label{eq_energy_phi}
-\ep\Delta\phi=\sum_{j=1}^Nz_jc_j,\quad\pa_\nu\phi+\tau\phi=\xi.
\end{equation}
Recall that $\xi$ is assumed to be a time-independent function on the surface $\pa\Omega$. Note also that
the term $\int_\Omega c_i\log c_idx$ is finite for $c_i\in L^2(\Omega)$ with $c_i\geq0$, since
\begin{equation}
\label{eq_xlogx}
x\log x\leq\begin{cases}
	0,&x\in(0,1),\\
	x^2,&x>1.
\end{cases}
\end{equation}
In the following we set $(x\log x)_{|x=0}:=0$.

Let $T_{max}$ be the maximal time of existence of solutions $(u,c)$. For $t\in[0,T_{max})$ we
define
\begin{equation}
\label{eq_lyapunov}
\begin{split}
V(t):=&E(u(t),c(t))\\
=&\frac1{2}\|u(t)\|_2^2+\sum_{i=1}^N\int_\Omega c_i(t)\log
c_i(t)dx+\frac\ep2\|\nabla\phi(t)\|_2^2+\frac{\ep\tau}2\|\phi(t)\|_{2,\pa\Omega}^2.
\end{split}
\end{equation}
The term $\int_\Omega c_i\log c_idx$ is bounded from below, since $x\log x\geq-1$, say, for $x\geq0$. As
the remaining terms in $V$ are nonnegative, $V$ is bounded from below.

\begin{lemma}
\label{lem_lya_der}
Suppose $(u,c)$ is the local strong solution to ($P$) from Theorem~\ref{theorem_strong_loc}.

(a) We have 
\[t\mapsto \int_\Omega\frac{1}{D_ic_i(t)}|j_i(t)|^2dx\in L^1_{loc}([0,T_{max})),\]
where $j_i=-D_i\nabla c_i-D_iz_ic_i\nabla\phi$ denotes the flux of species $i$.

(b) The function $V$ from (\ref{eq_lyapunov}) is a Lyapunov functional for system ($P$). More
precisely, it holds true that
\begin{equation}
\label{eq_lya_der}
\Dt V(t)=-\|\nabla u(t)\|_2^2-\sum_{i=1}^N\int_\Omega\frac{1}{D_ic_i(t)}|j_i(t)|^2dx\leq0
\end{equation}
for almost all $t\in(0,T_{max})$.
\end{lemma}

In the following proof we suppress all arguments of functions under integrals for better readability, whenever this does not lead to misunderstandings.

\begin{proof}[Proof. (of Lemma~\ref{lem_lya_der})]
In view of the derivation of (\ref{eq_lya_der}) and the mere nonnegativity of $c_i$ we have to first
give a meaning to the term $\int_\Omega\frac{1}{D_ic_i}|j_i|^2dx$, since it might not be well-defined. As will
become clear in the proof below, this integral is connected to the term $\Dt\int_\Omega c_i\log
c_idx$, if it exists. We cannot just interchange differentiation and integration and then
differentiate under the integral, since the derivative of the function $x\mapsto x\log x$ tends to
$\infty$ for $x\to0$. So it is not obvious that the term $\int_\Omega c_i\log c_idx$ is
differentiable with respect to time at all, since we only know $c_i\geq0$. This is the reason why we
first consider an approximation of this term. For $1\leq i\leq N$ and $\de>0$ let
\[g_{i,\de}(t):=\int_\Omega(c_i+\de)\log(c_i+\de)dx,\quad t\in[0,T_{max}).\]
Clearly $g_{i,\de}$ is differentiable with respect to time and we have
\begin{equation}
\label{eq_gide}
\Dt g_{i,\de}(t)=\int_\Omega \pa_tc_i\log(c_i+\de)dx+\int_\Omega \pa_tc_idx.
\end{equation}
The last term is zero by (\ref{eq_mass_cons}); with (\ref{eq_NP}), (\ref{eq_NP_bdy}), integration
by parts, the notation in (\ref{eq_def_J}) and straight forward manipulations we obtain
\begin{equation}
\label{eq_diff_gide}
\begin{split}
\Dt g_{i,\de}(t)&=-\int_\Omega (\divv J_i)\log(c_i+\de)dx=\int_\Omega J_i\cdot\nabla\log(c_i+\de)dx\\
	&=\int_\Omega\frac1{c_i+\de}\nabla c_i\cdot J_idx\\
&=\int_\Omega\frac{1}{D_i(c_i +\de)}(D_i\nabla c_i +D_iz_ic_i \nabla\phi )\cdot(-D_i\nabla c_i
-D_iz_ic_i \nabla\phi )dx\\
&\qquad+ z_i\int_\Omega\frac{c_i}{c_i+\de}\nabla\phi \cdot(D_i\nabla c_i + D_iz_ic_i \nabla\phi
-c_iu )dx\\
	&\qquad+\int_\Omega\frac{c_i}{c_i+\de} u \cdot( z_ic_i\nabla\phi +\nabla c_i )dx\\
&=-\int_\Omega\frac{1}{D_i(c_i+\de)}|j_i|^2dx-z_i\int_\Omega\frac{c_i}{c_i+\de}\nabla\phi\cdot
J_idx\\
&\quad+z_i\int_\Omega\frac{c_i}{c_i+\de}c_iu\cdot\nabla\phi dx+\int_\Omega\frac{c_i}{c_i+\de}\nabla
c_i\cdot udx,\quad t\in(0,T_{max}).
\end{split}
\end{equation}
Integrating equation (\ref{eq_diff_gide}) form $0$ to $t\in(0,T_{max})$ yields, after rearrangement,
\begin{equation}
\label{ez_int_diff_gide}
\begin{split}
\int_0^t\int_\Omega\frac{1}{D_i(c_i+\de)}&|j_i|^2dx=g_{i,\de}(0)-g_{i,\de}(t)-z_i\int_0^t\int_\Omega\frac{c_i}{c_i+\de}\nabla\phi\cdot J_idxds\\
	&+z_i\int_0^t\int_\Omega\frac{c_i}{c_i+\de}c_iu\cdot\nabla\phi dxds+\int_0^t\int_\Omega\frac{c_i}{c_i+\de}\nabla c_i\cdot udxds.
\end{split}
\end{equation}
From this equation we are going to conclude assertion (a). The monotone convergence theorem implies that
\[\int_0^t\int_\Omega\frac{1}{D_ic_i}|j_i|^2dx=\lim_{\de\to0}\int_0^t\int_\Omega\frac{1}{D_i(c_i+\de)}|j_i|^2dx,\]
however this limit might be infinite. To exclude this case we consider $\lim_{\de\to\infty}$ of the
right-hand side of (\ref{ez_int_diff_gide}). Relation (\ref{eq_xlogx}) and the fact that $x\log x\geq-1$, say, implies
\[|(c_i+\de)\log(c_i+\de)|\leq1+(c_i+\de)^2,\]
so there is an integrable majorant due to $c_i(t)\in L^2(\Omega)$, thus by continuity of $x\mapsto x\log x$, $x\geq0$,
\[\lim_{\de\to0}g_{i,\de}(t)=\int_\Omega c_i\log c_idx, \quad t\in[0,T_{max}).\]
For the remaining terms, note that $c_i/(c_i+\de)\leq1$, so by
\begin{align*}
\int_0^t\int_\Omega |\nabla\phi\cdot\nabla c_i|dxds&\leq\|\nabla\phi\|_{L^2(0,t;L^2)}\|\nabla
c_i\|_{L^2(0,t;L^2)}<\infty,\\
\int_0^t\int_\Omega
\bigg|c_i|\nabla\phi|^2\bigg|dxds&\leq\|c_i\|_{L^2(0,t;L^2)}\|\nabla\phi\|_{L^4(0,t;L^4)}<\infty,\\
\int_0^t\int_\Omega
|c_iu\cdot\nabla\phi|dxds&\leq\|c_i\|_{L^2(0,t;L^2)}\|u\|_{L^4(0,t;L^4)}\|\nabla\phi\|_{L^4(0,t;L^4)}<\infty,\\
\int_0^t\int_\Omega |\nabla c_i\cdot u|dxds&\leq\|\nabla c_i\|_{L^2(0,t;L^2)}\|u\|_{L^2(0,t;L^2)}<\infty,
\end{align*}
the premises of the dominated convergence theorem are satisfied. Because of $c_i/(c_i+\de)\to1$ as $\de\to0$ taking the limit on both sides of equation
(\ref{ez_int_diff_gide}) gives
\begin{equation}
\label{eq_diff_L1}
\begin{split}
\int_0^t\int_\Omega\frac{1}{D_ic_i}|j_i|^2dx&=\int_\Omega c_i^0\log c_i^0dx-\int_\Omega c_i(t)\log
c_i(t)dx\\
&-z_i\int_0^t\int_\Omega\nabla\phi\cdot J_idxds+z_i\int_0^t\int_\Omega c_iu\cdot\nabla\phi
dxds<\infty,
\end{split}
\end{equation}
where we used the fact that due to $u\in L^2_\sigma(\Omega)$ the term $\int_\Omega \nabla c_i\cdot udx=0$. Claim (a) is proven.

Relation (\ref{eq_diff_L1}) implies that $t\mapsto\int_\Omega c_i\log c_idx$ is absolutely
continuous and its a.e.\ derivative is given by
\begin{equation}
\label{eq_diff_gi}
\Dt\int_\Omega c_i\log c_idx=-\int_\Omega\frac{1}{D_ic_i}|j_i|^2dx-z_i\int_\Omega\nabla\phi\cdot
J_idx+z_i\int_\Omega c_iu\cdot\nabla\phi dx
\end{equation}
for almost all $t\in(0,T_{max})$; note that the right-hand side of (\ref{eq_diff_gi}) is in
$L^1_{loc}(0,T_{max})$ due to (\ref{eq_diff_L1}).

In two dimensions the energy equality for weak solutions of the Navier-Stokes equations is valid, cf.\ \cite[Theorem V.1.4.2]{Sohr}. Since strong solutions are also weak solutions, we have
\begin{equation}
\label{eq_abl_u}
\frac1{2}\Dt\|u\|_2^2=-\|\nabla u\|_2^2- \sum_{i=1}^Nz_i\int_\Omega c_i\nabla\phi\cdot udx, \quad
t\in(0,T_{max}).
\end{equation}
Recalling the definition of $V$ in (\ref{eq_lyapunov}) with the help of summation of
(\ref{eq_diff_gi}) over $i$, (\ref{eq_abl_u}) and interchanging differentiation and integration, we
obtain
\begin{equation}
\label{eq_V_der_1}
\begin{split}
\Dt V(t)&=-\|\nabla
u\|_2^2-\sum_{i=1}^N\int_\Omega\frac{1}{D_ic_i}|j_i|^2dx-\sum_{i=1}^Nz_i\int_\Omega\nabla\phi\cdot
J_idx\\
	&\quad+\ep\int_\Omega\nabla\phi\cdot\nabla\phi_tdx+\ep\tau\int_{\pa\Omega}\phi\phi_tdx
\end{split}
\end{equation}
for almost all $t\in(0,T_{max})$, where we write $\phi_t:=\pa_t\phi$. Note that the respective last terms of (\ref{eq_diff_gi}) and
(\ref{eq_abl_u}) cancel. Observe also that due to the time independence of $\xi$ differentiating the
boundary condition of (\ref{eq_energy_phi}) in time yields $\pa_\nu\phi_t+\tau\phi_t=0$. Hence using
integration by parts, taking into account (\ref{eq_NP}), (\ref{eq_NP_bdy}) and the time-independence of $\xi$ gives
\begin{equation}
\label{eq_lya_1}
\begin{split}
\sum_{i=1}^Nz_i\int_\Omega\nabla\phi\cdot J_idx&=\int_\Omega\phi \pa_t\left(\sum_{i=1}^Nz_ic_i
\right)dx=-\ep\int_\Omega\phi \Delta\phi_t dx\\
&=\ep\int_\Omega\nabla\phi\cdot \nabla\phi_t dx-\ep\int_{\pa\Omega}\phi \pa_\nu\phi_t  dx\\
&=\ep\int_\Omega\nabla\phi\cdot \nabla\phi_t dx+\ep\tau\int_{\pa\Omega}\phi \phi_t dx.
\end{split}
\end{equation}
Thus, collecting (\ref{eq_V_der_1}) and (\ref{eq_lya_1}), we see
\[\Dt V(t)=-\|\nabla u\|_2^2-\sum_{i=1}^N\int_\Omega\frac{1}{D_ic_i}|j_i|^2dx\leq0\]
for almost all $t\in(0,T_{max})$.
\end{proof}

\begin{remark}
\label{rem_der_lya}
Observe that due to $V\in W_{loc}^{1,1}(0,T_{max})$ the function $V$ is absolutely
continuous, so for almost all $t\in(0,T_{max})$ it holds true that
\begin{equation}
\label{eq_V_abs_cont}
V(t)=V(0)-\int_0^t\left(\|\nabla
u(s)\|_2^2+\sum_{i=1}^N\int_\Omega\frac1{D_ic_i(s)}|j_i(s)|^2dx\right)ds.
\end{equation}
Hence $V$ is non-increasing.
\end{remark}

\begin{lemma}
\label{lem_u_2}
There is a constant $C>0$, depending only on the initial data, such that for almost all
$t\in(0,T_{max})$
\[\|u(t)\|_2+\|c_i(t)\log c_i(t)\|_1\leq C.\]
\end{lemma}

\begin{proof}
From (\ref{eq_V_abs_cont}) we have $V(t)\leq V(0)$, hence
\[\|u(t)\|_2^2+\int_\Omega c_i(t)\log c_i(t)dx\leq C_0,\]
where $C_0>0$ only depends on the initial data. From $c_i\log c_i\geq-1$ we deduce
\[\int_\Omega|c_i(t)\log c_i(t)|dx\leq\int_\Omega c_i\log c_idx+2 |\Omega|\leq C'_0,\]
which shows the claim.
\end{proof}

With Lemma~\ref{lemma33} the boundedness of the term $\|c_i\log c_i\|_1$ enables us to show
that the $L^2(\Omega)$-norm of concentrations $c_i$ remains uniformly bounded as long as the solution
exists.

\begin{lemma}
\label{lem_ci_l2}
There is a constant $C>0$, depending only on the initial data, such that for almost all
$t\in(0,T_{max})$
\[\|c_i(t)\|_2\leq C\quad \ \text{and }\quad\int_0^t\|\nabla c_i(s)\|_2^2ds\leq C(1+t),\quad
i=1,\ldots ,N.\]
\end{lemma}

\begin{proof}
We multiply equation (\ref{eq_NP}) by $c_i$ and integrate over 
$\Omega$ by parts, making use of
(\ref{eq_NP_bdy}), to get
\begin{equation*}
\begin{split}
\frac12\Dt\|c_i\|_2^2+D_i&\|\nabla c_i\|_2^2=- D_iz_i\int_\Omega c_i\nabla\phi\cdot\nabla c_i
dx+\int_\Omega c_iu\cdot\nabla c_idx.
\end{split}
\end{equation*}
Since
\[\int_\Omega c_iu\cdot\nabla c_idx=\frac12\int_\Omega u\cdot\nabla(c_i^2)dx=0\]
due to $u\in L^2_\sigma(\Omega)$ and $\nabla(c_i^2)\in L^2(\Omega)$, we are left with
\begin{equation}
\label{eq_l2_est}
\frac12\Dt\|c_i\|_2^2+D_i\|\nabla c_i\|_2^2\leq C\int_\Omega|c_i\nabla\phi\cdot\nabla c_i|dx.
\end{equation}
In order to use a Gronwall argument we have to estimate $I_i:=\int_\Omega|c_i\nabla\phi\cdot\nabla
c_i|dx$ appropriately. For convenience we reproduce the computation from \cite{CL_multidim} correcting also a slight mistake. We use the same splitting into $\phi_1$ and $\phi_2$ for the potential $\phi$ as in
(\ref{eq_phi_1}) and (\ref{eq_phi_2}). Let $1\leq q,r<\infty$ be such that $1/2=1/q+1/r$. Then
\begin{equation}
\label{ez_ii_1}
I_i\leq\|\nabla c_i\|_2\|c_i\|_r\|\nabla\phi\|_q\leq\|\nabla
c_i\|_2\|c_i\|_r(\|\nabla\phi_1\|_q+\|\nabla\phi_2\|_q).
\end{equation}

With $W^{2,2}(\Omega)\inj W^{1,q}(\Omega)$ for any $q\in[1,\infty)$, since $\Omega\subset\R^2$, due to
Sobolev's embedding theorem and the help of Remark~\ref{rem_triebel} we compute
\begin{equation}
\label{eq_l2_phi_1}
\|\nabla\phi_1\|_q\leq\|\phi_1\|_{W^{1,q}}\leq C\|\phi_1\|_{W^{2,2}}\leq
C'\|\xi\|_{W^{1/2}_2(\pa\Omega)},
\end{equation}
hence $\|\nabla\phi_1\|_q\leq C$ independent of time. Now let $\de\in(1,2)$ and $q:=\frac
{2\de}{2-\de}$, then by Sobolev's embedding theorem we have $W^{2,\de}(\Omega)\inj W^{1,q}(\Omega)$.
So again taking advantage of Remark~\ref{rem_triebel} we have
\begin{equation}
\label{eq_l2_phi_2}
\|\nabla\phi_2\|_q\leq\|\phi_2\|_{W^{1,q}}\leq C\|\phi_2\|_{W^{2,\de}}\leq
C'\sum_{j=1}^N\|c_j\|_\de.
\end{equation}
Combining (\ref{ez_ii_1})-(\ref{eq_l2_phi_2}) yields
\begin{equation}
I_i\leq C\|\nabla c_i\|_2\|c_i\|_r\left(1+\sum_{j=1}^N\|c_j\|_\de\right),
\end{equation}
where the constant $C$ does not depend on $t$ or $c_i$.

Applying Lemma~\ref{lemma32} and Lemma~\ref{lemma33} and the fact that $\|c_i\|_1\equiv m_i$ and
$\|c_i\log c_i\|_1\leq C$ due to Corollary~\ref{cor_mass_cons} and Lemma~\ref{lem_u_2}, we see that
for any $p\geq1$ and $\ep>0$ there is a constant $C>0$ such that
\[\|c_i\|_p\leq\ep\|\nabla c_i\|_2^{1-1/p}+C,\]
where $C$ depends on $m_i$ and the parameters indicated in Lemmas~\ref{lemma32} and \ref{lemma33}.
Thus we may estimate $I_i$ by
\begin{equation}
\label{eq_l2_Ii_2}
\begin{split}
I_i&\leq C\|\nabla c_i\|_2(\ep\|\nabla c_i\|_2^{1-1/r}+C)\left(\ep\|\sum_{j=1}^N\|\nabla
c_j\|_2^{1-1/\de}+C\right)\\
&\leq C\left(\ep \|\nabla c_i\|_2^{2-1/r}\bigg(1+\sum_{j=1}^N\|\nabla c_j\|_2^{1-1/\de}\bigg)+\ep \sum_{j=1}^N\|\nabla c_j\|_2^{1-1/\de}+1\right)
\end{split}
\end{equation}
for $\ep<1$. With the help of
\begin{equation*}
\begin{split}
\sum_{i=1}^N\|\nabla c_i\|_2^{2-1/r}\sum_{j=1}^N\|\nabla
c_j\|_2^{1-1/\de}&\leq C\max_{1\leq i\leq N}\|\nabla c_i\|_2^{3-1/r-1/\de}\\
	&\leq C\sum_{i=1}^N\|\nabla c_i\|_2^{3-1/r-1/\de}
\end{split}
\end{equation*}
and Young's inequality we obtain
\begin{equation}
\label{eq_l2_Ii_sum}
\begin{split}
\sum_{i=1}^NI_i\leq C\bigg(\ep\sum_{i=1}^N(\|\nabla c_i\|_2^{3-1/r-1/\de}+\|\nabla c_i\|_2^2)+1\bigg).
\end{split}
\end{equation}
From $n=2$ and the definition of $r$ and $\de$ we see $1/r+1/\de=1$. So, writing $d':=\min\{D_i\}$
and choosing $\ep>0$ small enough, we infer from (\ref{eq_l2_Ii_sum})
\begin{equation}
\label{eq_l2_Ii_sum_2}
\sum_{i=1}^NI_i\leq\frac{d'}2\sum_{i=1}^N\|\nabla c_i\|_2^2+C
\end{equation}
Summing (\ref{eq_l2_est}) over $i$ and using estimate (\ref{eq_l2_Ii_sum_2}) yields
\begin{align}
\label{eq_nabla_ci_int}
\Dt \sum_{i=1}^N\|c_i \|_2^2&\leq-d'\sum_{i=1}^N\|\nabla c_i \|_2^2+C.
\end{align}
Poincar\'e's inequality (cf.\ Lemma~\ref{lem_poinc}) implies that there is a constant $C>0$ such
that
\[-\|\nabla c_i\|_2\leq -C\|c_i\|_2+C\|c_i\|_1\leq-C\|c_i\|_2+C'.\]
Plugging this into (\ref{eq_nabla_ci_int}) results in
\begin{align*}
\Dt \sum_{i=1}^N\|c_i \|_2^2&\leq-d''\sum_{i=1}^N\|c_i \|_2^2+C'.
\end{align*}
With Gronwall's inequality we deduce that there is a $C>0$ such that
\[\|c_i(t)\|_2\leq C,\quad t\in[0,T_{max}).\]
Integrating (\ref{eq_nabla_ci_int}) from $0$ to $t$ gives
\[\int_0^t\|\nabla c_i(s)\|_2^2ds\leq C'(1+t).\]
\end{proof}

\begin{remark}
\label{rem_gron_1}
If we integrate (\ref{eq_nabla_ci_int}) only over $(t,t+r)$ for $t<T_{max}$ and $r>0$ small enough
such that the integration makes sense, we obtain
\begin{equation}
\label{eq_nabla_ci_growth}
\int_t^{t+r}\|\nabla c_i(s)\|_2^2ds\leq C(r)
\end{equation}
independently of $t$. This will prove important for the application of the uniform Gronwall
inequality in Lemma~\ref{lem_gron}.
\end{remark}

\begin{remark}
\label{rem_3d}
In three dimensions it is also possible to estimate the term $I_i$ along the lines of the proof of
Lemma~\ref{lem_ci_l2} using Corollary \ref{cor_mass_cons} and Lemmas~\ref{lemma32}, \ref{lemma33}
and \ref{lem_u_2}. However performing the same calculations one observes that in this case
\[\frac1r+\frac1\de<1,\]
so $I_i$ cannot be absorbed into $\|\nabla c_i\|_2^2$, cf.\ \cite{CL_multidim}.
\end{remark}

\begin{lemma}
\label{lem_nabla_phi_infty}
There is a constant $C>0$, depending only on the initial data, such that for almost all
$t\in(0,T_{max})$
\[\|c_i(t)\|_4+\|\phi(t)\|_{W^{1,\infty}}\leq C.\]
\end{lemma}

\begin{proof}
Note that once the uniform boundedness of $\|c_i\|_4$ for $t\geq0$ is obtained, by Sobolev's embedding
theorem and Remark~\ref{rem_triebel} it follows that
\[\|\phi\|_{W^{1,\infty}}\leq C\|\phi\|_{W^{2,4}}\leq C'\left(\sum_j\|c_j\|_4+1\right)\leq C''\quad
(t\in(0,T_{max})).\]
Let $k\in\N$ and $t\geq0$. We multiply (\ref{eq_NP}) by $c_i^{2k-1}$, integrate over $\Omega$ and by
parts and use again the no-flux condition (\ref{eq_NP_bdy}) to obtain
\begin{equation}
\label{eq_l4_1}
\begin{split}
&\frac1{2k}\Dt\|c_i^k\|_2^2=\int_\Omega J_i\nabla(c_i^{2k-1})dx\\
&=(2k-1)\int_\Omega \left(-D_ic_i^{2k-2}|\nabla c_i|^2-D_iz_i c_i^{2k-1}\nabla
c_i\nabla\phi+c_i^{2k-1}u\nabla c_i\right)dx\\
&=-D_i\frac{(2k-1)}{k^2}\int_\Omega|\nabla(c_i^k)|^2dx- D_iz_i\frac{2k-1}k\int_\Omega
c_i^k\nabla(c_i^k)\nabla\phi dx\\
	&\quad+\frac{2k-1}{2k}\int_\Omega u\nabla (c_i^{2k})dx.
\end{split}
\end{equation}
To make this formal calculation rigorous, it is important to note that $(u,c)$ satisfies problem ($P$) almost everywhere and that $c_i\in
W^{1/2}_2(0,T;W^{1,2}(\Omega)\inj L^q(0,T;L^q(\Omega))$ for any $q<\infty$, thus
$\|c_i^k\|_2=\|c_i\|_{2k}^{k}$ is well-defined for almost all $t\in(0,T_{max})$.

Due to the divergence free condition on $u$ the term $\int_\Omega u\cdot\nabla (c_i^{2k})dx$ is zero. If
we denote
\begin{equation}
\label{eq_def_vk}
v_k:=c_i^k,
\end{equation}
we can rewrite (\ref{eq_l4_1}) as
\begin{align}
\label{eq_hoe_v}
\frac1{2k}\Dt\|v_k\|_2^2=-D_i\frac{2k-1}{k^2}\|\nabla v_k\|_2^2- D_iz_i\frac{2k-1}{k}\int_\Omega
v_k\nabla v_k\cdot\nabla\phi dx.
\end{align}
The H\"older inequality yields
\begin{equation}
\label{eq_hoe}
\frac1{2k}\Dt \|v_k\|_2^2\leq-D_i\frac{2k-1}{k^2}\|\nabla v_k\|_2^2+
C\frac{2k-1}{k}\|v_k\|_3\|\nabla v_k\|_2\|\nabla\phi\|_6.
\end{equation}
The uniform boundedness of $\|c_i\|_2$ due to Lemma~\ref{lem_ci_l2} and the Sobolev embedding
$W^{2,2}(\Omega)\inj W^{1,6}(\Omega)$ imply
\[ \|\nabla\phi\|_6\leq\|\phi\|_{W^{2,2}}\leq C\left(\sum_j\|c_j\|_2+1\right)\leq C'.\]
Applying Lemma~\ref{lemma34} to $\|v_k\|_3$ shows that
\[\|v_k\|_3\leq C(\|\nabla v_k\|_2^{2/3}\|v_k\|_1^{1/3}+\|v_k\|_1),\]
where $C>0$ is independent of $v_k$ and $t$. Hence we estimate (\ref{eq_hoe}) further
by\begin{equation*}
\begin{split}
\frac1{2k}\Dt& \|v_k\|_2^2\leq-\frac{D_i}k\|\nabla v_k\|_2^2+C\|\nabla v_k\|_2(\|\nabla
v_k\|_2^{2/3}\|v_k\|_1^{1/3}+\|v_k\|_1).
\end{split}
\end{equation*}
To obtain the lemma we set $k=2$; note that $\|v_2(t)\|_1=\|c_i(t)\|_2^2\leq C$, so
\begin{equation*}
\begin{split}
\frac1{4}\Dt \|v_2 \|_2^2&\leq-\frac{D_i}2\|\nabla v_2 \|_2^2+C'\|\nabla v_2 \|_2(\|\nabla v_2
\|_2^{2/3}+1)\\
	&\leq-\frac{D_i}4\|\nabla v_2 \|_2^2+C'',
\end{split}
\end{equation*}
where we made use of Young's inequality in the last step. Applying
Poincar\'e's inequality (Lemma~\ref{lem_poinc}) and taking into account
the boundedness of $\|v_2\|_1$ yields constants $d',C>0$, such that
\[ \frac14\Dt \|v_2 \|_2^2\leq-d'\|v_2 \|_2^2+C.\]
Finally, from Gronwall's inequality $\|v_2 \|_2=\|c_i \|_4^2$ is uniformly bounded for $t\geq0$,
since $\|c_i(0)\|_4=\|c_i^0\|_4<\infty$.
\end{proof}

We are now in position to obtain suitable estimates on the velocity field~$u$.

\begin{lemma}
\label{lem_nabla_u_l2}
There is a constant $C>0$, depending only on the initial data, such that for almost all
$t\in(0,T_{max})$
\[\int_0^t\|\nabla u(s)\|_2^2ds\leq C(1+t).\]
\end{lemma}

\begin{proof}
Recall the energy equality from (\ref{eq_abl_u}), i.e.\
\begin{equation}
\label{eq_energy_est}
\begin{split}
\frac12\Dt \|u \|_2^2&+\|\nabla u \|_2^2=-\int_\Omega\sum_{i=1}^Nz_ic_i \nabla\phi \cdot u dx.
\end{split}
\end{equation}
The right-hand side can be estimated with the uniform boundedness of $\|u\|_2$, $\|c\|_2$ and
$\|\nabla\phi\|_\infty$, cf.\ Lemmas \ref{lem_u_2}, \ref{lem_ci_l2} and \ref{lem_nabla_phi_infty},
as follows:
\[\left|\int_\Omega c_i \nabla\phi\cdot u dx\right|\leq \|c_i \|_2\|\nabla\phi \|_\infty\|u \|_2\leq
C\quad (t\in(0,T_{max})).\]
So integrating (\ref{eq_energy_est}) from $0$ to $t$ in time gives
\[\frac12\|u(t)\|_2^2+\int_0^t\|\nabla u(s) \|_2^2ds\leq\frac12\|u^0\|_2^2+Ct\leq C'(1+t),\]
which proves the claim.
\end{proof}

\begin{remark}
\label {rem_gron_2}
Note that if we choose $t$ and $r$ in the same way as in Remark~\ref{rem_gron_1} and integrate
(\ref{eq_energy_est}) over $(t,t+r)$, this results in
\begin{equation}
\label{ez_int_A_S_12}
\int_t^{t+r}\|\nabla u(s)\|_2^2ds\leq C'(r),
\end{equation}
where $C'$ only depends on $r$, but not on $t$.
\end{remark}

We also need estimates on higher derivatives of $u$.

\begin{lemma}
\label{lem_delta_u_l2}
There is a constant $C>0$, depending only on the initial data, such that for almost all
$t\in(0,T_{max})$
\begin{equation}
\label{eq_uni_gron}
\|\nabla u(t)\|_2^2\leq C\quad\text{and }\quad\int_0^t\|A_Su(s)\|_2^2ds\leq C(1+t)
\end{equation}
where $A_S$ denotes the Stokes operator defined in Remark~\ref{rem_sohr}.
\end{lemma}

\begin{proof}
Multiplying (\ref{eq_NS}) by $A_Su$ and integrating over $\Omega$ yields
\begin{equation}
\label{eq_nabla_u_est}
\begin{split}
\frac12\Dt&\|\nabla u \|_2^2+\|A_Su \|_2^2=-\int_\Omega(u \cdot\nabla)u \cdot A_Su
dx-\int_\Omega\sum_{i=1}^Nz_ic_i \nabla\phi \cdot (A_Su) dx
\end{split}
\end{equation}
where we used the well-known fact that
\begin{equation}
\label{eq_energy_est_nabla}
\int_\Omega u_t\cdot A_Sudx=-\int_\Omega u_t\Delta udx=\int_\Omega\nabla u_t\colon\nabla
u-\int_{\pa\Omega} u_t\pa_\nu udx=\frac12\Dt\|\nabla u\|_2^2.
\end{equation}

Working in two dimensions, the Gagliardo-Nirenberg inequality implies
\begin{equation}
\label{eq_gag_ge}
\|v\|_4\leq C\|v\|_2^{1/2}\|\nabla v\|_2^{1/2},\quad\text{for }v\in W^{1,2}(\Omega).
\end{equation}
The Stokes operator in $L^2_\sigma(\Omega)$, $\Omega\subset\R^n$ a bounded smooth domain,
is closed with $0\in\rho(A_S)$, cf.\ \cite[Theorem III.2.1.1]{Sohr}.
Hence, we infer $\|\nabla^2 u\|_2\leq \|u\|_{W^{2,2}}\leq C\|A_Su\|_2$ for
$u\in\DD(A_S)$.
This in combination with (\ref{eq_gag_ge}), the uniform boundedness of $\|u\|_2$ by
Lemma~\ref{lem_u_2}, and Young's inequality yield the following estimate on the non-linearity.
\begin{equation}
\label{eq_nabla_u_nlin_1}
\begin{split}
\left|\int_\Omega(u \cdot\nabla)u \cdot A_Su dx\right|&\leq\|u \|_4\|\nabla u \|_4\|A_Su \|_2\\
	&\leq C\|u \|_2^{1/2}\|\nabla u \|_2\|\nabla^2u \|_2^{1/2}\|A_Su \|_2\\
	&\leq \frac14\|A_Su \|_2^2+C'\|\nabla u \|_2^4.
\end{split}
\end{equation}
Using the a priori estimates on $c_i$ and $\nabla\phi$ from Lemma~\ref{lem_ci_l2} and
\ref{lem_nabla_phi_infty} and Young's inequality the nonlinear electro-kinetic part can be treated by
\begin{equation}
\label{eq_nabla_u_nlin_2}
\begin{split}
\left|\int_\Omega\sum_{i=1}^Nz_ic_i \nabla\phi \cdot A_Su dx\right|&\leq C\|\nabla\phi \|_\infty\|c
\|_2\|A_Su \|_2\\
	&\leq \frac14\|A_Su \|_2^2+C'.
\end{split}
\end{equation}
So estimating (\ref{eq_nabla_u_est}) with (\ref{eq_nabla_u_nlin_1}) and (\ref{eq_nabla_u_nlin_2})
results in
\begin{equation}
\label{pr_lem_glob}
\Dt\|\nabla u \|_2^2+\|A_Su \|_2^2\leq C\|\nabla u \|_2^4+C.
\end{equation}
Setting $g(t) :=C\|\nabla u \|_2^2$ this implies
\[\Dt\|\nabla u \|_2^2\leq g(t) \|\nabla u \|_2^2+C\quad (t\in(0,T_{max})).\]
If we let $f(t)=\|\nabla u \|_2^2$, and $h(t):=C$, with Remark~\ref {rem_gron_2} all premises for
the uniform Gronwall inequality Lemma~\ref{lem_gron} are fullfilled, and hence
\begin{equation*}
\|\nabla u \|_2^2\leq C'\quad (t\in (0,T_{max})).
\end{equation*}
Integrating (\ref{pr_lem_glob}) over $(0,t)$ therefore yields
\[\int_0^t\|A_Su (s)\|_2^2ds\leq C'(1+t).\]
\end{proof}

The estimation of higher derivatives of $c_i$ turns out to pose certain difficulties since a
calculation analogous to (\ref{eq_energy_est_nabla}), creates, in
general, non-vanishing and unpleasant boundary integrals. Indeed it gives
\begin{align*}
\int_\Omega \pa_tc_i\cdot\Delta c_idx&=-\int_\Omega\nabla (\pa_tc_i)\cdot\nabla c_idx+\int_{\pa\Omega}
(\pa_tc_i)(\pa_\nu c_i)dx\\
	&=-\frac12\Dt\|\nabla c_i\|_2^2+\int_{\pa\Omega} (\pa_tc_i)(\pa_\nu c_i)dx.
 \end{align*}
By this fact, we prefer to introduce new variables instead of 
working with $c_i$ directly.
In \cite{CL_multidim} a problem similar to (\ref{eq_NP})-(\ref{eq_NP_bdy}) is transformed to a
problem for the new variable
\begin{equation}
\label{eq_trans_conc}
\zeta_i:=c_i\exp( z_i\phi).
\end{equation}
We will show that $\zeta$ possesses exactly the regularity as 
desired for $c$ in Theorem~\ref{theorem_strong_loc}, i.e.
\[\zeta\in W^{1,2}(\Omega)(0,T;L^2(\Omega))\cap L^2(0,T;W^{2,2}(\Omega)),\]
cf.\ Lemma~\ref{lem_reg_zeta}. For, we plug $c_i=\zeta_i\exp(-z_i\phi)$ into
(\ref{eq_NP})-(\ref{eq_NP_bdy}) and obtain the following nonlinear heat equation subject to
homogeneous Neumann boundary conditions for $\zeta_i$:
\begin{align}
\label{eq_trans_prob}
\pa_t\zeta_i-D_i\Delta \zeta_i&=- D_iz_i\nabla\phi\cdot\nabla \zeta_i+ z_i\zeta_i\phi_t+
z_i\zeta_iu\cdot\nabla\phi-u\cdot\nabla \zeta_i\quad\text{in }\Omega,\\
\label{eq_trans_prob_bdy}
\pa_\nu \zeta_i&=0\quad\text{on }\pa\Omega,
\end{align}
for $i=1,\ldots,N$. Those homogeneous Neumann boundary conditions enable us to apply the
corresponding relation to (\ref{eq_energy_est_nabla}) for $\zeta$, namely
\begin{equation}
\label{eq_energy_zeta}
\int_\Omega \pa_t\zeta_i\cdot\Delta \zeta_idx=-\int_\Omega\nabla
(\pa_t\zeta_i)\cdot\nabla\zeta_idx+\int_\Omega(\pa_t\zeta_i)(\pa_\nu \zeta_i)dx=-\frac12\Dt\|\nabla \zeta_i\|_2^2.
\end{equation}
Once we have ''good'' estimates for higher derivatives of $\zeta$ we will be able to infer the same
behaviour for higher derivatives of $c$, since $c$ and $\zeta$ are
connected via $\phi$, which already has nice regularity properties.

However, in order to obtain suitable estimates for $\nabla\zeta$ and $\Delta\zeta$ via the energy
method, we need to gather more information on the terms on the right-hand side of
(\ref{eq_trans_prob}). This will be done in two steps.

First we shall prove the existence of a uniform $L^\infty$-bound in space and time for
concentrations $c_i$ (Lemma~\ref{lem_ci_infty}). With this information at hand we gain sufficient
information on the new term $\phi_t$ (Lemma \ref{lem_phi_t_p}).

\begin{lemma}
\label{lem_ci_infty}
There is a constant $C>0$, depending only on the initial data, such that for almost all
$t\in(0,T_{max})$
\[\|c_i(t)\|_\infty\leq C.\]
\end{lemma}

The computation for the proof of Lemma~\ref{lem_ci_infty} is essentially contained in \cite{CL_multidim}, for convenience we provide a proof here correcting a minor mistake. The idea is as follows: If we can show that there is a sequence $p_k\to\infty$ with
the property that $\operatorname{ess\,sup}_{t\in[0,\infty)}\|c_i(t)\|_{p_k}\leq C$ independently of $k$, then not only
is $c_i(t)\in L^\infty(\Omega)$ for almost all $t\in[0,\infty)$, but it also holds that $c_i\in
L^\infty([0,\infty)\times\Omega)$. This method is sometimes referred to as ''Moser iteration'', cf.
\cite{Wu}.

\begin{proof}[Proof. (of Lemma \ref{lem_ci_infty})]
Recall the notation $v_k=c_i^k$ for $k\in\N$. Taking into account the
$L^\infty$-bound for $\nabla\phi$, we improve the estimation of (\ref{eq_hoe_v}) given in (\ref{eq_hoe}) and get
\[\frac1{2k}\Dt \|v_k\|_2^2\leq-\frac{D_i}{k}\|\nabla v_k\|_2^2+C\|v_k\|_2\|\nabla v_k\|_2.\]
Thanks to Lemma~\ref{lemma34} and Young's inequality, for every $\ep>0$ there is a $C(\ep)>0$ such
that
\begin{align*}
\frac1{2k}\Dt \|v_k\|_2^2&-\frac{D_i}{k}\|\nabla v_k\|_2^2+C\|\nabla v_k\|_2(\|\nabla
v_k\|_2^{1/2}\|v_k\|_1^{1/2}+\|v_k\|_1)\\
&\leq-\frac{D_i}{k}\|\nabla v_k\|_2^2+\ep\|\nabla v_k\|_2^2+C(\ep)\|v_k\|_1^2.
\end{align*}
If we choose $\ep=\frac{D_i}{2k}$, we get
\begin{equation}
\label{eq_ci_infty_1}
\begin{split}
\frac1{2k}\Dt \|v_k\|_2^2\leq -\frac{D_i}{2k}\|\nabla v_k\|_2^2+C'(k)\|v_k\|_1^2,
\end{split}
\end{equation}
where $C'=C'(k)$ can be controlled by a polynomial in $k$.
Poincar\'e's inequality (Lemma~\ref{lem_poinc}) implies
\[-\|\nabla v_k\|_2^2\leq-\frac{1}{C_P}\|v_k\|_2^2+\|v_k\|_1^2\]
with a constant $C_P>\max\{1,|\Omega|\}$, without loss of generality. Plugging this into
(\ref{eq_ci_infty_1}) and multiplying with $2k$ gives
\begin{equation}
\label{eq_ci_infty_2}
\Dt \|v_k\|_2^2\leq -\frac{D_i}{C_P}\|v_k\|_2^2+(2kC'(k)+D_i)\|v_k\|_1^2.
\end{equation}
Considering a fixed $t\in(0,T_{max})$, we define
\[S_k:=\max\left\{\|c_i(0)\|_\infty,\operatorname*{ess\,sup}_{0\leq s\leq t}\left(\int_\Omega
c_i^k(s,x)dx\right)^{1/k}\right\}.\]
Note that we already proved that $S_1,S_2,S_4$ can be bounded independently of~$t$.

With the abbreviations $C_1=D_i/C_P$ and $C_2(k)=2kC'(k)+D_i$ we deduce from equation
(\ref{eq_ci_infty_2}) that
\[\Dt \|v_k(t)\|_2^2\leq -C_1\|v_k(t)\|_2^2+C_2(k)S_k^{2k}, \quad 0\leq t\leq T.\]
Multiplying with $e^{C_1t}$ yields
\[
	\Dt e^{C_1t}\|v_k(t)\|_2^2\leq C_2(k)S_k^{2k}e^{C_1t}.
\]
By a standard Gronwall argument we therefore can achieve that
\begin{equation}
\label{eq_ci_infty_3}
\begin{split}
\|v_k(t)\|_2^2&\leq\|v_k(0)\|_2^2+\frac{C_2(k)}{C_1}S_k^{2k}\leq|\Omega|\|c_i^0\|_\infty^{2k}+\frac{C_2(k)S_k^{2k}}{C_1}\\
&\leq\left(|\Omega|+\frac{C_2(k)}{C_1}\right)S_k^{2k}\leq\frac{2C_2(k)}{C_1}S_k^{2k}, \quad 0\leq
t\leq T_{max},
\end{split}
\end{equation}
since we assumed that $C_P>|\Omega|$. Therefore
\begin{equation}
\label{eq_ci_infty_4}
\left(\int_\Omega
c_i(t,x)^{2k}dx\right)^{1/2k}=\left(\|v_k(t)\|_2^2\right)^{1/2k}\leq\left(\frac{2C_2(k)}{C_1}\right)^{1/2k}S_k.
\end{equation}
We can find constants $C,m>0$ such that
\begin{equation}
\label{eq_ci_infty_5}
\frac{2C_2(k)}{C_1}=2C_P\left(1+2k\frac{C'(k)}{D_i}\right)\leq Ck^m,\quad k\geq1.
\end{equation}
Collecting (\ref{eq_ci_infty_3})-(\ref{eq_ci_infty_5}), we obtain the following recursion relation
for $S_k$:
\begin{equation}
\label{eq_ci_infty_rec}
\begin{split}
S_{2k}&\leq\max\bigg\{\|c_i(0)\|_\infty,\bigg(\frac{2C_2(k)}{C_1}\bigg)^{1/2k}S_k\bigg\}=\bigg(\frac{2C_2(k)}{C_1}\bigg)^{1/2k}S_k\\
	&\leq C^{1/2k}k^{m/2k}S_k.
\end{split}
\end{equation}
Thus from the fact that $S_4<\infty$, we deduce that $S_8<\infty$ and so on. Note that this
recursion is valid for all $0<t<T_{max}$, and the constant $C$ does
neither depend on $t$ nor on $k$.

If we set $k=2^\mu$, $\mu=0,1,2\ldots$, relation (\ref{eq_ci_infty_rec}) implies
\[S_{2^{\mu+1}}\leq C^{1/2^{\mu+1}}2^{m\mu/2^{\mu+1}}S_{2^\mu}, \quad\mu=0,1,2,\ldots.\]
Since $S_1<\infty$ due to Corollary~\ref{cor_mass_cons} and the series $\sum2^{-(\mu+1)}$ and
$\sum\mu/2^{\mu+1}$ converge, the sequence $(S_{2^\mu})_\mu$ is bounded. Note that this means that
\[\operatorname*{ess\,sup}_{0\leq t<T_{max}}\|c(t)\|_{2^\mu}\leq S_{2^\mu}\leq C\]
independently of $\mu$. We therefore proved $\operatorname{ess\,sup}_{0\leq
t<T_{max}}\|c_i(t)\|_\infty\leq C$.
\end{proof}

Let us consider $\phi_t=\pa_t\phi$; for almost all $t\in(0,T_{max})$ it holds true that
\begin{alignat}{2}
\label{eq_P_phit}
-\ep\Delta\phi_t&=\sum_{i=1}^Nz_i\pa_tc_i=-\sum_{i=1}^Nz_i\divv J_i,&x\in\Omega,\\
\label{eq_P_bdy_phit}
\pa_\nu\phi_t+\tau\phi_t&=0,&x\in\pa\Omega
\end{alignat}
by interchanging $\pa_t$ and $\Delta$ and taking advantage of (\ref{eq_NP}) and the
time-independence of $\xi$.

\begin{lemma}
\label{lem_phi_t_p}
There is a constant $C>0$, depending only on the initial data, such that for almost all
$t\in(0,T_{max})$
\[\|\phi_t(t)\|_2\leq C.\]
\end{lemma}

\begin{proof}
The proof follows by a duality argument. 
Let 
\[
	Bh:=-\varepsilon\Delta h,
	\quad h\in \mathcal D(B):=\{v\in W^{2,2}(\Omega):\ \pa_\nu v+\tau v=0\}.
\]
Remark~\ref{rem_triebel} then implies that $B:\mathcal D(B)\to L^2(\Omega)$
is isomorphic. By duality also 
$
	B^*:L^2(\Omega)\to \mathcal D(B)^*
$
is isomorphic. By the self-adjointness of $B$ and since $\mathcal D(B)$
is dense in $L^2(\Omega)$ we further can regard above $B^*$ as the unique
extension of $B$, hence we write
\[
	B:L^2(\Omega)\to \mathcal D(B)^*.
\]
According to Theorem~\ref{theorem_strong_loc} we have
\[
	\divv J_i
	\in L^2(\Omega)\hookrightarrow \mathcal D(B)^*.
\]
We now intend to show that the norm of each $\divv J_i$
in $\mathcal D(B)^*$ can be estimated uniformly in $t\in (0,T_{max})$.
Thanks to (\ref{eq_NP_bdy}) we can calculate
\begin{equation}
\label{eq_phit_1}
\begin{split}
\langle \divv J_i, h\rangle_{\mathcal D(B)^*,\mathcal D(B)}
&=\int_\Omega h \divv J_i dx=\int_\Omega \nabla h\cdot J_i dx\\
	&=\int_\Omega\nabla h\cdot(-D_i\nabla c_i -D_i z_ic_i \nabla\phi +c_i u )dx
\end{split}
\end{equation}
for all $h\in \mathcal D(B)$ and for almost all $t\in (0,T_{max})$.

We need to estimate the terms on the right-hand side of (\ref{eq_phit_1}). Applying integration by
parts and H\"older's inequality we compute
\begin{align*}
\left|\int_\Omega\nabla h\cdot\nabla c_i dx\right|&\leq\left|
\int_\Omega c_i\Delta h 
dx\right|+\tau\left|\int_{\pa\Omega}c_i hdx\right|\\
	&\leq\|h\|_{\mathcal D(B)}\|c_i \|_2
	+\tau\|c_i \|_{2,\pa\Omega}\|h\|_{2,\pa\Omega}\\
	&\leq C\left(\|c_i \|_2
	     +\|c_i \|_{\infty}\right)\|h\|_{\mathcal D(B)}\\
	&\leq C\|h\|_{\mathcal D(B)}\quad (t\in (0,T_{max})),
\end{align*}
where we made use of the uniform boundedness of $\|c_i\|_2$ and $\|c_i\|_\infty$, cf.
Lemmas~\ref{lem_ci_l2} and \ref{lem_ci_infty}. With the uniform
bound on $\|\nabla\phi\|_\infty$, cf.\ Lemma~\ref{lem_nabla_phi_infty}, we deduce
\[\left|\int_\Omega c_i\nabla\phi\cdot\nabla hdx\right|\leq\|c_i \|_\infty\|\nabla\phi
\|_\infty\|h\|_{W^{1,1}}\leq C\|h\|_{\mathcal D(B)}\]
for the second term. Due to Lemma~\ref{lem_u_2} the quantity $\|u\|_2$ is uniformly bounded in
time, so the last term can be estimated by
\[\left|\int_\Omega c_iu\cdot\nabla h dx\right|\leq\|c_i \|_\infty\|u \|_2\|\nabla h\|_2\leq
C\|h\|_{\mathcal D(B)}.\]
In total we obtain
\[
	\|\phi_t\|_2\le \sum_{i=1}^N z_i\|B^{-1}\divv J_i\|_2
	\le C\sum_{i=1}^N\|\divv J_i\|_{\mathcal D(B)^*}\le C
	\quad (t\in(0,T_{max})).
\]
\end{proof}

Before treating problem (\ref{eq_trans_prob})-(\ref{eq_trans_prob_bdy}) let us note the following
lemma on the regularity of $\zeta$.

\begin{lemma}
\label{lem_reg_zeta}
Let $c\in\E_{T,2}^c$ be as in Theorem~\ref{theorem_strong_loc} and $\zeta$ be defined as in
(\ref{eq_trans_conc}). Then it holds true that $\zeta\in\E_{T,2}^c$.
\end{lemma}

\begin{proof}
This is an immediate consequence of $\phi$, $\nabla\phi$, $c_i$ and, hence, $\Delta\phi$ being
uniformly bounded in $L^\infty(\Omega)$ and $\phi_t$ being uniformly bounded in $L^2(\Omega)$ for
any time.
\end{proof}

\begin{lemma}
\label{lem_delta_c_l2}
There is a constant $C>0$, depending only on the initial data, such that for almost all
$t\in(0,T_{max})$
\[\|\nabla c_i(t)\|^2_2\leq C\quad\text{and}\quad \int_0^t\|\Delta c_i(s)\|_2^2ds\leq C(1+t).\]
\end{lemma}

\begin{proof}
We first derive the corresponding estimate for the transformed variable $\zeta_i$, defined in
(\ref{eq_trans_conc}).
We multiply (\ref{eq_trans_prob}) with $-\Delta \zeta_i$ and get with (\ref{eq_energy_zeta})
\begin{equation}
\label{eq_nabla_zeta_1}
\begin{split}
\frac12\Dt\|\nabla \zeta_i \|_2^2+D_i\|\Delta \zeta_i \|_2^2&= z_iD_i\int_\Omega\nabla\phi
\cdot\nabla \zeta_i \Delta \zeta_i dx- z_i\int_\Omega \zeta_i \phi_t \Delta \zeta_i dx\\
&\quad- z_i\int_\Omega \zeta_i u \cdot\nabla\phi \Delta \zeta_i dx+\int_\Omega u \cdot\nabla \zeta_i
\Delta \zeta_i dx.
\end{split}
\end{equation}
Note that the $L^\infty$-norm of $\zeta_i$ is bounded independently of time, since both
$\|c\|_\infty$ and $\|\phi\|_\infty$ are bounded independently of time, cf.\ Lemmas
\ref{lem_nabla_phi_infty} and \ref{lem_ci_infty}. Estimating the integrals on the right-hand side,
apart from the last one, with H\"older's inequality, Lemmas \ref{lem_u_2}, \ref{lem_delta_u_l2},
\ref{lem_phi_t_p}, and Young's inequality gives
\begin{equation}
\label{eq_nabla_zeta_2}
\begin{split}
\left|\int_\Omega\nabla\phi \cdot\nabla \zeta_i \Delta \zeta_i dx\right|&\leq\|\nabla\phi
\|_\infty\|\Delta \zeta_i \|_2\|\nabla \zeta_i \|_2\leq\ep_1\|\Delta \zeta_i
\|_2^2+C(\ep_1)\|\nabla\zeta_i \|_2^2,\\
\left|\int_\Omega \zeta_i \phi_t \Delta \zeta_i dx\right|&\leq\|\zeta_i \|_\infty\|\phi_t
\|_2\|\Delta \zeta_i \|_2\leq\ep_2\|\Delta \zeta_i \|_2^2+C(\ep_2),\\
\left|\int_\Omega \zeta_i u \cdot\nabla\phi \Delta \zeta_i dx\right|&\leq\|\zeta_i
\|_\infty\|\nabla\phi \|_\infty\|u \|_2\|\Delta \zeta_i \|_2\leq \ep_3\|\Delta \zeta_i
\|_2^2+C(\ep_3).
\end{split}
\end{equation}
For the remaining integral $\int_\Omega u \cdot\nabla \zeta_i \Delta \zeta_i dx$ recall that if
$v\in W^{2,2}(\Omega)$ satisfies homogeneous Neumann boundary conditions, one has
$\|\nabla^2v\|_2\leq C\|\Delta v\|_2$. So H\"older's inequality and the Gagliardo-Nirenberg inequality (\ref{eq_gag_ge}) give
\begin{equation}
\label{eq_nabla_zeta_3}
\begin{split}
\bigg|\int_\Omega u &\cdot\nabla \zeta_i \Delta \zeta_i dx\bigg|\leq\|u \|_4\|\nabla \zeta_i
\|_4\|\Delta \zeta_i \|_2\\
&\leq C\|u \|_2^{1/2}\|\nabla u \|_2^{1/2}\|\nabla \zeta_i
\|_2^{1/2}\|\nabla^2\zeta_i\|_2^{1/2}\|\Delta \zeta_i \|_2\leq C'\|\nabla \zeta_i
\|_2^{1/2}\|\Delta\zeta_i \|_2^{3/2}\\
	&\leq\ep_4\|\Delta \zeta_i \|_2^2+C(\ep_4)\|\nabla \zeta_i \|_2^2,
\end{split}
\end{equation}
where we also applied Lemmas~\ref{lem_u_2} and \ref{lem_delta_u_l2}.

Collecting (\ref{eq_nabla_zeta_2}) and (\ref{eq_nabla_zeta_3}) and choosing $\ep_i$ small enough we
obtain from (\ref{eq_nabla_zeta_1})
\begin{equation}
\label{eq_nabla_zeta_4}
\Dt\|\nabla \zeta_i \|_2^2+D_i\|\Delta \zeta_i \|_2^2\leq C\|\nabla \zeta_i \|_2^2+C.
\end{equation}
It holds that $\nabla\zeta_i=\nabla c_i\exp(z_i\phi)+z_ic_i\exp(z_i\phi)\nabla\phi$. Since $\phi$
and $\nabla\phi$ are uniformly bounded in time and space,
$\nabla\zeta$ has the same growth as $c_i$ in (\ref{eq_nabla_ci_growth}), namely $\int_t^{t+r}\|\nabla
\zeta_i(s)\|_2^2ds\leq C(r)$, $t\geq0$. Letting $f(t)=\|\nabla\zeta_i\|_2^2$, $g(t)=h(t)=C$ in the
formulation of Lemma~\ref{lem_gron} we find
\begin{equation}
\label{eq_nabla_zeta_i}
\|\nabla \zeta_i \|_2^2\leq C' \quad (t\in(0,T_{max})).
\end{equation}
Integrating (\ref{eq_nabla_zeta_4}) in time finally gives
\begin{equation}
\label{eq_delta_zeta_i}
\int_0^t\|\Delta \zeta_i(s)\|_2^2ds\leq C'(1+t)
\end{equation}
for some $C'>0$.

Since
\begin{align*}
\nabla c_i&=\nabla \zeta_i\exp(- z_i\phi)- z_i\zeta_i\nabla\phi\exp(- z_i\phi)\text{ and}\\
\Delta c_i&=\Delta \zeta_i\exp(- z_i\phi)-2 z_i\nabla \zeta_i\cdot\nabla\phi\exp(- z_i\phi)\\
&\quad- z_i\zeta_i\Delta\phi\exp(- z_i\phi)+z_i^2\zeta_i|\nabla\phi|^2\exp(- z_i\phi)
\end{align*}
the estimates (\ref{eq_nabla_zeta_i}) and (\ref{eq_delta_zeta_i}) are also valid for $\nabla c_i $
and $\Delta c_i $ by the uniform $L^\infty$-boundedness of $\phi$, $\nabla\phi$ and $\zeta_i$.
\end{proof}

The estimates on $(u,c)$ and on $\phi$ derived so far enable us now to conclude the statement of
Theorem~\ref{theorem_glob_ex}.

\begin{proof}[Proof. (of Theorem~\ref{theorem_glob_ex})]
From Lemmas \ref{lem_u_2}, \ref{lem_ci_l2}, \ref{lem_delta_u_l2} and \ref{lem_delta_c_l2} it
follows that
\begin{equation}
\label{eq_ex_glob_strong_1}
\|u\|^2_{L^2(0,T;\DD(A_S))}+\|c\|^2_{L^2(0,T;W^{2,2}(\Omega))}\leq C(1+T).\
\end{equation}
We use this estimate to show the boundedness of the nonlinearities of ($P$) in order to prove the
boundedness of $\pa_t(u,c)$ in $L^2(0,T;L^2)$.
We have:
\begin{equation*}
\|c_ic_j\|_{L^2(L^2)}^2+\|\nabla c_i\cdot\nabla\phi\|_{L^2(L^2)}^2\leq C T.
\end{equation*}
For the treatment of the term $u\cdot\nabla c_i$ we will, as in the proof of
Lemma~\ref{lem_delta_c_l2}, work with the transformed variable $\zeta_i$. The reason lies in the
homogeneous Neumann conditions for $\zeta_i$ which allows for the 
estimate $\|\nabla^2\zeta_i\|_2\leq
C\|\Delta \zeta_i\|_2$ when using the Gagliardo-Nirenberg inequality. Since
\[u\cdot\nabla c_i=u\cdot\nabla\zeta_i\exp(- z_i\phi)- z_i\zeta_iu\cdot\nabla\phi\exp(- z_i\phi),\]
the estimate
\begin{equation*}
\begin{split}
\|u\cdot\nabla \zeta_i\|_{L^2(L^2)}^2&\leq\int_0^T\|u \|_4^2\|\nabla \zeta_i \|_4^2dt\\
	&\leq C\int_0^T\|u \|_2\|\nabla u \|_2\|\nabla \zeta_i \|_2\|\Delta \zeta_i \|_2dt\leq C(1+T),
\end{split}
\end{equation*}
is sufficient to show $\|u\cdot\nabla c_i\|_{L^2(L^2)}^2\leq C(1+T)$, since $\zeta_i$, $\phi$ and
$\nabla\phi$ are uniformly bounded in $L^\infty$, cf.\ Lemmas \ref{lem_nabla_phi_infty} and
\ref{lem_ci_infty}, and $u$ is uniformly bounded in $L^2$, cf.\ Lemma~\ref{lem_u_2}. We infer
\begin{equation}
\label{eq_ex_glob_strong_2}
\|\pa_tc_i\|_{L^2(L^2)}\leq C(1+T).
\end{equation}

We proceed in the same way for the Navier-Stokes part, i.e.
\begin{align*}
\|P(u\cdot\nabla)u\|_{L^2(L^2)}^2&\leq C\int_0^T\|u \|_4^2\|\nabla u \|_4^2dt\\
	&\leq C'\int_0^T\|u \|_2\|\nabla u \|_2^2\|\nabla^2u \|_2dt\\
	&\leq C''\int_0^T\|u \|_2\|\nabla u \|_2^2\|A_Su \|_2dt\leq C'''(1+T),\\
\|\sum_ {i=1}^Nz_ic_i\nabla\phi\|_{L^2(L^2)}^2&\leq C T,
\end{align*}
where we used the fact that due to $0\in\rho(A_S)$ we have $\|\nabla^2
u\|_2\leq C\|A_Su\|_2$ for $u\in\DD(A_S)$, hence also
\begin{equation}
\label{eq_ex_glob_strong_3}
\|\pa_tu\|_{L^2(L^2)}\leq C(1+T).
\end{equation}
In conclusion it follows from (\ref{eq_ex_glob_strong_1})-(\ref{eq_ex_glob_strong_3}) that $\|(u,c)\|_{\E_{T,2}}\leq C(1+T)$ for finite $T$. So $(u,c)$ can be extended to a global
strong solution. It is unique by the local uniqueness from Theorem~\ref{theorem_strong_loc}.
\end{proof}

\begin{remark}
\label{rem_weak_sol_2}
It is possible to show that a global existence result in the spirit of
Theorem~\ref{theorem_glob_ex} remains still valid if we only impose
$(u^0,c^0)\in L^2_\sigma(\Omega)\times L^2(\Omega)$ with $c_i^0\geq0$
for the initial data. This will be published in the 
doctoral thesis of the second author. We only briefly indicate a sketch of the proof.

\textit{1. Existence and uniqueness of local-in-time weak solutions $(u,c)$:}

The idea is to define a map $L$ which sends $\bar c\in C([0,T];L^2(\Omega))\cap
L^2(0,T;W^{1,2}(\Omega))$ with $\bar c(0)=c^0$ to the unique (weak) solution $c\in
C(0,T;L^2(\Omega))\cap L^2(0,T;W^{1,2}(\Omega))$ of
\begin{equation}
\label{eq_fix_c}
\left.
\begin{split}
\partial_tc_i+\divv (-D_i\nabla c_i+c_i\bar u)&=D_iz_i\divv(\bar c_i\nabla\bar\phi),\quad
i=1,\ldots,N,\\
\pa_\nu c_i|_{\pa\Omega}&=-(z_i\bar c_i\pa_\nu\bar\phi)|_{\pa\Omega},\quad i=1,\ldots ,N,\\
c_i|_{t=0}&=c_i^0,\quad i=1,\ldots ,N,\\
\end{split}
\right\}
\end{equation}
$i=1,\ldots ,N$, where $\bar\phi$ is the unique solution to
\begin{equation}
\label{eq_fix_phi}
-\ep\Delta\bar\phi=\sum_{j=1}^Nz_j\bar c_j,\quad\text{in
}\Omega,\quad\pa_\nu\bar\phi+\tau\bar\phi=\xi,\quad\text{on }\pa\Omega
\end{equation}
and $\bar u$ is the unique (weak) solution to
\begin{equation}
\label{eq_fix_u}
\pa_t\bar u+A_S\bar u+P(\bar u\cdot\nabla)\bar u=-P\left(\sum_{i=1}^Nz_i\bar
c_i\nabla\bar\phi\right),\qquad\bar u(0)=u^0.
\end{equation}
The unique weak solutions $\bar u$ and $c$ exist from well-known results on weak solutions for
Navier-Stokes equations, see e.g. \cite{Sohr}, and for parabolic equations, see e.g. \cite{LSU},
whereas $\bar\phi$ is well-defined by Remark~\ref{rem_triebel}. By choosing the time $T$ small
enough it is possible to show that $L$ is a contraction in a closed subset of
\[Z_T=\{v\in C([0,T];L^2(\Omega))\cap L^2(0,T;W^{1,2}(\Omega)),v(0)=c^0\}.\]
Hence $L$ has a unique fixed point, which represents a local-in-time weak solution. This idea has also been used in \cite{Schmuck}.

\textit{2. Usage of regularizing effects of Navier Stokes and parabolic equations:}

The goal is to infer that for some positive time the trajectory $(u(\cdot),c(\cdot))$ enters the
regularity class for initial data considered in this work, namely $W^{1,2}_{0,\sigma}(\Omega)\times
(W^{1,2}(\Omega)\cap L^\infty(\Omega))$. This can be seen as follows.

Since the weak solution
$(u,c)$ is contained in $L^1(0,T;W^{1,2}_{0,\sigma}(\Omega)\times
W^{1,2}(\Omega))$, by Lebegue's differentiation theorem we obtain
that almost every $t\in(0,T)$ is a Lebesgue point for $(u,c)$.
Now let $t_0\in(0,T)$ be such a Lebesgue point. Then 
$(u(t_0),c(t_0))\in
W^{1,2}_{0,\sigma}(\Omega)\times W^{1,2}(\Omega)$ 
is well-defined.
Thanks to Theorem~\ref{theorem_strong_loc} there
is a unique local strong solution
\[(u^s,c^s)\in W^{1,2}(0,T_0;L^2_\sigma(\Omega)\times L^2(\Omega))\cap
L^2(0,T_0;\DD(A_S)\times W^{2,2}(\Omega))\]
with $(u^s(0),c^s(0))=(u(t_0),c(t_0))$ and for some $T_0\le T$. 
Note that strong solutions are also weak solutions, so by
uniqueness of weak solutions they must coincide, i.e.\ $(u(t_0+t),c(t_0+t))=(u^s(t),c^s(t))$ for
$t\in[0,T_0)$. Therefore $(u(t),c(t))\in W^{1,2}_{0,\sigma}(\Omega)\times W^{1,2}(\Omega)$ for
all $t\in(t_0,t_0+T_0)$. 

It remains to gain $L^\infty$-regularity for $c$.
To this end, however, we can argue in the same fashion. From
\[c\in L^1(t_0,t_0+T_0;W^{2,2}(\Omega)),\]
we see that $c(t_1)\in W^{2,2}(\Omega)\inj
L^\infty(\Omega)$ for a Lebesgue point $t_1\in (t_0,T_0)$. 
Consequently, we have reached the initial situation of this paper.
\end{remark}

\section{Steady states}
\label{stat}

For the definition of steady states we only require that they solve ($P_s$) in the weak sense.

\begin{definition}(Steady state)
\label{def_stat}
Let $\Omega\subset\R^n$, $n\geq2$. The function $(u,c)\in W^{1,2}_{0,\sigma}(\Omega)\times W^{1,2}(\Omega)$ is called a \textit{steady state}
of ($P$), if $(u,c)$ is a weak solution to ($P_s$), i.e.\ if for all $(\ph,\psi)\in
C^\infty_{0,\sigma}(\Omega)\times C^\infty(\overline\Omega)$ the following is valid:
\begin{align}
\label{eq_stat_w_1}
-\int_\Omega \nabla u\cdot\nabla\ph dx-\int_\Omega u(u\cdot\nabla\ph)dx&=\sum_{j=1}^Nz_j\int_\Omega c_j\nabla\phi\cdot\ph dx,\\
\label{eq_stat_w_2}
\int_\Omega(-D_i\nabla c_i-D_iz_ic_i\nabla\phi+c_iu)\cdot\nabla\psi_i dx&=0,\qquad i=1,\ldots ,N,
\end{align}
where $\phi\in W^{3,2}(\Omega)$ satisfies
\begin{equation}
\label{eq_stat_phi_steady_state}
-\ep\Delta\phi=\sum_{j=1}^Nz_jc_j\quad\text{in }\Omega,\quad\pa_\nu\phi+\tau\phi=\xi\quad\text{on
}\pa\Omega.
\end{equation}
\end{definition}

It turns out that steady states have in fact better regularity properties as required from
Definition \ref{def_stat}.

\begin{lemma}
\label{lem_stat_smooth}
Let $(u^\infty,c^\infty)\in W^{1,2}_{0,\sigma}(\Omega)\times W^{1,2}(\Omega)$ be a steady state of
($P$). Then $(u^\infty,c^\infty)\in W^{2,2}(\Omega)\times W^{2,2}(\Omega)$.
\end{lemma}

We omit a proof here. It is an easy consequence of a bootstrap argument
taking into account the semilinear structure and regularity properties 
of the Stokes operator and the Neumann Laplacian, cf., e.g.,
\cite{galdi}, \cite{Sohr}, and
\cite{triebel}.

Let $\Omega\subset \R^2$ be bounded and smooth.
Before considering existence and uniqueness of steady states we give a characterization of weak
solutions to problem ($P_s$).

\begin{lemma}
\label{lem_stat_zeta_infty}
The function $(u^\infty,c^\infty)\in W^{1,2}_{0,\sigma}(\Omega)\times W^{1,2}(\Omega)$, where
$c_i^\infty\geq0$ not identically zero, $i=1,\ldots,N$, is a weak solution to ($P_s$) if and only if
the following conditions are fulfilled:
\begin{enumerate}
\item[(i)] $u^\infty=0$,
\item[(ii)] $\zeta_i^\infty:=c_i^\infty\exp( z_i\phi^\infty)\equiv const$, $i=1,\ldots ,N$,
\end{enumerate}
where $\phi^\infty$ is the corresponding electrical potential from (\ref{eq_stat_phi_steady_state}).
In this situation there is a constant $\eta_0>0$ such that $c_i^\infty\geq \eta_0$, $i=1,\ldots ,N$.
The associated pressure $\pi^\infty$ satisfies $\pi^\infty=\Sigma_jc_j^\infty$ up to some
constant.
\end{lemma}

The proof of this lemma employs the functional $V$ from (\ref{eq_lyapunov}). In terms of the
transformed variable $\zeta$ from (\ref{eq_trans_conc}), the dissipation rate $\Dt V(t)$ can be
rewritten as
\begin{equation}
\label{eq_lya_der_zeta}
\Dt V(t)=-\|\nabla
u(t)\|_2^2-\sum_{i=1}^N\int_\Omega\frac{D_i}{c_i}\exp(-2z_i\phi(t)|\nabla\zeta_i|^2dx,
\end{equation}
whenever $(u,c)$ is a global strong solution from Theorem~\ref{theorem_glob_ex}.

\begin{proof}[Proof. (of Lemma~\ref{lem_stat_zeta_infty})]
Suppose conditions $(i)$ and $(ii)$ are satisfied. Since $u^\infty=0$ and
\begin{equation}
\label{eq_press}
0=\nabla\zeta_i^\infty=(\nabla c_i^\infty+z_ic_i^\infty\nabla\phi^\infty)\exp(z_i\phi^\infty),
\end{equation}
we infer that (\ref{eq_stat_w_2}) is satisfied. From (\ref{eq_press}) it is now evident that the
term $-\sum_jz_jc_j^\infty\nabla\phi^\infty=\nabla \sum_jc_j^\infty$ is a gradient field, so also
(\ref{eq_stat_w_1}) is valid. From this the last assertion of this lemma concerning
the pressure follows, too. Observe that $\zeta_i^\infty=c_i^\infty\exp(z_i\phi^\infty)>0$, because $c_i\geq0$
and $c_i\not\equiv0$. The fact that $\phi^\infty\in W^{3,2}(\Omega)\inj BUC(\overline\Omega)$ due to
Sobolov's embedding theorem shows that there is a constant $\eta_0>0$ such that $c_i^\infty\geq
\eta_0$, $i=1,\ldots ,N$.

Now suppose $(u^\infty,c^\infty)$ is a steady state to ($P$). From Lemma~\ref{lem_stat_smooth} the
function $(u^s(t),c^s(t)):=(u^\infty,c^\infty)$, $t\geq0$, is a global strong solution to ($P$) in
the sense of Definition \ref{def_glob_well}. For
this stationary solution we have $V(t)\equiv E(u^\infty,c^\infty)$, hence $\Dt V(t)\equiv0$. On the
other hand, with (\ref{eq_lya_der_zeta}) we deduce $\nabla u^\infty=0$ and
$\nabla\zeta_i^\infty=0$, which results in $\zeta_i^\infty=const$ and $u^\infty=0$ because of ${u^\infty}_{|\pa\Omega}=0$.
\end{proof}

In order to make use of the global strong solutions $(u,c)$ from Theorem~\ref{theorem_glob_ex} it is
crucial to know that the dissipation terms in (\ref{eq_lya_der_zeta}) become small as time tends to
infinity.

\begin{lemma}
\label{lem_AS_nabla_yi}
Let $(u,c)$ be a global strong solution to ($P$) from Theorem~\ref{theorem_glob_ex} and
$\zeta_i=c_i\exp( z_i\phi)$ be defined as above. Then
\[\lim_{t\to\infty}\left(\|\nabla u(t)\|_2^2+\|\nabla \zeta_i(t)\|_2^2\right)=0.\]
\end{lemma}

\begin{proof}
From (\ref{eq_lya_der_zeta}) and the uniform boundedness of $\|\phi\|_\infty$ and $\|c\|_\infty$,
cf.\ Lemmas \ref{lem_nabla_phi_infty} and \ref{lem_ci_infty}, we have
\begin{equation*}
\begin{split}
\Dt V(t)&\leq-\|\nabla
u\|_2^2-\sum_{i=1}^N\frac{1}{\|c_i\|_\infty\|\exp(-2z_i\phi)\|_\infty}\|\nabla \zeta_i\|_2^2\\
	&\leq -C\left(\|\nabla u\|_2^2+\sum_{i=1}^N\|\nabla \zeta_i\|_2^2\right).
\end{split}
\end{equation*}
The functional $V$ is bounded from below, so using the fact that it is 
absolutely continuous, cf.\
Remark~\ref{rem_der_lya}, we compute
\[L\leq V(t)\leq V(0)-C\int_0^t\left(\|\nabla u\|_2^2+\sum_{i=1}^N\|\nabla
\zeta_i\|_2^2ds\right),\quad t\geq0.\]
Rearrangement therefore yields
\begin{equation}
\label{ez_int_inf}
\int_0^\infty\left(\|\nabla u\|_2^2+\sum_{i=1}^N\|\nabla \zeta_i\|_2^2\right)ds<\infty.
\end{equation}
From (\ref{pr_lem_glob}) there is a constant $C>0$ such that
\begin{align*}
\Dt\|\nabla u\|_2^2\leq C\|\nabla u\|_2^4+C.
\end{align*}
For $\ep>0$ there is $T_\ep>0$ such that
\[\int_{T_\ep}^\infty\|\nabla u\|_2^2ds<\ep\]
by (\ref{ez_int_inf}). With $f(t):=\|\nabla u\|_2^2$, $g(t):=C\|\nabla u\|_2^2$, $h(t):=C$,
$a_1=\ep$, $a_2=C\ep$ and $a_3=Cr$ application of Lemma~\ref{lem_gron} yields
\[\|\nabla u(t+r)\|_2^2\leq (\ep/r+Cr)\exp(C\ep),\quad t\geq T_\ep.\]
So choosing first $r>0$ and then $\ep>0$ small enough shows $\|\nabla u(t)\|_2^2<\de$ for $t\geq
T_\ep+r=:T_\de$. Note that from Remark~\ref{rem_gron_1}, the uniform bounds on $\phi$ and
$\nabla\phi$, and (\ref{eq_nabla_zeta_4}) we have exactly the same situation for $\nabla\zeta_i$ as
for $\nabla u$. So the assertion for $\nabla \zeta_i$ is obtained in the same way.
\end{proof}

\begin{proof}[Proof. (of Theorem~\ref{theorem_stat})]

\noindent\textit{Existence.}

Let $(u^0,c^0)\in W^{1,2}(\Omega)\times (W^{1,2}(\Omega)\cap L^\infty(\Omega))$ with $c_i^0\geq0$
and $\int_\Omega c_i^0(x)dx=m_i>0$, let $(u,c)$ be the corresponding global strong solution from
Theorem~\ref{theorem_glob_ex} and let $(t_k)_k$ be a sequence with $t_k\to\infty$. The estimates
derived in Section~\ref{global} show that for the trajectory $(u,c)$ there is a constant $C$ such
that
\[\|u(t_k)\|_2+\|\nabla u(t_k)\|_2+\|c(t_k)\|_2+\|\nabla c_i(t_k)\|_2\leq C\]
independently of time, cf.\ Lemmas \ref{lem_u_2}, \ref{lem_ci_l2}, \ref{lem_delta_u_l2} and
\ref{lem_delta_c_l2}. So $(u(t_k),c(t_k))$ is bounded in $W^{1,2}_{0,\sigma}(\Omega)\times W^{1,2}(\Omega)$. By
the compact embedding $W^{1,2}(\Omega)\inj L^2(\Omega)$, the trajectory is
relatively compact in $L^2(\Omega)$, so there is a subsequence of $(t_{k})_{k\in\N}$, which will be
denoted in the same way, and an element $(u^\infty,c^\infty)\in L^2_\sigma(\Omega)\times
L^2(\Omega)$ such that
\[(u(t_k),c(t_k))\to(u^\infty,c^\infty)\ \text{in }L^2.\]
This convergence implies convergence almost everywhere in $\Omega$ of a
subsequence, thus by nonnegativity of $c_i(t_k)$ we deduce nonnegativity of $c_i^\infty$. Since
$\|c_i(t_k)\|_1=m_i$ we readily approve $\|c_i^\infty\|_1=m_i$, $i=1,\ldots,N$. The convergence of
$c_i(t_k)$ in $L^2$ implies convergence of $\phi(t_k)\to\phi^\infty$ in $W^{2,2}(\Omega)$ where
$\phi^\infty$ denotes the unique solution to
\begin{equation*}
-\ep\Delta\phi=-\sum_{i=1}^Nz_ic_i^\infty,\quad \pa_\nu\phi+\tau\phi=\xi.
\end{equation*}
With Poincar\'e's inequality we obtain
\[\|u(t_k)\|_2\leq C\|\nabla u(t_k)\|_2\to0,\]
due to Lemma~\ref{lem_AS_nabla_yi}, so $u^\infty=0$.

Recall the definition $\zeta_i=c_i\exp(z_i\phi)$. From the uniform $L^\infty$-bounds on $\phi$ and
$c$, cf.\ Lemmas \ref{lem_nabla_phi_infty} and \ref{lem_ci_infty}, we see that the sequence
$(\zeta_i(t_k))_k$ is bounded in $L^2(\Omega)$. Since $\nabla\zeta_i(t_k)\to0$ in $L^2(\Omega)$, cf.
Lemma~\ref{lem_AS_nabla_yi}, we infer that $\zeta_i$ is bounded in $W^{1,2}(\Omega)$. Since the
embedding $W^{1,2}(\Omega)\inj L^2(\Omega)$ is compact there exists $\zeta_i^\infty\in
L^2(\Omega)$ such that, up to a subsequence, $\zeta_i(t_k)\to\zeta^\infty_i$ in $L^2(\Omega)$.
Because $\nabla\zeta_i(t_k)\to0$ we deduce that
$\zeta_i^\infty$ is in fact constant and that $\zeta_i(t_k)\to\zeta_i^\infty$ in $W^{1,2}(\Omega)$.

From the uniform $L^\infty$-bounds of $\phi(t_k)$ and $c_i(t_k)$ it follows that $\sum_j\|c_j^\infty\|_\infty+\|\phi^\infty\|_\infty\leq C$. Thus from the continuity of the function $(x,y)\mapsto x\exp(z_iy)$ we have $\zeta_i^\infty=c_i^\infty\exp(z_i\phi^\infty)$ almost everywhere.

Hence if we can show that $c^\infty\in W^{1,2}(\Omega)$, then Lemmas
\ref{lem_stat_smooth} and \ref{lem_stat_zeta_infty} imply that $(u^\infty,c^\infty)$ is a solution
to ($P_s$).

Note that from the definition of $\zeta_i$ we have
\begin{equation}
\label{eq_lem_zeta}
\exp(-z_i\phi)\nabla \zeta_i=\nabla c_i+ z_ic_i\nabla\phi.
\end{equation}
Since $\nabla\zeta_i(t)\to0$ in $L^2(\Omega)$ as $t\to\infty$ and the sequence $\phi(t_k)$ is uniformly bounded in
$L^\infty(\Omega)$, for the proof of $\nabla c_i^\infty\in L^2$ it is sufficient to show the
convergence of $c_i(t_k)\nabla\phi(t_k)$ in $L^2$. We compute
\begin{equation*}
\begin{split}
\|c_i&(t_k)\nabla\phi(t_k)-c_i^\infty\nabla\phi^\infty\|_2\\
&\leq\|c_i(t_k)-c_i^\infty\|_2\|\nabla\phi(t_k)\|_\infty+\|c_i^\infty\|_\infty\|\nabla\phi(t_k)-\nabla\phi^\infty\|_2\to0,
\end{split}
\end{equation*}
as $k\to\infty$, where we use the convergence of $\phi(t_k)\to\phi^\infty$ in $W^{1,2}(\Omega)$ and
the convergence of $c_i(t_k)\to c_i^\infty$ in $L^2(\Omega)$. So by (\ref{eq_lem_zeta}) we see that
$\nabla c_i(t_k)$ converges in $L^2(\Omega)$ as well, therefore $c_i^\infty\in W^{1,2}(\Omega)$ and
$(u^\infty,c^\infty)=(0,c^\infty)$ is a steady state to problem ($P$). Therefore $c^\infty\in
W^{2,2}(\Omega)$ is a solution to~($P'_s$).

Note that the assertion concerning the pressure $\pi^\infty$ is already contained in Lemma~\ref{lem_stat_zeta_infty}.

\noindent\textit{Uniqueness.}

Due to the fact that $u^\infty=0$ it is sufficient to prove that $c^\infty$ with corresponding
$\phi^{\infty}$ is the unique nonnegative solution to ($P'_s$). So suppose there are two nonnegative solutions
$c^{\infty},\tilde c^\infty$ with $\int_\Omega c_i^\infty dx=\int_\Omega \tilde c_i^\infty dx=m_i$, $i=1,\ldots ,N$. Note that $\int_\Omega
c_i^\infty -\tilde c_i^\infty dx=m_i-m_i=0$. Let us denote the corresponding potentials by
$\phi^\infty ,\tilde\phi^\infty $. From Lemma~\ref{lem_stat_zeta_infty} we have
$\zeta_i^\infty =c_i^\infty \exp(z_i\phi^\infty )=const$ and $\tilde\zeta_i^\infty =\tilde c_i^\infty \exp(z_i\tilde\phi^\infty )=const$. Since $\zeta_i^\infty ,\tilde\zeta_i^\infty >0$, we
can take the logarithm. It holds $\log \zeta_i^\infty =\log
c_i^\infty +z_i\phi^\infty =const$ and correspondingly for $\tilde\zeta_i^\infty $. Therefore, integrating by parts we obtain
\begin{align*}
0&=\sum_{i=1}^N\int_\Omega(c_i^\infty -\tilde c_i^\infty )(\log c_i^\infty -\log
\tilde c_i^\infty +z_i(\phi^\infty -\tilde\phi^\infty ))dx\\
&=\int_\Omega\sum_i\left(z_i(c_i^\infty -\tilde c_i^\infty )(\phi^\infty -\tilde\phi^\infty )+(c_i^\infty -\tilde c_i^\infty )\log
\frac{c_i^\infty }{\tilde c_i^\infty }\right)dx\\
&=\ep\|\nabla(\phi^\infty -\tilde\phi^\infty )\|_2^2+\ep\tau\|\phi^\infty -\tilde\phi^\infty \|_{2,\pa\Omega}^2\\
&\qquad+\int_\Omega\sum_i(c_i^\infty -\tilde c_i^\infty )\log\frac{c_i^\infty }{\tilde c_i^\infty }dx\geq0,
\end{align*}
so $c^\infty =\tilde c^\infty$ and therefore $\phi^\infty =\tilde\phi^\infty $.

\noindent\textit{Convergence}

It was shown that semi-orbits $(u,c)$ are compact and any accumulation point $(u^\infty,c^\infty)$
represents a steady state. The convergence $(u(t_k),c(t_k))\to(u^\infty,c^\infty)$ in this situation
means convergence with respect to the topology in $W^{1,2}_{0,\sigma}(\Omega)\times
W^{1,2}(\Omega)$. From the existence and uniqueness of steady states and the fact that the
sequence $(t_k)_{k\in\N}$ was arbitrary, the convergence (\ref{eq_conv_steady}) follows.
\end{proof}

\section{Asymptotic behaviour}
\label{asymp}

Let $\Omega\subset\R^2$ be bounded and smooth, let $(u^0,c^0)\in W^{1,2}_{0,\sigma}(\Omega)\times
(W^{1,2}(\Omega)\cap
L^\infty(\Omega))$ with $c_i^0\geq0$ and corresponding masses $m_i=\int_\Omega c_i^0dx>0$,
$i=1,\ldots,N$, be given and let $c^\infty\in W^{2,2}(\Omega)$ denote the unique solution to
(\ref{eq_stat_u0}) subject to $\int_\Omega c_i^\infty dx=m_i$, $i=1,\ldots,N$ from Theorem \ref{theorem_stat}.

Let us define
\[\SS:= W^{1,2}_{0,\sigma}(\Omega)\times\left\{v\in W^{1,2}(\Omega)\cap
L^\infty(\Omega): v_i\geq0,\int_\Omega v_i(x)dx=m_i\right\}\]
and introduce the function $\Psi\colon\SS\to\R$, given by
\begin{align*}
\Psi(u,c):=&E(u,c)-E(u^\infty,c^\infty)\\
=&\frac12\|u\|_2^2+\sum_{i=1}^N\int_\Omega \left(c_i\log c_i-c_i^\infty\log c^\infty\right) dx\\
&\quad+\frac\ep2(\|\nabla\phi\|_2^2-\|\nabla\phi^\infty\|_2^2)+\frac{\ep\tau}2(\|\phi\|_{2,\pa\Omega}^2-\|\phi^\infty\|_{2,\pa\Omega}^2)\\
=&\frac12\|u\|_2^2+\sum_{i=1}^N\int_\Omega
\left(c_i(\log\frac{c_i}{c_i^\infty}-1)+c_i^\infty\right)dx+\frac\ep2\|\nabla(\phi-\phi^\infty)\|_2^2\\
&+\frac{\ep\tau}2\|\phi-\phi^\infty\|_{2,\pa\Omega}^2+\ep\int_\Omega\nabla\phi^\infty\cdot\nabla(\phi-\phi^\infty)dx\\
&+\ep\tau\int_{\pa\Omega}\phi^\infty(\phi-\phi^\infty)dx+\sum_{i=1}^N\int_\Omega(c_i-c_i^\infty)(1+\log
c_i^\infty)dx.
\end{align*}
Using integration by parts and the fact that $c_i^\infty\exp(z_i\phi^\infty)$ is constant we
obtain
\begin{align*}
\ep&\int_\Omega\nabla\phi^\infty\cdot\nabla(\phi-\phi^\infty)dx+\ep\tau\int_{\pa\Omega}\phi^\infty(\phi-\phi^\infty)dx\\
	&\hspace{6cm}+\sum_{i=1}^N\int_\Omega(c_i-c_i^\infty)(1+\log c_i^\infty)dx\\
=&\int_\Omega\sum_{i=1}^N(c_i-c_i^\infty)(1+\log c_i^\infty+z_i\phi^\infty)dx\\
=&\int_\Omega\sum_{i=1}^N(c_i-c_i^\infty)(1+\log\zeta_i^\infty)dx=0.
\end{align*}
Therefore
\begin{equation}
\label{eq_def_PSI}
\begin{split}
\Psi(u,c)&=\frac12\|u\|_2^2+\sum_{i=1}^N\int_\Omega\left(c_i(\log\frac{c_i}{c_i^\infty}-1)+c_i^\infty\right)
dx\\
&\quad+\frac\ep2\|\nabla(\phi-\phi^\infty)\|_2^2+\frac{\ep\tau}2\|\phi-\phi^\infty\|_{2,\pa\Omega}^2.
\end{split}
\end{equation}
Let $(u,c)$ be the global strong solution to ($P$) from Theorem~\ref{theorem_glob_ex} with initial
data $(u^0,c^0)$. Note that due to Theorem~\ref{theorem_strong_loc} and
Corollary~\ref{cor_mass_cons} for almost all $t\geq0$ we have $(u(t),c(t))\in\SS$, hence $t\mapsto\Psi(u(t),c(t))$ is well-defined. The value $\Psi(u(t),c(t))$ represents
the difference between the energy at some time $t$ and the energy in the equilibrium state. From the
monotone behaviour of $V$ along solutions of ($P$), cf.\ Lemma~\ref{lem_lya_der}, we see that $\Psi$
is monotonely decreasing as
well with the same dissipation rate given in~(\ref{eq_lya_der_zeta}).

\begin{proposition}
\label{prop_napier}
Let $a,b>0$.

(a) It holds
\[(\sqrt a-\sqrt b)^2\leq a(\log a-\log b)+b-a.\]
This inequality remains valid in the limits $a\to0$ and $b\to0$ as well.

(b) Let $b\geq b_0>0$ and $K:=1/b_0$. Then
\[a(\log a-\log b)+b-a\leq K(a-b)^2.\]
\end{proposition}

Proposition~\ref{prop_napier} can be proven by elementary calculus employing the relation
\begin{equation}
\label{eq_napier}
\frac2{a+b}<\frac{\log a-\log b}{a-b}<\frac1{\sqrt{ab}},\quad a>b>0.
\end{equation}

\begin{theorem}
\label{thm_asymp}
Let $\Omega\subset\R^2$ be bounded and smooth, let $(u^0,c^0)\in W^{1,2}_{0,\sigma}(\Omega)\times
(W^{1,2}(\Omega)\cap L^\infty(\Omega))$ with
initial masses $m_i=\int_\Omega c_i^0dx>0$, $i=1,\ldots,N$, be given, let $(u,c)$ be the global
strong solution to ($P$) from Theorem~\ref{theorem_glob_ex} with initial data $(u^0,c^0)$, and let
$\Psi$ be defined as in~(\ref{eq_def_PSI}). Then there are constants $C,\omega>0$ depending on the
initial data such that
\[\Psi(u(t),c(t))\leq Ce^{-\omega t}\quad\text{for a.a. }t\geq0.\]
\end{theorem}

\begin{proof}
Recall the definition of $V$ in (\ref{eq_lyapunov}) and note that $\Dt
V=\Dt\Psi(u(\cdot),c(\cdot))$, so from (\ref{eq_lya_der_zeta}) we have
\begin{equation}
\label{eq_Psi_der}
\Dt\Psi(u(t),c(t))=-D(u(t),c(t)),
\end{equation}
where
\[D(u,c)=\|\nabla
u(t)\|_2^2+\sum_{i=1}^N\int_\Omega\frac{D_i}{c_i(t)}\exp(-2z_i\phi(t))|\nabla\zeta_i(t)|^2dx\]
with $\zeta_i=c_i\exp(z_i\phi)$. The idea of the proof is to estimate $\Psi$ suitably in terms of
the dissipation rate $D$ and obtain exponential decay via Gronwall.

Let us define the auxiliary variable
\begin{equation}
\label{eq_def_psi_i}
\psi_i(t):=\sqrt{\frac{c_i(t)}{c_i^\infty}}\exp\left(z_i\frac{\phi(t)-\phi^\infty}2\right)-1.
\end{equation}
Recall that $c_i^\infty\exp(z_i\phi^\infty)=\zeta_i^\infty>0$ due to
Lemma~\ref{lem_stat_zeta_infty}. For $c_i\ne0$, the gradient of $\psi_i$ satisfies
\begin{equation}
\label{eq_nabla_psi_i}
\begin{split}
|\nabla \psi_i(t)|^2&=\frac1{4c_i(t)c_i^\infty\exp(z_i(\phi^\infty-\phi(t)))}|\nabla
c_i(t)+z_ic_i(t)\nabla\phi(t)|^2\\
	&\leq\frac{CD_i}{c_i(t)}\exp(-z_i\phi(t))|\nabla\zeta_i(t)|^2,
\end{split}
\end{equation}
where we use the fact that $c_i^\infty\exp(z_i\phi^\infty)=const$ and $\|\phi\|_\infty\leq C$
uniformly in time, so
\[\|\nabla u(t)\|_2^2+\sum_{i=1}^N\|\nabla \psi_i(t)\|_2^2\leq CD(u(t),c(t)).\]
For proving the assertion it is sufficient to estimate
\begin{equation}
\label{eq_assert_1}
\Psi(u(t),c(t))\leq C\left(\|\nabla u(t)\|_2^2+\sum_{i=1}^N\|\nabla \psi_i(t)\|_2^2\right)
\end{equation}
with some constant $C>0$ independent of $t$, since then, by (\ref{eq_Psi_der}) and Gronwall, we
obtain exponential decay of $\Psi(u(\cdot),c(\cdot))$. Note that we only require (\ref{eq_assert_1})
to hold true for those states $(u(t),c(t))$ which are attained by the solution $(u,c)$.

By Poincar\'e's inequality we already know that $\|u\|_2\leq C\|\nabla u\|_2$ for all $u\in
W^{1,2}_{0,\sigma}(\Omega)$. So for (\ref{eq_assert_1}) it suffices to show that for given $R,R'>0$
we have
\begin{equation}
\label{eq_assert_2}
\begin{split}
&\tilde\Psi(c):=\\
&:=\sum_{i=1}^N\int_\Omega \left(c_i(\log\frac{c_i}{c_i^\infty}-1)+c_i^\infty\right)
dx+\frac\ep2\|\nabla(\phi-\phi^\infty)\|_2^2+\frac{\ep\tau}2\|\phi-\phi^\infty\|_{2,\pa\Omega}^2\\
&\ \leq C\sum_i\|\nabla \psi_i\|_2^2
\end{split}
\end{equation}
for all $c\in Z_{R,R'}$ with
\begin{equation*}
\begin{split}
Z_{R,R'}:=\bigg\{v\in W^{1,2}(\Omega)\cap L^\infty&(\Omega):\\
	&\quad v_i\geq0,\int_\Omega v_i=m_i,\tilde\Psi(v)\leq R,\|v\|_\infty\leq R'\bigg\},
\end{split}
\end{equation*}
where $\phi\in W^{3,2}(\Omega)$ is the unique solution to
\[-\ep\Delta\phi=\sum_{i=1}^Nz_ic_{i}\ \text{in}\ \Omega,\quad\pa_\nu\phi+\tau\phi=\xi\
\text{on}\ \pa\Omega.\]
Note that for given initial data with initial masses $m_i>0$ the phase space of $c$ is contained in
$Z_{R,R'}$ for sufficiently large $R,R'>0$, cf.\
Corollary~\ref{cor_mass_cons}, Remark~\ref{rem_der_lya},
and Lemma~\ref{lem_ci_infty}. From the Sobolev embedding $W^{2,2}(\Omega)\inj BUC(\overline\Omega)$
we compute
\begin{equation}
\label{eq_bound_phi_k}
\|\phi\|_\infty\leq C\|\phi\|_{W^{2,2}}\leq C'\|c\|_2\leq C''\|c\|_\infty\leq C''R'
\end{equation}
for $c\in Z_{R,R'}$, so we observe that the corresponding potentials $\phi$ are uniformly bounded
in $L^\infty(\Omega)$.

For the proof of relation (\ref{eq_assert_2}) observe that by the boundedness of $c_i^\infty$ from
below due to Lemma~\ref{lem_stat_zeta_infty} and Proposition~\ref{prop_napier} there are constants
$C_1,C_2>0$ such that
\begin{equation}
\label{eq_Psi_lu}
\begin{split}
C_1\bigg(\sum_i\|\sqrt{c_i}-\sqrt{c_i^\infty}\|_2^2+\|\phi-\phi^\infty&\|_{W^{1,2}}^2\bigg)\leq\tilde\Psi(c)\\
&\leq C_2\bigg(\sum_i\|c_i-c_i^\infty\|_2^2+\|\phi-\phi^\infty\|_{W^{1,2}}^2 \bigg)
\end{split}
\end{equation}
for all $c\geq0$ and corresponding potentials $\phi$.

We will derive (\ref{eq_assert_2}) indirectly. So suppose relation (\ref{eq_assert_2}) does not
hold.

\noindent\textit{Step 1.}

Then there is a sequence $l_k\to\infty$ and $(c^k)_k\subset Z_{R,R'}$ with corresponding
$\phi^k$, solution to
\[-\ep\Delta\phi^k=\sum_iz_ic_i^k,\ \text{in}\ \Omega,\quad\pa_\nu\phi^k+\tau\phi^k=\xi,\
\text{on}\ \pa\Omega,\]
and auxiliary functions $\psi_i^k$, given by
\begin{equation}
\label{eq_def_psi_ki}
\psi_i^k=\sqrt{\frac{c_i^k}{c_i^\infty}}\exp\left(z_i\frac{\phi^k-\phi^\infty}2\right)-1,
\end{equation}
such that
\begin{equation}
\label{eq_ass}
l_k\sum_i\|\nabla \psi_i^k\|_2^2\leq\tilde\Psi(c^k)\leq R.
\end{equation}
Note that for $c_i^k,\psi_i^k$ superscripts refer to the index of the sequence and subscripts to the corresponding components. Without loss of generality we may exclude the case that $c^k= c^\infty$ for some $k\in\N$, since in
this case inequality (\ref{eq_assert_2}) is trivially satisfied.

By relation (\ref{eq_ass}) we see $\nabla \psi_i^k\to0$ in $L^2(\Omega)$ for $k\to\infty$. Due to the
boundeds for $c^k$ and $\phi^k$, implied by the definiton of $Z_{R,R'}$ and
(\ref{eq_bound_phi_k}), we see that $\|\psi_i^k\|_2$ can be bounded by a constant independent of
$k$, cf.\ (\ref{eq_def_psi_ki}). Hence the sequence $(\psi_i^k)_k$ is bounded in $W^{1,2}(\Omega)$,
and therefore $\psi_i^k\rightharpoonup\hat\psi_i$ in $W^{1,2}(\Omega)$ and $\psi_i^k\to\hat\psi_i$
in $L^2(\Omega)$ via compactness. Because both $\psi_i^k\to\hat\psi_i$ and $\nabla\psi_i^k\to0$
in $L^2(\Omega)$, it follows that
\begin{equation}
\label{eq_conv_psi_ki}
\psi_i^k\to\hat\psi_i\ \text{in }W^{1,2}(\Omega)\ \text{as }k\to\infty,
\end{equation}
where $\nabla\hat\psi_i=0$, hence $\hat\psi_i=const.$

By (\ref{eq_Psi_lu}) we have $\|\phi^k-\phi^\infty\|_{W^{1,2}}\leq C$, so there is $\hat\phi\in
W^{1,2}(\Omega)$ such that
\begin{equation}
\label{eq_conv_phi_k}
\phi^k\rightharpoonup\hat\phi\ \text{in}\ W^{1,2}(\Omega),\ \text{hence}\ \phi^k\to\hat\phi\
\text{in}\ L^2(\Omega),\ \text{as }k\to\infty,
\end{equation}
again using compactness. From the uniform boundedness of $\|\phi^k\|_\infty$ by
(\ref{eq_bound_phi_k}) and its convergence almost everywhere we deduce $\|\hat\phi\|_\infty<\infty$. Let
\begin{equation}
\label{eq_c_hat}
\hat{c_i}:=c_i^\infty(1+\hat\psi_i)^2\exp(z_i(\phi^\infty-\hat\phi)).
\end{equation}
By the definition of $\psi_i^k$ in (\ref{eq_def_psi_ki}) we have $c_i^k=
c_i^\infty(1+\psi_i^k)^2\exp(z_i(\phi^\infty-\phi^k))$. Straight forward manipulations reveal
\begin{align*}
c_i^k-\hat
c_i&=c_i^\infty\exp(z_i(\phi^\infty-\phi^k))[(1+\psi_i^k)^2-(1+\hat\psi_i)^2\exp(z_i(\phi^k-\hat\phi))]\\
&=c_i^\infty\exp(z_i(\phi^\infty-\phi^k))[(\psi_i^k-\hat\psi_i)^2+2(\psi_i^k-\hat\psi_i)+2\hat\psi_i(\psi_i^k-\hat\psi_i)\\
&\quad+(1+\hat\psi_i)^2(1-\exp(z_i(\phi^k-\hat\phi)))].
\end{align*}
Using the a priori estimates
$\|\phi^k\|_\infty+\|\hat\phi\|_\infty+\|\phi^\infty\|_\infty+\|c^\infty\|_\infty\leq C$ and the
fact that $W^{1,2}(\Omega)\inj L^4(\Omega)$ gives
\begin{equation}
\label{eq_conv_c_ki_hat}
\|c_i^k-\hat c_i\|_2\leq
C(\|\psi_i^k-\hat\psi_i\|_4^2+\|\psi_i^k-\hat\psi_i\|_2+\|\phi^k-\hat\phi\|_2)\to0\ \text{as
}k\to\infty.
\end{equation}
By continuity it is true that $-\ep\Delta\hat\phi=\sum_iz_i\hat c_i$ in $\Omega$, and
$\pa_\nu\hat\phi+\tau\hat\phi=\xi$ on $\pa\Omega$, and
\[\hat c_i\exp(z_i\hat\phi)=c_i^\infty\exp(z_i\phi^\infty)(1+\hat\psi_i)^2\equiv const.\]
From relation (\ref{eq_c_hat}) we readily approve that $\hat c_i\in W^{1,2}(\Omega)$, since both
$\hat\psi_i$ and $c_i^\infty\exp(z_i\phi^\infty)$ are constant and $\hat\phi\in W^{1,2}(\Omega)$.
Thus $(\hat c,\hat\phi)$ is the unique steady state to ($P$) from Lemma~\ref{lem_stat_zeta_infty}
and Theorem~\ref{theorem_stat},
i.e.\ $(\hat c,\hat\phi)=(c^\infty,\phi^\infty)$. So relation (\ref{eq_conv_c_ki_hat}) shows
\begin{equation}
\label{eq_conv_c_ki}
c_i^k\to c_i^\infty\ \text{in }L^2(\Omega)\ \text{as }k\to\infty.
\end{equation}
For the sequence $(\tilde\Psi(c^k))_k$ this implies
\begin{equation}
\label{eq_conv_Psi_k}
\la_k^2:=\tilde\Psi(c_k)\to0
\end{equation}
by (\ref{eq_Psi_lu}).

\noindent\textit{Step 2.}

Recall that we assume $c_k\ne c^\infty$ for $k\in\N$, so $\la_k^2=\tilde\Psi(c_k)>0$ for all
$k\in\N$. Having obtained the convergence of the sequences $(\psi_i^k)_k$, $(\phi^k)_k$,
$(c_i^k)_k$ and $\la_k$ in (\ref{eq_conv_psi_ki}), (\ref{eq_conv_phi_k}),
(\ref{eq_conv_c_ki}) and (\ref{eq_conv_Psi_k}) we define new variables
\begin{equation}
\label{eq_def_w_ki}
w_i^k:=\frac{\psi_i^k}{\la_k},\quad y_i^k:=\frac{c_i^k-c^\infty}{\la_k},\quad
\chi^k:=\frac{\phi^k-\phi^\infty}{\la_k}.
\end{equation}
Again for $w_i^k$ and $y_i^k$ superscripts represent the index of the sequence and subscripts denote the respective components.
Note that for $k\in\N$ the function $\chi^k$ is the unique solution to
\begin{equation}
\label{eq_poisson_chi_k}
-\ep\Delta \chi^k=\sum_iz_i y_i^k\quad \text{in }\Omega,\ \text{ and }\ \pa_\nu \chi^k+\tau \chi^k=0\quad\text{on
}\pa\Omega.
\end{equation}
Dividing (\ref{eq_ass}) by $l_k\la_k^2$ it follows that $\sum_i\|\nabla w_i^k\|_2^2\leq1/l_k$, so
$\nabla w_i^k\to0$ in $L^2(\Omega)$. We shall argue in an analogue way as for
(\ref{eq_conv_psi_ki}). For the $L^2$-boundedness of $w_i^k$, observe that $\psi_i^k$ can be
rewritten as
\[\psi_i^k=\frac1{\sqrt{c_i^\infty}}\left(\sqrt{c_i^k}-\sqrt{c_i^\infty}\right)+\sqrt{\frac{c_i^k}{c_i^\infty}}\left(\exp\left(z_i\frac{\phi^k-\phi^\infty}2\right)-1\right).\]
With the help of (\ref{eq_Psi_lu}) we conclude
\begin{equation}
\label{eq_abs_wni}
\begin{split}
\sum_{i=1}^N\|\psi_i^k\|_2^2+\|\phi^k-\phi^\infty\|_{W^{1,2}}^2&\leq
C\bigg(\sum_{i=1}^N\|\sqrt{c_i^k}-\sqrt{c_i^\infty}\|_2^2+\|\phi^k-\phi^\infty\|_{W^{1,2}}^2\bigg)\\
&\leq C'\tilde\Psi(c^k).
\end{split}
\end{equation}
Dividing (\ref{eq_abs_wni}) by $\la_k^2=\tilde\Psi(c_k)$ we have $\|w_i^k\|_2\leq C'$ independent
of $k$, hence
\begin{equation}
\label{eq_conv_w_ki}
w_i^k\to \hat w_i\ \text{in } W^{1,2}(\Omega)\ \text{as }k\to\infty,
\end{equation}
where $\hat w$ is constant. Also from the division of (\ref{eq_abs_wni}) by $\la_k^2$ we infer in
the same way as for (\ref{eq_conv_phi_k}) the existence of a $\hat\chi \in W^{1,2}(\Omega)$ such
that
\begin{equation}
\label{eq_conv_chi_k}
\chi^k\rightharpoonup\hat\chi \ \text{in }W^{1,2}(\Omega), \ \text{hence }\chi^k\to\hat\chi \
\text{in
}L^2(\Omega),\ \text{as }k\to\infty.
\end{equation}
With relations (\ref{eq_conv_w_ki}) and (\ref{eq_conv_chi_k}) we define
\begin{equation}
\label{eq_def_y_hat}
\hat y_i:=c_i^\infty(2\hat w_i-z_i\hat\chi ).
\end{equation}
For the difference $y_i^k-\hat y_i$ we obtain
\begin{equation*}
\begin{split}
&y_i^k-\hat
y_i=c_i^\infty\left(\frac1{\la_k}\left((1+\psi_i^k)^2\exp(z_i(\phi^\infty-\phi^k))-1\right)-2\hat
w_i+z_i\hat\chi \right)\\
&=c_i^\infty\bigg(2\frac{\psi_i^k}{\la_k}\big(\exp(z_i(\phi^\infty-\phi^k))-1\big)+2\left(\frac{\psi_i^k}{\la_k}-\hat
w_i\right)+z_i\left(\hat\chi -\frac{\phi^k-\phi^\infty}{\la_k}\right)\\
&\quad+\frac{(\psi_i^k)^2}{\la_k^2}\la_k\exp(z_i(\phi^\infty-\phi^k))+\frac{\exp(z_i(\phi^\infty-\phi^k))-1-z_i(\phi^\infty-\phi^k)}{\la_k}\bigg)\\
&=c_i^\infty\Big(2w_i^k\big(\exp(z_i(\phi^\infty-\phi^k))-1\big)+2(w_i^k-\hat w_i)+z_i(\hat\chi -\chi^k)\\
	&\quad+\la_k (w_i^k)^2\exp(z_i(\phi^\infty-\phi^k))+\frac{\la_k}2z_i^2(\chi^k)^2\exp(hz_i(\phi^\infty-\phi^k))\Big)
\end{split}
\end{equation*}
for some function $h\in[0,1]$ using Taylor's formula.
Hence
\[\|y_i^k-\hat y_i\|_2\leq C(\|\phi^k-\phi^\infty\|_4^2+\|w_i^k-\hat w_i\|_2+\la_k+\|\chi^k-\hat
z\|_2)\to0.\]
By continuity it follows from (\ref{eq_poisson_chi_k}) that
\begin{equation}
\label{eq_poisson_z_hat}
-\ep\Delta \hat\chi =\sum_iz_i \hat y_i\ \text{in }\Omega,\ \text{ and }\ \pa_\nu\hat\chi
+\tau\hat\chi =0\
\text{on }\pa\Omega,
\end{equation}
and that $\chi^k\to\hat\chi $ in $W^{2,2}(\Omega)$. Therefore, using integration by parts
and representation (\ref{eq_def_y_hat}), we obtain
\begin{align*}
0&\leq\ep\|\nabla\hat\chi \|_2^2+\ep\tau\|\hat\chi
\|_{2,\pa\Omega}^2+\int_\Omega\sum_ic_i^\infty(2\hat
w_i-z_i\hat\chi )^2\\
&=-\ep\int_\Omega \hat\chi \Delta\hat\chi +\int_\Omega\sum_ic_i^\infty(2\hat w_i-z_i\hat\chi )^2\\
&=\sum_i\int_\Omega z_i\hat\chi \hat y_i+\hat y_i(2\hat w_i-z_i\hat\chi )=2\sum_i\int_\Omega \hat
y_i\hat
w_i=0,
\end{align*}
since $\int_\Omega y_i^k=0$ for all $k$. Therefore $\hat\chi =0$, $\hat w=0$ and finally $\hat
y_i=0$.
However, dividing (\ref{eq_Psi_lu}) by $\la_k^2$ now implies that
\[1\leq\sum_i\|y_i^k\|_2^2+\|\chi^k\|_{W^{1,2}}^2\to0\ \text{as}\ k\to\infty,\]
a contradiction.
\end{proof}

Now the proof of Theorem~\ref{theorem_conv} remains an easy task.

\begin{proof}[Proof. (of Theorem~\ref{theorem_conv})]
From the definition of $\Psi$ in (\ref{eq_def_PSI}) it is clear that $\Psi\geq0$, so
Theorem~\ref{thm_asymp} implies $\|u(t)\|_2\leq Ce^{-\omega t}$. The claim for $c$ follows with the
help of the estimates in Lemma~\ref{lem_ci_infty}, Theorem~\ref{thm_asymp} and relation
(\ref{eq_Psi_lu}) by
\[\|c_i-c_i^\infty\|_1\leq C\|(\sqrt{c_i}+\sqrt{c_i^\infty})(\sqrt{c_i}-\sqrt{c_i^\infty})\|_2\leq
C'\Psi(u,c)^{1/2}\leq C''e^{-\omega' t}.\]
\end{proof}

\section*{Acknowledgement} This work is supported by the Deutsche
Forschungsgemeinschaft within the cluster of excellence ``Center of 
Smart Interfaces'' at the TU Darmstadt. The authors would like to thank 
A.~Glitzky for useful hints concerning the boundary 
conditions (\ref{eq_P_bdy_}) imposed on
the electro-static potential.

\bibliography{Literatur}{}

\begin{thebibliography}{10}

\bibitem{Adams}
R.A. Adams and J.J.F. Fournier.
\newblock {\em Sobolev {S}paces}.
\newblock Academic press, Kidlington, Oxford, 2 edition, 2003.

\bibitem{amann}
H.~Amann.
\newblock {\em Linear and {Q}uasilinear {P}arabolic {P}roblems - {A}bstract
  {L}inear {T}heory}, volume~1 of {\em Monographs in mathematics}.
\newblock Birkh{\"a}user Verlag Basel, 1994.

\bibitem{amann_renardy}
H.~Amann and M.~Renardy.
\newblock Reaction-diffusion problems in electrolysis.
\newblock {\em Nonlinear Differential Equations and Applications}, pages
  91--117, 1994.

\bibitem{atkins}
P.W. Atkins and J.~de~{P}aula.
\newblock {\em Physikalische {C}hemie}.
\newblock Wiley, 4 edition, 2006.

\bibitem{bedthom2011}
L.~Bedin and M.~Thompson.
\newblock Existence theory for a {P}oisson-{N}ernst-{P}lanck model of
  electrophoresis.
\newblock {\em arxiv:1102.5370v1}, 2011.

\bibitem{BSL}
R.B. Bird, W.E. Stewart, and E.N. Lightfoot.
\newblock {\em Transport {P}henomena}.
\newblock Wiley, 2 edition, 2001.

\bibitem{BothePruss}
D.~Bothe and J.~Pr{\"u}ss.
\newblock Mass transport through charged membranes.
\newblock {\em Proc. 4th European Conf. on Elliptic and Parabolic Problems},
  pages 332--342, 2002.

\bibitem{chang}
S.T. Chang.
\newblock {\em New {E}lectrokinetic {T}echniques for {M}aterial {M}anipulation
  on the {M}icroscale}.
\newblock BiblioBazaar, 2011.

\bibitem{CL92}
Y.S. Choi and R.~Lui.
\newblock Analysis of an electrochemistry model with zero-flux boundary
  conditions.
\newblock {\em Appl. Anal.}, 49:277--288, 1993.

\bibitem{CL_multidim}
Y.S. Choi and R.~Lui.
\newblock Multi-dimensional electrochemistry model.
\newblock {\em Arch. Rational Mech. Anal.}, 130:315--342, 1995.

\bibitem{cussler}
E.L. Cussler.
\newblock {\em Diffusion: {M}ass {T}ransfer in {F}luid {S}ystems}.
\newblock Cambridge University Press, 2 edition, 1997.

\bibitem{DHP}
R.~Denk, M.~Hieber, and J.~Pr{\"u}ss.
\newblock Optimal {$L^p-L^q$}-regularity for parabolic problems with
  inhomogeneous boundary data.
\newblock {\em Math. Z.}, 257:193--224, 2007.

\bibitem{DF06}
L.~Desvillettes and K.~Fellner.
\newblock Exponetial decay toward equilibrium via entropy methods for
  reaction-diffusion equations.
\newblock {\em Journal of Mathematical Analysis and Applications},
  319:157--176, 2006.

\bibitem{Evans}
L.C. Evans.
\newblock {\em Partial differential equations}.
\newblock Providence, RI: American Mathematical Society, 1998.

\bibitem{gaj85}
H.~Gajewski.
\newblock On existence, uniqueness and asymptotic behaviour of solutions of the
  basic equations for carrier transport in semiconductors.
\newblock {\em Z. Angew. Math. u. Mech}, 65:101--108, 1985.

\bibitem{galdi}
G.P. Galdi.
\newblock {\em An {I}ntroduction to the {M}athematical {T}heory of the
  {N}avier-{S}tokes {E}quations, {V}ol. {I}, {L}inearised {S}teady {P}roblems}.
\newblock Springer Verlag, New York, 2 edition, 2011.

\bibitem{giga85}
Y.~Giga.
\newblock Domains of fractional powers of the {S}tokes operator in {$L_r$}
  spaces.
\newblock {\em Arch. Ration. Mech. Anal.}, 89:251--265, 1985.

\bibitem{GilTrud}
D.~Gilbarg and N.S. Trudinger.
\newblock {\em Elliptic {P}artial {D}ifferential {E}quations of {S}econd
  {O}rder}.
\newblock Springer, 2001.

\bibitem{GGH95}
A.~Glitzky, K.~Gr{\"o}ger, and R.~H{\"u}nlich.
\newblock Free energy and dissipation rate for reaction diffusion processes of
  electrically charged species.
\newblock {\em Appl. Anal.}, 60:201--217, 1995.

\bibitem{Glitzky}
A.~Glitzky and R.~H{\"u}nlich.
\newblock Energetic estimates and asymptotics for electro-reaction-diffusion
  systems.
\newblock {\em Z. Angew. Math. Mech.}, 77:823--832, 1997.

\bibitem{Glitzky_2}
A.~Glitzky and R.~H{\"u}nlich.
\newblock Global estimates and asymptotics for electro-reaction-diffusion
  systems in heterostructures.
\newblock {\em Appl. Anal.}, 66:205--226, 1997.

\bibitem{HHV}
C.H. Hamann, A.~Hamnett, and W.~Wielstich.
\newblock {\em Electrochemistry}.
\newblock Wiley-VCH, 1998.

\bibitem{jerome2002}
J.W. Jerome.
\newblock Analytical approaches to charge transport in a moving medium.
\newblock {\em Transport Theory and Statistical Physics}, 31(4-6):333--336,
  2002.

\bibitem{kirby}
B.~Kirby.
\newblock {\em Micro- and {N}anoscale {F}luid {M}echanics: {T}ransport in
  {M}icrofluidic {D}evices}.
\newblock Cambridge University Press, 2010.

\bibitem{koehnepruesswilke}
M.~K{\"o}hne, J.~Pr{\"u}ss, and M.~Wilke.
\newblock Qualitative behaviour of solutions for the two-phase
  {N}avier-{S}tokes equations with surface tension.
\newblock {\em Math. Ann.}
\newblock to appear.

\bibitem{TaylorKrishna}
R.~Krishna and R.~Taylor.
\newblock {\em Multicomponent mass transfer}.
\newblock Wiley (New York), 1993.

\bibitem{LSU}
O.A. Ladyzenskaja, V.A. Solonnikov, and N.N. Ural'seca.
\newblock {\em Linear and quasilinear equations of parabolic type}, volume~23
  of {\em Translation of {M}athematical {M}onographs}.
\newblock Amer. Math. Soc, 1968.

\bibitem{MasBha}
J.H. Masliyah and S.~Bhattacharjee.
\newblock {\em Electrokinetic and {C}olloid {T}ransport {P}henomena}.
\newblock Wiley-Interscience, 1993.

\bibitem{Meyries}
M.~Meyries and R.~Schnaubelt.
\newblock Interpolation, embeddings and traces of anisotropic fractional
  {S}obolev spaces with temporal weights.
\newblock {\em J. Funct. Anal.}, 262:1200--1229, 2012.

\bibitem{Li2011}
S.~Movahed and D.~Li.
\newblock Electrokinetic transport through nanochannels.
\newblock {\em Electrophoresis}, 32:1259--1267, 2011.

\bibitem{Newman}
J.S. Newman.
\newblock {\em Electrochemical systems}.
\newblock Prentice Hall, 2 edition, 1991.

\bibitem{noll2003}
A.~Noll and J.~Saal.
\newblock {$H^\infty$}-calculus for the {S}tokes operator on {$L_q$}-spaces.
\newblock {\em Math.Z.}, 244:651--688, 2003.

\bibitem{Probstein}
R.F. Probstein.
\newblock {\em Physicochemical {H}ydrodynamics}.
\newblock Butterworths, 1989.

\bibitem{PruessSaalSimonett}
J.~Pr{\"u}ss, J.~Saal, and G.~Simonett.
\newblock Existence of analytic solutions for the classical {S}tefan problem.
\newblock {\em Math. Ann.}, 338:703--755, 2007.

\bibitem{Schmuck}
M.~Schmuck.
\newblock Analysis of the {N}avier-{S}tokes-{N}ernst-{P}lanck-{P}oisson system.
\newblock {\em Mathematical Models and Methods in Applied Sciences},
  19:993--1014, 2009.

\bibitem{Sob75}
P.E. Sobolevskii.
\newblock Fractional powers of coercively positive sums of operators.
\newblock {\em Soviet Math. Dokl.}, 6:1638--1641, 1975.

\bibitem{Sohr}
H.~Sohr.
\newblock {\em The {N}avier-{S}tokes {E}quations: {A}n {E}lementary
  {F}unctional {A}nalytic {A}pproach}.
\newblock Birkh{\"a}user, Boston Basel Berlin, 2001.

\bibitem{sol77}
V.A. Solonnikov.
\newblock Estimates for solutions of nonstationary {N}avier-{S}tokes equations.
\newblock {\em J. Soviet Math}, 8:213--317, 1977.

\bibitem{TemamDyn}
R.~Temam.
\newblock {\em Infinite-{D}imensional {D}ynamical {S}ystems in {M}echanics and
  {P}hysics}, volume~68 of {\em Applied Mathematical Sciences}.
\newblock Springer New York, 2 edition, 1997.

\bibitem{triebel}
H.~Triebel.
\newblock {\em Interpolation {T}heory, {F}unction {S}paces, {D}ifferential
  {O}perators}.
\newblock North Holland, 1978.

\bibitem{troj}
M.~Trojanowicz.
\newblock {\em Advances in {F}low {A}nalysis}.
\newblock Wiley-VCH, 2008.

\bibitem{wiedmann}
J.~Wiedmann.
\newblock {\em An electrolysis model and its solutions}.
\newblock Phd-thesis, University of Zurich, 1997.

\bibitem{Wu}
Z.~Wu, J.~Yin, and C.~Wang.
\newblock {\em Elliptic and {P}arabolic {E}quations}.
\newblock World Scientific Publishing, 2006.

\end{thebibliography}
\bibliographystyle{plain}

\end{document}